\documentclass[a4paper,10pt]{article}
\usepackage[french, english]{babel}
\usepackage[utf8]{inputenc}
\usepackage[T1]{fontenc}
\usepackage{lmodern}
\usepackage{csquotes}
\usepackage[marginparwidth=2.5cm, marginparsep=3pt, top=2.7cm, left=3cm, right=3cm]{geometry}
\usepackage{amsmath,amssymb,amsthm}
\usepackage{mathtools}
\usepackage{tikz}
\usetikzlibrary{calc, positioning, intersections, patterns, arrows.meta,
decorations.markings}
\usepackage[titletoc,title]{appendix}
\usepackage{dsfont}
\usepackage{microtype}
\usepackage[font=footnotesize, margin=1em]{caption}
\usepackage{hyperref}
\usepackage{cleveref}

\definecolor{sienne}{RGB}{136, 45, 23}
\hypersetup{
colorlinks = true,
linkcolor = sienne
}
\allowdisplaybreaks[1]

\usepackage{enumitem}

\newtheorem{theorem}{Theorem}
\newtheorem{prop}[theorem]{Proposition}
\newtheorem{corollary}[theorem]{Corollary}
\newtheorem{lemma}[theorem]{Lemma}
\theoremstyle{definition}
\newtheorem{definition}[theorem]{Definition}
\theoremstyle{remark}
\newtheorem{remark}[theorem]{Remark}

\newcommand{\puncfootnote}[1]{\kern-0.2em\footnote{#1}}

\mathtoolsset{mathic}
\newcommand*{\noic}{\sb{}\kern-\scriptspace }

\DeclareMathOperator{\Sp}{Sp}
\DeclareMathOperator{\Ima}{Im}
\DeclareMathOperator{\Span}{Span}
\DeclareMathOperator{\Supp}{supp}

\DeclareMathOperator{\codim}{codim}

\newcommand{\norme}[1]{\left\lVert#1\right\rVert}

\newcommand{\bigO}{O}
\newcommand{\longmapsfrom}{\mathrel{\mathord{\leftarrow}\mkern-1mu\rule{1pt}{1.1ex}\mkern1mu}}

\newcommand{\diff}[1][-3]{\ensuremath{\mathop{}\mkern#1mu\mathrm{d}}}
\newcommand{\set}{\mathbb}
\newcommand{\N}{\set{N}}
\newcommand{\Z}{\set{Z}}

\newcommand{\R}{} 
\renewcommand{\R}{\set{R}}
\newcommand{\C}{\set{C}}

\newcommand{\T}{\set{T}}

\newcommand{\hol}{{\mathcal O}}

\newcommand{\Tr}[1]{{#1}^{\mathrm{tr}}}

\newcommand{\coloneqq}{\mathrel{\mathord{:}\mathord=}}
\newcommand{\eq}{\Leftrightarrow}
\newcommand{\longeq}{\Longleftrightarrow}

\newcommand{\h}{{\mathrm h}}
\newcommand{\p}{{\mathrm p}}
\newcommand{\m}{{\mathrm m}}
\newcommand{\eu}{\mathrm e}
\newcommand{\iu}{\mathrm i}

\renewcommand{\epsilon}{\varepsilon}

\let\originalleft\left
\let\originalright\right
\renewcommand{\left}{\mathopen{}\mathclose\bgroup\originalleft}
\renewcommand{\right}{\aftergroup\egroup\originalright}

\newcommand{\ssqrt}[1]{\sqrt{\smash[b]{#1}}}
\newcommand{\iun}{\left(\frac\iu n\right)}

\renewcommand{\paragraph}[1]{\vskip-\lastskip\medskip\noindent\textit{#1}}

\title{Null-controllability of linear parabolic-transport systems}
\author{Karine Beauchard\thanks{Univ Rennes, CNRS, IRMAR - UMR 6625, F-35000 Rennes, France. email: \texttt{karine.beauchard@ens-rennes.fr}}, Armand Koenig\thanks{Université Côte d'Azur, CNRS, LJAD, France. email:~\texttt{armand.koenig@unice.fr} }, Kévin Le Balc'h\thanks{Univ Rennes,  CNRS, IRMAR - UMR 6625, F-35000 Rennes, France. email: \texttt{kevin.lebalch@ens-rennes.fr}}}

\begin{document}
\maketitle
\begin{abstract}
Over the past two decades, the controllability of several examples of parabolic-hyperbolic systems has been investigated. The present article is the beginning of an attempt to find a unified framework that encompasses and generalizes the previous results. 

We consider constant coefficients heat-transport systems with coupling of order zero and one, with a locally distributed control in the source term, posed on the one dimensional torus.

We prove the null-controllability, in optimal time (the one expected because of the transport component) when there is as much controls as equations. 
When the control acts only on the transport (resp. parabolic) component, we prove an algebraic necessary and sufficient condition, on the coupling term, for the null controllability.

The whole study relies on a careful spectral analysis, based on perturbation theory. 
For high frequencies, the spectrum splits into a parabolic part and an hyperbolic part. 
The negative controllability result in small time is proved on solutions localized on high hyperbolic frequencies, that solve a pure transport equation up to a compact term.
The positive controllability result in large time is proved by projecting the dynamics onto 3 eigenspaces associated to hyperbolic, parabolic and low frequencies; that defines 3 weakly coupled systems. 
\end{abstract}


\setcounter{tocdepth}{1}
\tableofcontents
\section{Introduction}
\subsection{Parabolic-transport systems}
We consider the linear control system
\begin{equation}
\label{Syst}\tag{Sys}
\left\{
\begin{array}{l l}
\partial_t f-  B \partial_{x}^2 f + A \partial_x f + K f  = M u 1_{\omega} &\text{in}\ (0,T)\times \T,\\
f(0,\cdot)=f_0& \text{in}\ \T,
\end{array}
\right.
\end{equation}
where  
\begin{itemize}
\item $T >0$, $\T = \R/(2\pi \Z)$, $\omega$ is a nonempty open subset of $\T$, $d \in \N^{*}$, $m\in\{1,\dots,d\}$, $A, B, K \in \R^{d \times d}$, $M\in\R^{d\times m}$,
\item the state is $f:[0,T]\times\T \to \R^d$,
\item the control is $u:[0,T]\times\T \to \R^m$.
\end{itemize}
We assume 
\begin{gather}
d=d_1 + d_2\ \text{with}\ 1 \leq d_1 < d,\ 1 \leq d_2 < d,\label{dd1d2}\tag{H.1}\\
B = \begin{pmatrix}0&0\\0&D\end{pmatrix},\text{ with } D\in \set R^{d_2\times d_2},\label{h:B}\tag{H.2}\\
\Re(\Sp(D)) \subset (0,\infty). \label{h:D}\tag{H.3}
\end{gather}

Introducing the analogue block decomposition for the $d\times d$ matrices $A$ and $K$, the $d \times m$ matrix $M$ and the function $f$,
\[A = \begin{pmatrix}A'& A_{12}\\A_{21}&A_{22}\end{pmatrix}, \quad
K = \begin{pmatrix}K_{11}& K_{12}\\ K_{21}& K_{22}\end{pmatrix}, \quad
M= \begin{pmatrix}M_1 \\ M_2\end{pmatrix}, \quad
f(t,x)=\begin{pmatrix} f_1(t,x) \\ f_2(t,x) \end{pmatrix},\]
we see that the system~\eqref{Syst} couples a transport equation on $f_1$ with a parabolic equation on $f_2$
\begin{equation}
\label{Syst_bloc}
\mkern-3mu\left\{\!
\begin{array}{ll}
\left( \partial_t +A'\partial_x+K_{11}\right)f_1+\left(A_{12}\partial_x+K_{12}\right)f_2= M_1 u  1_\omega
&\text{in}\ (0,T)\times \T,\!\\
\left( \partial_t -D\partial_x^2+A_{22}\partial_x+K_{22}\right)f_2 +\left(A_{21}\partial_x+K_{21}\right)f_1 = M_2 u  1_\omega &\mathrm{in}\ (0,T)\times \T,\!\\
(f_1,f_2)(0,\cdot)=(f_{01},f_{02}) & \text{in}\ \T.
\end{array}
\right.
\end{equation}
We make the following hypothesis on the matrix $A'$
\begin{equation}
\tag{H.4}
A'\text{ is diagonalizable with }\Sp(A')\subset \set R.\label{h:A1}
\end{equation}
We will prove later, with vector valued Fourier series and a careful spectral analysis, that for every $f_0 \in L^2(\T,\C^d)$ and $u\in L^2((0,T)\times\T,\C^m)$, there exists a unique solution $f\in C^0([0,T],L^2(\T)^d)$ of \eqref{Syst} (see Section~\ref{sec:wp}).
In this article, we are interested in the null controllability of~\eqref{Syst}.
\begin{definition} \label{definition:NC}
The system \eqref{Syst} is null-controllable on $\omega$ in time $T$ if for every $f_0 \in L^2(\T;\C^d)$, there exists a control $u \in L^2((0,T)\times\T,\C^m)$ supported on $[0,T]\times\omega$ such that the solution $f$ of \eqref{Syst} satisfies $f(T,\cdot)=0$.
\end{definition}
We aim at
\begin{itemize}
\item identifying the minimal time for null controllability, 
\item controlling the system with a small number of controls $m<d$,
\item understanding the influence of the algebraic structure $(A,B,K,M)$ on the above properties. 
\end{itemize}

\subsection{Statement of the results}

\subsubsection{Control on any component, minimal time}

Our first result identifies the minimal time, when the control acts on each of the $d$ equations.
\begin{theorem}\label{th:main}
We assume that $\omega$ is a strict open subset of $\set T$. We also assume~\eqref{dd1d2}--\eqref{h:A1} and that the control matrix is $M=I_d$ (and so $m=d$).
We define\footnote{If $I\subset \set R$ is measurable, we note $|I|$ its Lebesgue measure.}
\begin{equation}\label{eq:def_l}\ell(\omega) \coloneqq \sup\{ |I|;\ I \text{ connected component of }  \T \setminus \omega \},\end{equation}
\[\mu_*=\min\{ |\mu|;\ \mu \in \Sp(A') \},\]
and
\begin{equation}\label{eq:T}
T^\ast = 
\left\lbrace \begin{array}{ll}
\frac{\ell(\omega)}{\mu_*}  & \text{ if } \mu_* >0, \\
+\infty & \text{ if } \mu_*=0.
\end{array}\right.
\end{equation}
Then
\begin{enumerate}
\item the system~\eqref{Syst} is not null-controllable on $\omega$ in time $T<T^\ast$,
\item the system~\eqref{Syst} is null-controllable on $\omega$ in any time $T>T^\ast$.
\end{enumerate}
\end{theorem}

In particular, when $\omega$ is an interval of $\T$ and $\mu_*>0$, then the minimal time for null controllability is $T^\ast = \frac{2 \pi - |\omega|}{\mu_*}$.

Actually, the controls may be more regular than in \Cref{definition:NC}: we construct controls of the form $u=(u_1,u_2)$ where $u_1 \in L^2((0,T)\times\omega)^{d_1}$ and $u_2 \in C^\infty_c((0,T)\times\omega)^{d_2}$.

The proof of Theorem \ref{th:main} relies on a spectral decomposition: for high frequencies, the spectrum splits into a parabolic part and a hyperbolic part. 

The negative result in time $T<T^\ast$ is expected, because of the transport component of the system, but its proof does require some care.
Indeed, because of the coupling with a parabolic component, in general, there does not exist pure transport solutions to the system~\eqref{Syst}, concentrated outside $(0,T)\times\omega$ (see Appendix~\ref{app:transport} for more precision).


The proof of the positive result, in time $T>T^\ast$ relies on an adaptation, to systems with arbitrary size, of the strategy introduced by Lebeau and Zuazua~\cite{lebeau_1998} to control the system of linear thermoelasticity, that couples a scalar heat equation and a scalar wave equation. By projecting the dynamics onto appropriate eigenspaces, the system is decomposed into 3 weakly coupled systems. The first one behaves like a transport system, its controllability is handled by hyperbolic methods from~\cite{alabau-boussouira_2017}. The second one behaves like a parabolic system, its controllability is handled by the Lebeau-Robbiano method. The third one, associated to low frequencies, has finite dimension; its controllability is handled by a compactness/uniqueness argument.

The null controllability of the system~\eqref{Syst} in time $T=T^\ast$ is an open problem.

\subsubsection{Control on the hyperbolic component}

Our second result concerns controls acting on the whole transport component, $M_1=I_{d_1}$, but not on the parabolic component, $M_2=0$.
To get an aesthetic necessary and sufficient algebraic condition for null controllability, we also assume that the diffusion is given by $D=I_{d_2}$, the coupling is realized exclusively by the transport term $A_{21}\partial_x f_1$, i.e.\ $K_{21}=0$ and there is no zero order term in the parabolic dynamics, i.e.\ $K_{22}=0$, which corresponds to the system
\begin{equation} \label{Syst_bloc_hyp_A}
\mkern-3mu\left\{\!
\begin{array}{ll}
\left( \partial_t +A'\partial_x+K_{11}\right)f_1+\left(A_{12}\partial_x+K_{12}\right)f_2= u_1  1_\omega
&\text{in}\ (0,T)\times \T,\!\\
\left( \partial_t - \partial_x^2+A_{22}\partial_x\right)f_2 + A_{21}\partial_x f_1 = 0 &\mathrm{in}\ (0,T)\times \T.
\end{array} \right.
\end{equation}
By integrating with respect to the space variable the second equation of (\ref{Syst_bloc_hyp_A}), we see that, for being steered to zero, an initial condition $f_0=(f_{01},f_{02})\in L^2(\T)^{d_1} \times L^2(\T)^{d_2}$ has to satisfy
\begin{equation} \label{Hyp_A_contrainte}
\int_{\T} f_{02}(x) \diff x =0.
\end{equation}
For any vector subspace $E$ of $L^1(\T)$ we denote by $E_\m$ the vector subspace made of functions $f\in E$ with zero mean value, i.e.\ $\int_{\T} f(x)dx =0$.

\begin{theorem}\label{th:main_2}
We assume~\eqref{dd1d2}--\eqref{h:A1}, $D=I_{d_2}$ $m=d_1$, $M_1=I_{d_1}$, $M_2=0$, $K_{21}=0$ and $K_{22}=0$. Let $T^\ast$ be defined by~\eqref{eq:T}.
The following statements are equivalent:
\begin{itemize}
\item For every $T>T^{\ast}$ and $f_0=(f_{01},f_{02})\in L^2(\T)^{d_1} \times L^2_\m(\T)^{d_2}$, there exists $u_1 \in L^2((0,T)\times\omega)^{d_1}$ such that the solution of (\ref{Syst_bloc_hyp_A}) satisfies $f(T)=0$.
\item The couple of matrices $(A_{22},A_{21})$ satisfies the Kalman rank condition:
\begin{equation} \label{Kalman_Thm2}
\Span \{ A_{22}^j A_{21} X_1 ; X_1 \in \C^{d_1}, 0 \leq j \leq d_2-1 \} = \C^{d_2}.
\end{equation}
\end{itemize}
\end{theorem}

With the same proof, similar statements can be proved for the following systems:
\begin{equation} \label{Syst_bloc_hyp_K_0}
\mkern-3mu\left\{\!
\begin{array}{ll}
\left( \partial_t +A'\partial_x+K_{11}\right)f_1+\left(A_{12}\partial_x+K_{12}\right)f_2= u_1  1_\omega
&\text{in}\ (0,T)\times \T,\!\\
\left( \partial_t -\partial_x^2+K_{22}\right)f_2 + K_{21} f_1 = 0 &\text{in}\ (0,T)\times \T,\!\\
\end{array} \right.
\end{equation}
with arbitrary initial conditions $f_0 \in L^2(\T)^d$ and  Kalman rank condition on $(K_{22},K_{21})$ (see Section \ref{sec:cont_parab_simple}),
\begin{equation}
\mkern-3mu\left\{\!
\begin{array}{ll}
\left( \partial_t +A'\partial_x+K_{11}\right)f_1+\left(A_{12}\partial_x+K_{12}\right)f_2= u_1  1_\omega
&\text{in}\ (0,T)\times \T,\!\\
\left( \partial_t -\partial_x^2+A_{22} \partial_x \right)f_2 + K_{21} f_1 = 0 &\mathrm{in}\ (0,T)\times \T,\!\\ 
\end{array} \right.
\end{equation}
with arbitrary initial conditions $f_0 \in L^2(\T)^d$ and  Kalman rank condition on $(A_{22},K_{21})$,
\begin{equation}
\mkern-3mu\left\{\!
\begin{array}{ll}
\left( \partial_t +A'\partial_x+K_{11}\right)f_1+\left(A_{12}\partial_x+K_{12}\right)f_2= u_1  1_\omega
&\text{in}\ (0,T)\times \T,\!\\
\left( \partial_t -\partial_x^2+K_{22} \right)f_2 + A_{21} \partial_x f_1 = 0 &\mathrm{in}\ (0,T)\times \T,
\end{array} \right.
\end{equation}
with initial conditions $f_0 \in L^2(\T)^d$ satisfying (\ref{Hyp_A_contrainte}) and Kalman rank condition on $(K_{22},A_{21})$.

The proof of the controllability of (\ref{Syst_bloc_hyp_A}) uses 2 ingredients. The first ingredient is a strengthened version of Theorem \ref{th:main} with smoother controls, more precisely, the associated observability inequality with observation of negative Sobolev norms of the parabolic component. The second ingredient is a cascade structure (or Brunovski form) of the system (\ref{Syst_bloc_hyp_A}) ensured by the Kalman condition, to eliminate the observation of the parabolic component.

Proving an algebraic necessary and sufficient condition for null controllability of (\ref{Syst}), involving both matrices $D$, $A$ and $K$ is an open problem.
In the context of parabolic systems, this difficulty already appeared, see \cite{ammarkhodja_2009a} and \cite{duprez_2016a}.

\subsubsection{Control on the parabolic component}

Our third result concerns controls acting on the whole parabolic component, $M_2=I_{d_2}$, but not on the hyperbolic component of the system, $M_1=0$. To get an aesthetic necessary and sufficient condition for null controllability, we also assume that the coupling is realized exclusively by the transport term $A_{12}\partial_x f_2$, i.e.\ $K_{12}=0$, and there is no zero order term in the hyperbolic dynamics, i.e.\ $K_{11}=0$. This corresponds to the system
\begin{equation} \label{Syst_bloc_parab_A}
\mkern-3mu\left\{\!
\begin{array}{ll}
\left( \partial_t +A'\partial_x \right)f_1+ A_{12}\partial_x f_2= 0
&\text{in}\ (0,T)\times \T,\\
\left( \partial_t - D \partial_x^2+A_{22}\partial_x+K_{22}\right)f_2 + (A_{21}\partial_x + K_{21}) f_1 = u_2 1_\omega &\mathrm{in}\ (0,T)\times \T,\\
(f_1,f_2)(0,\cdot)=(f_{01},f_{02}) & \text{in}\ \T.
\end{array} \right.
\end{equation}
By integrating with respect to the space variable the first equation of (\ref{Syst_bloc_parab_A}), we see that, for being steered to zero, an initial condition $f_0=(f_{01},f_{02})\in L^2(\T)^{d_1} \times L^2(\T)^{d_2}$ has to satisfy
\begin{equation} \label{Hyp_A_contrainte_parab}
\int_{\T} f_{01}(x) \diff x =0 
\end{equation}
i.e.\ $f_0=(f_{01},f_{02})\in L^2_\m(\T)^{d_1} \times L^2(\T)^{d_2}$.

We need to adapt the notion of null controllability, because null controllable initial conditions necessarily have a regular hyperbolic component. Indeed, in (\ref{Syst_bloc_parab_A}), the source term 
$A_{12}\partial_x f_2$ entering the hyperbolic equation on $f_1$ --- that has to serve as an indirect control for $f_1$ ---  is smooth, because of the parabolic smoothing on $f_2$. Such a smooth source term cannot steer to zero non-smooth initial conditions.

\begin{theorem}\label{th:main_3}
Let $\omega$ be an open interval of $\T$. We assume~\eqref{dd1d2}--\eqref{h:A1}, $m=d_2$, $M_1=0$, $M_2=I_{d_2}$, $K_{11}=0$ and $K_{12}=0$. Let $T^\ast$ be defined by~\eqref{eq:T}.
The following statements are equivalent.
\begin{itemize}
\item For every $T>T^\ast$ and $f_0=(f_{01},f_{02}) \in H^{d_1+1}_\m(\T)^{d_1} \times H^{d_1+1}(\T)^{d_2}$ there exists $u_2 \in L^2((0,T)\times\omega)^{d_2}$ such that the solution of~\eqref{Syst_bloc_parab_A} satisfies $f(T)=0$.
\item The couple of matrices $(A',A_{12})$ satisfies the Kalman rank condition:
\begin{equation} \label{Kalman_Thm3}
\Span \{ (A')^j A_{12} X_2 ; X_2 \in \C^{d_2}, 0 \leq j \leq d_2-1 \} = \C^{d_2}.
\end{equation}
\end{itemize}
\end{theorem}

In \Cref{th:main_3}, we assume that the open set of control $\omega$ is an interval because the proof uses \cite[Lemma 2.6]{alabau-boussouira_2017} (see \Cref{LemmaCutOffTransport} below). The generalization of this result to a general open set $\omega$ of $\T$ is not known.

A similar statement can be obtained with the same proof, when
$K_{11}=0$, $A_{12}=0$ under Kalman rank condition on $(A',K_{12})$.

The proof of \Cref{th:main_3} follows essentially the same strategy as the one of \Cref{th:main_2}: a strengthened version of \Cref{th:main} and a cascade structure ensured by Kalman condition. The regularity assumption on the hyperbolic component allows the elimination of the observation of the hyperbolic component.

After \Cref{th:main_3}, two problems are still open:
\begin{itemize}
\item the characterization of null controllable initial conditions: it may be a larger space than $H^{d_1+1}_\m(\T)^{d_1} \times H^{d_1+1}(\T)^{d_2}$, see Section \ref{Sec:Par},
\item the algebraic necessary and sufficient condition for null controllability, involving both matrices $A$ and $K$. In the context of parabolic systems, this difficulty already appeared, see \cite{ammarkhodja_2009a} and~\cite{duprez_2016a}.
\end{itemize}

\subsection{Organization of the article}

Section \ref{Sec:Prel} is dedicated to preliminary results concerning the spectral analysis of $-B\partial_x^2+A\partial_x+K$ on $\T$, the well posedness of (\ref{Syst}) and the Hilbert uniqueness method.

In Section \ref{Sec:Obst}, we prove the negative null controllability result in time $T<T^\ast$ of Theorem \ref{th:main}. 

In Section \ref{sec:NC}, we prove the positive null controllability result in time $T>T^\ast$ of Theorem \ref{th:main}. 

In Section \ref{sec:cont_parab_simple}, we explain how to adapt this proof to get the null controllability in time $T>T^\ast$ of system (\ref{Syst_bloc_hyp_K_0}). The interest of this section is to introduce the proof strategy of  \Cref{th:main_2} and \Cref{th:main_3}, in a less technical framework. 

Then, in Section \ref{Sec:Hyp}, we prove Theorem \ref{th:main_2} and in Section \ref{Sec:Par}, we prove Theorem \ref{th:main_3}.

\subsection{Bibliographical comments}

\subsubsection{Wave equation with structural damping}
We consider the 1D wave equation with structural damping and control $h$
\begin{equation} \label{wave}
\partial_t^2 y - \partial_x^2 y -\partial_t\partial_x^2y + b \partial_t y = h(t,x),
\end{equation}
where $b\in\R$. This equation can be splitted in a system of the form~\eqref{Syst} by considering $z\coloneqq \partial_t y - \partial_x^2 y+(b-1)y$,
\begin{equation} \label{wave_syst}
\left\lbrace\begin{array}{l}
\partial_t z+z+(1-b)y= h(t,x), \\
\partial_t y-\partial_x^2 y  -z + (b-1)y =0,
\end{array}\right.
\end{equation}
i.e.\ \eqref{Syst} with $d=2$, $d_1=d_2=1$, $m=1$,
\begin{equation} \label{wave_matrix}
f=\begin{pmatrix}
z \\ y
\end{pmatrix}, \quad
B=\begin{pmatrix}
0 & 0 \\ 0 & 1
\end{pmatrix}, \quad 
A=\begin{pmatrix}
0 & 0 \\ 0 & 0
\end{pmatrix}, \quad 
K=\begin{pmatrix}
1 & 1-b \\ -1 & b-1
\end{pmatrix}, \quad 
M=\begin{pmatrix}
1  \\ 0
\end{pmatrix}.
\end{equation}

Rosier and Rouchon~\cite{rosier_2007} studied the equation~\eqref{wave} on a 1D-interval, $x\in (0,1)$, with a boundary control at $x=1$ and $h=0$. This is essentially equivalent to take~\eqref{wave} with $x\in(0,1)$, Dirichlet boundary conditions at $x=0$ and $x=1$, and a source term of the form $h(t,x)=u(t)p(x)$, where $p$ is a fixed profile and $u$ is a scalar control. The authors prove that this equation is not controllable.

By Theorem~\ref{th:main}, we extend this negative result to general controls $h$ (i.e.\ without separate variables) for periodic boundary conditions.
Here, $A'=0$, $\mu_*=0$, $T^\ast=+\infty$, the system \eqref{wave_syst} is not controllable even with an additional control in the second equation.

In~\cite{rosier_2007}, the authors prove that this system is not even spectrally controllable, because of an accumulation point in the spectrum.
Indeed, by the moment method, a control that would steer the system from an eigenstate to another one would have a Fourier transform vanishing on a set with an accumulation point, which is not possible for an holomorphic function.

Martin, Rosier and Rouchon~\cite{martin_2013}, studied the null-controllability of the equation~\eqref{wave} on the 1D torus, $x\in\T$, with moving controls, i.e.\ $h(t,x)=u(t,x) 1_{\omega+ct}$ with $c\in\R^*$.
By the change of variable $x \longmapsfrom (x-ct)$, this is equivalent to study the null controllability of the system
\begin{equation} \label{wave_syst2}
\left\lbrace\begin{array}{l}
\partial_t z-c\partial_x z+z+(1-b)y= u(t,x) 1_{\omega}(x), \\
\partial_t y-c\partial_x y-\partial_x^2y  -z + (b-1)y =0
\end{array}\right.
\end{equation}
which has the form \eqref{Syst} with the same matrices $f$, $B$, $K$ as in~\eqref{wave_matrix} and 
\[A=\begin{pmatrix}
-c & 0 \\ 0 & -c
\end{pmatrix}.\]
In~\cite[Theorem 1.2]{martin_2013}, for $c=1$, the authors prove that any initial data $(y_0,y_1)\in H^{s+2}\times H^{s}(\T)$ with $s>15/2$ can be steered to $0$ in time $T>2\pi$ by mean of a control $u\in L^2((0,T)\times\omega)$.

By Theorem~\ref{th:main_2}, we recover this positive null controllability result with a smaller minimal time $T>\ell(\omega)/|c|$ and a weaker regularity assumption on the initial data $(y,\partial_t y)(0)=(y_0,y_1)\in H^2 \times L^2(\T)$ for (\ref{wave}). 
This corresponds to an initial data $(y,z)(0) \in L^2(\T)^2$ for (\ref{wave_syst2}) because $z(0)=y_1-\partial_x^2 y_0 + (b-1)y_0$. Actually, Theorem~\ref{th:main_2} can be applied for $b=1$ in \eqref{wave_syst2} but an easy adaptation of Theorem~\ref{th:main_2} gives the same result for every $b \in \R$.
We also prove the negative result in time $T<\ell(\omega)/|c|$. Here, $\mu_*=|c|$, $A_{21}=0$ and $K_{21}=-1$.

The limitations in~\cite[Theorem 1.2]{martin_2013} (regularity and time) are due to the use of controls with separate variables $u(t,x)=u_1(t)u_2(x)$. The proof relies on the moment method and the construction of a biorthogonal family. A key point in both~\cite{martin_2013} and the present article is a splitting of the spectrum in one parabolic-type part, and one hyperbolic-type part.

Finally, Chaves-Silva, Rosier and Zuazua~\cite{chaves-silva_2014} study the multi-dimensional case of equation~\eqref{wave}, $x\in \Omega$, with Dirichlet boundary conditions and locally distributed moving controls $h(t,x)=u(t,x)1_{\omega(t)}(x)$.
The control region $\omega(t)$ is assumed to be driven by the flow of an ODE that covers all the domain $\Omega$ within the alloted time $T$.
Then, the authors prove the null controllability of any initial data $(y_0,y_1)\in H^2\cap H^1_0(\Omega) \times L^2(\Omega)$ with a $L^2$-control. 

In the particular case $\Omega=\T$ with a motion with constant velocity, Theorem~\ref{th:main_2} gives the same minimal time for the null controllability and also the negative result in smaller time.

The proof strategy in~\cite{chaves-silva_2014} consists in proving Carleman estimates for the parabol\-ic equation and the ODE in~\eqref{wave_syst} with \emph{the same singular weight}, adapted to the geometry of the moving support of the control.

As explained in~\cite[Section 5.2]{chaves-silva_2014}, the same construction cannot be used with periodic boundary conditions. 

In the very recent preprint \cite{guzman_2019}, the authors propose another construction of a weight, to get Carleman estimates for parabolic and transport equations in the torus $\T^2$ (with the same weight). In the present article, we develop another strategy.

\subsubsection{Wave-parabolic systems}
Albano and Tataru~\cite{albano_2000} consider $2\times2$ parabolic-wave systems with boundary control, where
\begin{itemize}
\item the coupling term in the wave equation is given by a second order operator with respect to $x$,
\item the coupling term in the parabolic equation is given by a first order operator with respect to $(t,x)$.
\end{itemize}
This large class contains the linear system of thermoelasticity
\begin{equation} \label{syst:tataru}
\left\lbrace\begin{array}{ll}
\partial_t^2 w - \Delta w + \alpha \Delta \theta = 0, \quad & (t,x)\in(0,T)\times\Omega, \\
\partial_t \theta-\nu \Delta \theta + \beta \partial_t w = 0, \quad & (t,x)\in(0,T)\times\Omega, \\
w(t,x)=u_1(t,x), \quad & (t,x)\in(0,T)\times\partial \Omega, \\
\theta(t,x)=u_2(t,x), \quad &(t,x)\in(0,T)\times\partial \Omega, \\
\end{array}\right.
\end{equation}
where $\alpha, \beta,\nu>0$.

The authors of~\cite{albano_2000} prove the null controllability in large time of these systems, precisely in any time $T>2\sup\{|x|;x\in\Omega\}$ for the system~\eqref{syst:tataru}. The proof relies on Carleman estimates for the heat and the wave equation with the same singular weight. This strategy inspired Chaves-Silva, Rosier and Zuazua~\cite{chaves-silva_2014}.

Lebeau and Zuazua~\cite{lebeau_1998} prove the null-controllability of the linear system of thermoelasticity~\eqref{syst:tataru} with a locally distributed control in the source term of the wave equation, under the geometric control condition on $(\Omega,\omega,T)$.
The method is based on a spectral decomposition. For high frequencies, the spectrum splits into a parabolic part and a hyperbolic part. Projecting the dynamics onto the parabolic/hyperbolic subspaces, the system is decomposed into 2 weakly coupled systems, the first one behaving like a wave equation, the second one like a heat equation. The wave equation is handled by using the microlocal techniques developped for the wave equation~\cite{bardos_1992}. The parabolic equation is treated by using Lebeau and Robbiano's method~\cite{lebeau_1995}. The low frequency part is treated by a compactness argument relying on a unique continuation property.

The proof of the positive controllability results in the present article is an adaptation, to coupled transport-parabolic systems \emph{of any size}, of this approach, introduced for a $2\times2$ wave-parabolic system. The transport equation is handled by using the results from Alabau-Boussouira, Coron and Olive~\cite{alabau-boussouira_2017}.

The framework of systems (\ref{Syst}) does not cover the system (\ref{syst:tataru}) because the order of the coupling terms is too high.

\subsubsection{Heat equation with memory}
Ivanov and Pandolfi~\cite{ivanov_2009} and after them Guerrero and Imanuvilov~\cite{guerrero_2013} consider the heat equation with memory
\begin{equation} \label{chaleur_mem}
\left\lbrace\begin{array}{ll}
\partial_t y - \Delta y - \int_0^t \Delta y (\tau) \diff \tau = u 1_\omega, \quad & (t,x)\in(0,T)\times\Omega, \\
y(t,x)=0, & (t,x)\in(0,T)\times\partial\Omega.
\end{array}\right.
\end{equation}
In 1D, this equation can be splitted into a system of the form (\ref{Syst}) by considering $v(t,x)=-\int_0^t y_x(\tau) \diff \tau$:
\begin{equation} \label{chaleur_mem_syst}
\left\lbrace\begin{array}{ll}
\partial_t v +\partial_x y = 0, \\
\partial_t y - \partial_x^2 y + v_x = h 1_{\omega}, \\
y(t,0)=y(t,1)=v(t,0)=0, \\
\end{array}\right.
\end{equation}
i.e.
\[f=\begin{pmatrix}
v \\ y
\end{pmatrix}, \quad
B=\begin{pmatrix}
0 & 0 \\ 0 & 1
\end{pmatrix}, \quad 
A=\begin{pmatrix}
0 & 1 \\ 1 & 0
\end{pmatrix}, \quad 
K=\begin{pmatrix}
0 & 0 \\ 0 & 0
\end{pmatrix}.\]

In~\cite{ivanov_2009}, the authors prove that the heat equation with memory term is not ``null-controllable to the rest''.
In~\cite{guerrero_2013}, the authors prove that the scalar equation (\ref{chaleur_mem}) is not null controllable (whatever $T>0$).
Thus the system (\ref{chaleur_mem_syst}) is not null controllable.

Theorem \ref{th:main} proves that, when Dirichlet boundary conditions are replaced by periodic boundary conditions, then system (\ref{chaleur_mem_syst}) is not null controllable, even with an additional control in the first equation.

\subsubsection{1D-Linearized compressible Navier-Stokes equations}

The compressible Navier-Stokes equation on the 1D torus writes
\[\left\lbrace\begin{array}{ll}
\partial_t \rho + \partial_x (\rho  v ) = u_1(t,x) 1_\omega \quad & \text{ in } (0,T)\times\T \\
\rho[\partial_t v + v \partial_x v ] + \partial_x( a \rho^\gamma)- \mu \partial_x^2v  = u_2(t,x)1_\omega(x) & \text{ in } (0,T)\times\T
\end{array}\right.\]
where $a,\gamma, \mu>0$, $\rho, v$ are the density and velocity of the fluid. The state is $(\rho,v)$ and the control is $(u_1,u_2)$. We consider a constant stationary state $(\overline{\rho},\overline{v}) \in \R^*_+ \times \R^*$. The linearized system around the trajectory $((\rho,v)=(\overline{\rho},\overline{v}),(u_1,u_2)=(0,0))$ is
\begin{equation} \label{NSCL}
\left\lbrace\begin{array}{ll}
\partial_t \rho + \overline{v} \partial_x \rho + \overline{\rho}\partial_x  v  = u_1(t,x) 1_\omega, \ & \text{ in } (0,T)\times\T 
\\
\partial_t v + \overline{v} \partial_x v  +  a \overline{\rho}^{\gamma-2} \partial_x \rho - \frac{\mu}{\overline{\rho}} \partial_x^2 v  = u_2(t,x)1_\omega(x)\,, & \text{ in } (0,T)\times\T\,.
\end{array}\right.
\end{equation}
This system is in the form (\ref{Syst}) with 
\[f=\begin{pmatrix}
\rho \\ v
\end{pmatrix}, \quad
B=\begin{pmatrix}
0 & 0 \\ 0 & \frac{\mu}{\overline{\rho}}
\end{pmatrix}, \quad 
A=\begin{pmatrix}
\overline{v} & \overline{\rho} \\ a \overline{\rho}^{\gamma-2} & \overline{v}
\end{pmatrix}, \quad 
K=\begin{pmatrix}
0 & 0 \\ 0 & 0
\end{pmatrix}\]
and satisfies~\eqref{dd1d2}--\eqref{h:A1}.

By Theorem \ref{th:main}, the system (\ref{NSCL}) with two controls $(u_1,u_2)$ is null controllable in time $T>\frac{\ell(\omega)}{|\overline{v}|}$ and is not null controllable in time $T<\frac{\ell(\omega)}{|\overline{v}|}$.

By Theorem \ref{th:main_2}, the system (\ref{NSCL}) with one control $u_1$ in the first line (i.e.\ $u_2=0$) is null controllable in time $T>\frac{\ell(\omega)}{|\overline{v}|}$.

By Theorem \ref{th:main_3}, with one control $u_2$ in the second equation (i.e.\ $u_1=0$), any initial condition $(\rho_0,v_0)\in H^2_\m(\T)\times H^2(\T)$ can be steered to zero in time $T>\frac{\ell(\omega)}{|\overline{v}|}$ by a control $u_2 \in L^2((0,T)\times\omega)$.

In \cite{ervedoza_2012}, Ervedoza, Glass, Guerrero and Puel consider the (nonlinear) compressible Navier-Stokes equations on a bounded interval $x \in (0,L)$, without source term (i.e.\  $u_1=u_2=0$), but with a boundary control on both $\rho$ and $v$ at the two boundaries $x=0$ and $x=L$. They prove the local controllability of this nonlinear system, around the trajectory $(\rho,v)=(\overline{\rho},\overline{v})$, in appropriate functional spaces. A key ingredient is the controllability of the linearized system, which is proved to hold in time $T>\frac{L}{|\overline{v}|}$. Theorem \ref{th:main} of the present article allows to recover the same result with interior control instead of boundary control, and also proves the negative result in small time.

In \cite[Theorem 1.4]{chowdhury_2014},  Chowdhury, Mitra, Ramaswamy and Renardy prove the null controllability of (\ref{NSCL}) with two controls $(u_1,u_2)$ in time $T>\frac{2\pi}{|\overline{v}|}$, with spectral methods. Thus, Theorem \ref{th:main} of the present article allows to recover the same result but with a better minimal time $\smash{\frac{\ell(\omega)}{|\overline{v}|}}$ and also proves the negative result in time $T<\frac{\ell(\omega)}{|\overline{v}|}$.

In \cite[Theorem 1.3]{chowdhury_2015}, Chowdhury and Mitra prove with moment methods that any initial condition $(\rho_0,v_0)\in H^{s+1}_\m(\T)\times H^{s}(\mathbb{T})$ with $s> 6,5$ can be steered to zero in time $T>\frac{2\pi}{\overline{v}}$ by a control acting only on the second equation $u_2 \in L^2((0,T)\times\omega)$ (i.e.\  $u_1=0$).
In \cite[Theorem 1.2]{chowdhury_2014}, Chowdhury, Mitra, Ramaswamy and Renardy prove the same result for any initial conditions $(\rho_0,v_0)\in H^1_\m(\T) \times L^2(\T)$. Thus, Theorem \ref{th:main_3} of the present article provides a smaller minimal time $\frac{\ell(\omega)}{|\overline{v}|}$, for smoother initial conditions $(\rho_0,v_0)\in H^2_\m(\T)\times H^2(\T)$. It also proves the negative result in time $T<\frac{\ell(\omega)}{|\overline{v}|}$.

\section{Preliminary results} \label{Sec:Prel}

We want to understand the operator 
\begin{equation} \label{defL}
\mathcal{L} \coloneqq -B\partial_{x}^2 + A \partial_{x} + K
\end{equation}
with domain
\begin{equation} \label{def:D(L)}
D(\mathcal{L})=\left\{ f \in L^2(\T)^d ; -B\partial_{x}^2f + A \partial_{x}f + K f \in L^2(\T)^d \right\}
\end{equation}
where the derivatives are considered in the distributional sense $\mathcal{D}'(\T)$.
Throughout the article, we will note $e_n$ the function $x\mapsto \eu^{\iu nx}$.
We remark that applying $\mathcal L$ to $X e_n$, where $X \in \C^{d}$, we get
\begin{equation}
\label{FourierOp}
\mathcal L (X e_n) = n^{2}\left( B +\frac{\iu}{n} A + \frac{1}{n^{2}}K\right)X e_n.
\end{equation}
Thus, if we define $E(z)$ the following perturbation of $B$
\begin{equation}
\label{defBh}
\forall z \in \C,\ E(z) = B + z A - z^{2} K,
\end{equation}
then $\mathcal L$ acts on the Fourier side as a multiplication by $n^2E(\iu /n)$.

In \Cref{subsubsec:PT}, we apply the perturbation theory to the matrices $E(z)$ near $z=0$: the spectrum of $E(z)$ splits into 2 parts: one close to zero that defines the hyperbolic component, one close to the spectrum of $D$ that  defines the parabolic component. In \Cref{subsubsec:EF_WP}, we deduce the dissipation of the parabolic component and the boundedness of the hyperbolic component. Thanks to these estimates, we prove the well-posedness of System~\eqref{Syst}. Finally, in \Cref{Paragraphdual}, we recall the Hilbert Uniqueness Method.

\subsection{Perturbation theory}
\label{subsubsec:PT}

If we want to understand the semigroup $\eu^{t\mathcal L}$, we need to know the spectrum and the eigenvectors of $E(z)$. Here, we relate the spectral properties of $E(z)$ to those of $A$ and $B$, in the limit $z\to 0$. This is instrumental in all the article. Our proofs are essentially self-contained, but the reader unfamiliar with the analytic perturbation theory in finite dimension may read~\cite[Ch.~II \S1 and \S2]{kato_1995}. 

For $r>0$ and $m \in \N^{*}$, we define $\hol_{r}^{m\times m}$ as the set of holomorphic functions in the complex disk $D(0,r)$ with values in $\C^{m \times m}$. Our first result is the following one.

\begin{prop}\label{th:perturb_hyper}
There exist $r>0$ and a matrix-valued holomorphic function $P^\h\in \hol_{r}^{m\times  m}$ such that
\begin{enumerate}
\item $P^\h(0) = \big(\begin{smallmatrix}I_{d_1}&0\\0&0 \end{smallmatrix} \big)$,
\item for all $|z|<r$, $P^\h(z)$ is a projection that commutes with $E(z)$,
\item in the limit $z\to 0$, $E(z)P^\h(z) = \bigO(z)$.
\end{enumerate}
\end{prop}

\begin{proof}
The spectrum of $E(z)$ is continuous in $z$ (see~\cite[Ch.~II \S1.1]{kato_1995}). Let us consider the ``$0$-group'' of eigenvalues, i.e.\ the set of eigenvalues that tend to $0$ as $z\to 0$. Then we note $P^\h(z)$ the sum of the projections onto the eigenspace\footnote{We stress that when we talk about ``eigenspace'', we mean ``generalized eigenspace'' (or, in the terminology of Kato, algebraic eigenspace), i.e.\ the space of generalized eigenvectors.} of $E(z)$ associated with eigenvalues in the $0$-group along the other  eigenspaces. Another way to define $P^\h(z)$ is to choose $R=\frac12 \min_{\lambda\in\Sp(D)} |\lambda|$ and $r$ small enough so that for $|z|<r$, there is no eigenvalues of $E(z)$ on the circle $\partial D(0,R)$. Then, we define (see~\cite[Ch.~II, Eq.~(1.16)]{kato_1995})
\begin{equation}\label{eq:proj_0}
P^\h(z) = -\frac1{2\iu\pi}\int_{\partial D(0,R)} (E(z)-\zeta I_d)^{-1}\diff \zeta.
\end{equation}

In the terminology of Kato, $P^\h(z)$ is the ``total projection for the $0$-group''. Then, according to~\cite[Ch.~II \S1.4]{kato_1995}, $P^\h(z)$ is the projection onto the sum of eigenspaces associated to eigenvalues of $E(z)$ lying inside $D(0,R)$ along the other eigenspaces. It is holomorphic in $|z|<r$. For $z=0$, the formula~\eqref{eq:proj_0} that defines $P^\h(0)$ becomes
\[
P^\h(0) = -\frac1{2\iu\pi}\int_{\partial D(0,R)} (B-\zeta I_d )^{-1}\diff \zeta.
\]
Then, $P^\h(0)$ is the projection onto the eigenspace of $B$ associated to the eigenvalue $0$ along the other eigenspaces (see~\cite[Ch.~II \S1.4]{kato_1995}). So, according to the hypotheses~(\ref{h:B}--\ref{h:D}) on the blocks of $B$, $P^\h(0) =  \big(\begin{smallmatrix}I_{d_1}&0\\0&0 \end{smallmatrix} \big)$. This proves \textit{i)}.

According to the definition~\eqref{eq:proj_0}, $P^\h(z)$ commutes with $E(z)$. This proves \textit{ii)}. Then we have 
\[P^\h(0)E(0) = E(0)P^\h(0) = BP^\h(0) = 0,\]
which, along with the holomorphy of $P^\h$, proves \textit{iii)}.
\end{proof}

We say that $P^\h$ is the ``projection on the hyperbolic branches''. We note $P^\p(z) = I_d - P^\h(z)$, which we call the ``projection on the parabolic branches'', and satisfies properties analog to $P^\h$:
\begin{prop}\label{th:perturb_parab}
The matrix-valued function $P^\p$ is in $\hol_{r}^{m\times  m}$ and
\begin{enumerate}
\item $P^\p(0) = \big(\begin{smallmatrix}0&0\\0&I_{d_2} \end{smallmatrix} \big)$,
\item for all $|z|<r$, $P^\p(z)$ is a projection that commutes with $E(z)$,
\item in the limit $z\to 0$, $E(z)P^\p(z) = B + \bigO(z)$.
\end{enumerate}
\end{prop}

We will need to split the hyperbolic branches further.
\begin{prop}\label{th:perturb}
There exist $r>0$ and a family of matrix-valued holomorphic functions $(P_\mu^\h)_{\mu\in\Sp(A')} \in (\hol_{r}^{d\times d})^{\Sp(A')}$ satisfying
\begin{enumerate}[ref=\textit{\roman*)}]
\item
\label{eigenprojection}
for all $\mu \in \Sp(A')$ and $|z|<r$, $P_\mu^\h(z)$ is a non-zero projection that commutes with $E(z)$,
\item 
\label{Decompeigenprojection}
for all $|z|<r$, $P^\h(z) = \sum\limits_{\mu \in \Sp(A')} P^\h_\mu(z)$ and for all $\mu\neq \mu', P_\mu^\h(z) P_{\mu'}^\h(z) = 0$,
\item
\label{defAsymHyp}
for every $\mu\in \Sp(A')$, there exists $R_\mu^{\h} \in \mathcal{O}_{r}^{d\times d}$ such that
\begin{equation*}
\forall |z|<r,\ E(z)P_\mu^\h(z) = \mu z P_\mu^\h(z) + z^2R_\mu^\h(z).
\end{equation*}
\end{enumerate}
\end{prop}

\begin{remark}
\label{RmkRk}
For $\mu \in \Sp(A')$, the projection $P_\mu^\h$ is holomorphic and thus continuous in $D(0,r)$. Therefore, the rank of $P_\mu^\h(z)$, which is its trace, does not depend on $|z|<r$ (the $P_\mu^\h(z)$ even are similar, see~\cite[Ch.~I, \S4.6, Lem.~4.10]{kato_1995}). In the same vein, the ranks of $P^\h(z)$ and $P^\p(z)$ do not depend on $z$.
\end{remark}

\begin{proof}
The proof is essentially the ``reduction process'' of Kato~\cite[Ch.~II \S2.3]{kato_1995}. 
According to Prop.~\ref{th:perturb_hyper}, $P^\h$ is holomorphic and $P^\h(z)E(z) = \bigO(z)$. Then we define 
\[E^{(1)}(z) = z^{-1} E(z) P^\h(z) = z^{-1} P^\h(z)E(z),\]
which is holomorphic in $|z|<r$. Note that we have according to Kato~\cite[Ch.~II Eq.~(2.38)]{kato_1995}
\[E^{(1)}(0) = P^\h(0)E'(0)P^\h(0) = \begin{pmatrix}A'&0\\0&0\end{pmatrix}.\]

Let us assume for the moment that $0$ is not an eigenvalue of $A'$. Then, for $\mu\in \Sp(A')$, we define $P_\mu^\h(z)$ the total projection on the $\mu$-group of eigenvalues of $E^{(1)}(z)$. Said otherwise, and according to the definition of $E^{(1)}(z)$, $P^\h_\mu(z)$ is the total projection on the $\mu z$-group of eigenvalues of $E(z)$. The projection $P_\mu^\h(z)$ is defined and holomorphic for $z$ small enough according to~\cite[Ch.~II, \S1.4]{kato_1995}.

Since for $z$ small enough, $P_\mu^\h(z)$ is the projection on some eigenspaces of $E^{(1)}(z)$ associated with non-zero eigenvalues, 
\[ \Ima(P_\mu^\h(z))\subset\Ima(E^{(1)}(z)) \subset \Ima(P^\h(z)),\]
with the last inclusion coming from the definition of $E^{(1)}(z)$. Thus $P_\mu^\h(z)$ is a subprojection of $P^\h(z)$. Moreover, $P_\mu^\h(z)$ commutes with $E^{(1)}(z)$, so it commutes with $E(z)$. This proves Item~\ref{eigenprojection} in the case $0\notin \Sp(A')$.

For $\mu\neq \nu$, $P_\mu^\h(z)$ and $P_\nu^\h(z)$ are the projections on  some sums of eigenspaces associated with different eigenvalues, so $P_\mu^\h(z)P_\nu^\h(z) = 0$. Let us note for convenience $Q^\h(z) = \sum_{\mu\in\Sp(A')} P_\mu^\h(z)$. Then, for $z$ small, $Q^\h(z)$ is the projection on all the eigenspaces of $E^{(1)}(z)$ associated with non-zero eigenvalues. According to the definition of $E^{(1)}(z)$, this proves that $Q^\h(z)$ is a subprojection of $P^\h(z)$. Let us check that $Q^\h(z)$ and $P^\h(z)$ have the same rank. This will prove that for all $z$ small enough, $Q^\h(z) = P^\h(z)$. The rank of $Q^\h(z)$, which is its trace, does not depend on $z$. The same is true for $P^\h(z)$. For $z=0$, we have $E^{(1)}(0) = (\begin{smallmatrix} A'&0\\0&0 \end{smallmatrix})$, so by using the fact that $0\notin \Sp(A')$,
\[Q^\h(0) = \begin{pmatrix} I_{d_1}&0\\0&0 \end{pmatrix} = P^\h(0).\]
This proves that for all $z$ small enough, $Q^\h(z) = P^\h(z)$, and in turn finishes the proof of Item~\ref{Decompeigenprojection} in the case where $0\notin\Sp(A')$.

If $0\in\Sp(A')$, then we add $\alpha zI$ to $E(z)$ for some $\alpha\in\set C$. This amounts to adding $\alpha P^\h(z)$ to $E^{(1)}(z)$. This only shifts the eigenvalues of  the restriction of $E^{(1)}(z)$ to $\Ima(P^\h(z))$ (but not of its restriction to $\Ima(I_d-P^\h(z))$) by $\alpha$, while leaving the eigenprojections unchanged. Thus, choosing $\alpha$ so that $0\notin \alpha+\Sp(A')$, we get the Items~\ref{eigenprojection} and~\ref{Decompeigenprojection} in the case $0\in\Sp(A')$.

We still need to prove the asymptotics of Item~\ref{defAsymHyp}. Since $A'$ is diagonalizable, so is $E^{(1)}(0) = (\begin{smallmatrix} A'&0\\0&0 \end{smallmatrix})$. So, there is no nilpotent part in the spectral decomposition of $E^{(1)}(0)$. That is to say, for all $\mu\in\Sp(A')$,
\[E^{(1)}(0)P_\mu^\h(0) = \mu P_\mu^\h(0).\]
Since $z\mapsto E^{(1)}(z)P_\mu^\h(z)$ is holomorphic,  we have
\[E^{(1)}(z)P_\mu^\h(z) = \mu P_\mu^\h(z)+\bigO(z).\]
Finally, we multiply by $z$ to come back to $E(z)$, which gives us
\[E(z)P_\mu^\h(z) = \mu zP_\mu^\h(z)+\bigO(z^2).\qedhere\]
\end{proof}

\subsection{Estimates on Fourier components and well-posedness}
\label{subsubsec:EF_WP}
\subsubsection{Dissipation of the parabolic component}

The goal of this section is the proof of the following result.

\begin{prop} \label{Prop:dissip_p}
There exist $r, K_\p, c_\p >0$ such that for every $|z|<r$, $\tau >0$ and $X\in\Ima(P^\p(z))$,
\[|\eu^{- E(z) \tau} X | \leq K_\p \eu^{-c_\p \tau} |X|.\]
\end{prop}

\begin{proof}
By using \Cref{th:perturb_parab}, for $|z| \leq r$, we denote by $E^\p(z)$ the restriction of $E(z)$ to the vector subspace $\Ima[P^\p(z)]$, which is an endomorphism of $\Ima[P^\p(z)]$.

By assumption (\ref{h:D}), there exists $c>0$ such that $\Re(\Sp(D)) \subset (c, \infty)$. 
There exists an open disk $\Omega$ in the complex plane such that $\Sp(D) \subset \Omega$ and $\min\{ \Re(z);\ z\in \overline{\Omega} \} >c$. Then, by continuity of the spectrum, for $r$ small enough, we have, for every $|z|\leq r$, $\Sp( E^\p(z) ) \subset \Omega$.

\paragraph{Step 1: Cauchy formula.} We prove the following equality between endomorphisms of $\Ima[P^\p(z)]$
\begin{equation} \label{exp_Cauchy} 
\forall |z|\leq r,\ \tau \in \R, \quad
\eu^{ -E^\p (z)\tau}=\frac{1}{2\pi \iu } \int_{\partial\Omega} \eu^{-\tau \xi} \left( \xi I - E^\p(z) \right)^{-1} \diff\xi,
\end{equation}
where $I$ is the identity on $\Ima[P^\p(z)]$.
The right hand side is well defined because $\partial\Omega \cap\Sp(E^\p(z)) = \emptyset$. Let us denote it by $\phi(\tau)$.
Then
\begin{align*}
\phi'(\tau) & =\frac{-1}{2\pi \iu } \int_{\partial\Omega} \eu^{-\tau \xi} \xi \left( \xi I - E^\p(z) \right)^{-1} \diff\xi 
\\ & = \frac{-1}{2\pi \iu } \int_{\partial\Omega} \eu^{-\tau \xi} \left( (\xi I - E^\p(z))+E^\p(z) \right) \left( \xi I - E^\p(z) \right)^{-1} \diff\xi.
\end{align*}
By the Cauchy formula, $\int_{\partial\Omega} \eu^{-\tau \xi} \diff\xi=0$ thus
$\phi'(\tau)=-E^\p(z)\phi(\tau)$. Moreover $\phi(0)=I$ because all the eigenvalues of $E^\p(z)$ are inside $\Omega$ (see~\cite[Ch.~I, Problem~5.9]{kato_1995}). Thus $\phi(\tau)=\eu^{-\tau E^\p(z)}$.

\emph{Step 2: Estimate.}
We deduce from (\ref{exp_Cauchy}) the following equality between endomorphisms of $\C^d$
\begin{equation} \label{exp_Cauchy_Pp} 
\forall |z|\leq r,\ \tau \in \R, \quad
\eu^{ - E (z)\tau} P^\p(z) =\frac{1}{2\pi \iu } \int_{\partial\Omega} \eu^{-\tau \xi} \left( \xi I_d - E(z) \right)^{-1} P^\p(z) \diff\xi.
\end{equation} 
Note that, if $r$ is small enough, then the eigenvalues of $E(z)$ are either inside $\Omega$ (parabolic branch) or close to $0$ (hyperbolic branch), for instance in $\{ \Re(\xi) < c/2 \}$. Thus $\left( \xi I_d - E(z) \right)$ is invertible on $\C^d$ for every $\xi \in \partial\Omega$ and the above right hand side is well defined.

We deduce from (\ref{exp_Cauchy_Pp}) that
\[\left| \eu^{-E(z)\tau} P^\p(z) \right|
\leq \frac{1}{2\pi }\int_{\partial\Omega} \eu^{-\tau \Re(\xi)} \left| \left( \xi I_d - E(z) \right)^{-1} P^\p(z) \right| \diff\xi.\]
The map $(\xi,z)\in\partial\Omega \times \overline{D}(0,r) \mapsto \left| \left( \xi I_d - E(z) \right)^{-1} P^\p(z) \right|$ is continuous  on a compact set thus bounded. Thus there exists a positive constant $K$ such that, for every $|z|<r$  and $\tau>0$,
$\left| \eu^{-E(z)\tau} P^\p(z) \right| \leq K \eu^{-c\tau}$.
\end{proof}

\subsubsection{Boundedness of the transport component}

The goal of this section is to prove the following result.

\begin{prop} \label{Prop:dissip_h}
There exists $r, K_\h, c_\h >0$ such that for every $x\in [-r,r]\setminus\{0\}$, $t\in\R$ and $X \in \Ima(P^\h(\iu x))$,
\[\left\lvert \exp\left( \frac{t}{x^2} E(\iu x) \right) X \right\rvert \leq K_\h \eu^{c_\h |t|} |X|.\]
\end{prop}

\begin{proof}
Let $r$ be as in \Cref{th:perturb}, $x \in [-r,r]\setminus\{0\}$, $t\in\R$, $\mu \in \Sp(A')$ and $Y \in \Ima[P^\h_\mu(\iu x)]$.
Taking into account that $\Ima[P^\h_\mu(\iu x)]$ is stable by $E(\iu x)$, we get
\[\exp \left( \frac{t}{x^2} E(\iu x) \right) Y 
=\exp \left( \frac{t}{x^2} E(\iu x) P_\mu^\h(\iu x) \right) Y \\
=\exp \left( \frac{t}{x^2} \left( \iu\mu x  P_\mu^\h(\iu x)- x^2 R_\mu^\h(\iu x) \right)  \right) Y.\]
Note that $P_\mu^\h(\iu x)$ and $R_\mu^\h(\iu x)$ commute 
because $P_\mu^\h(\iu x)$ and $E(\iu x)$ commute and
$E(\iu x)P_\mu^\h(\iu x)=\mu ix P_\mu^\h(\iu x)-x^2 R_\mu^\h(\iu x)$. Thus, by using that $\iu \mu/ x \in \iu\R$,  we obtain
\[\left\lvert \exp \left( \frac{t}{x^2} E(\iu x) \right) Y \right\rvert =
\left\lvert\eu^{\iu\mu t/x} \exp\left(-tR_\mu^\h(\iu x) \right) Y\right\rvert 
\leq \eu^{c_\mu |t|} | Y |,\]
where $c_\mu=\max\{ \lvert R_\mu^\h(z)\rvert ;\: z \in \overline{D}(0,r) \}$.
We conclude for $X\in\Ima[P^\h(\iu x)]$ that
\begin{align*}
\left\lvert \exp \left( \frac{t}{x^2} E(\iu x) \right) X \right\rvert
& \leq
  \sum\limits_{\mu \in \Sp(A')} \left\lvert \exp \left( \frac{t}{x^2} E(\iu x) \right) P^\h_\mu(\iu x) X \right\rvert 
\\ & \leq
 \sum\limits_{\mu \in \Sp(A')}  \eu^{c_\mu |t|} \lvert  P^\h_\mu(\iu x) X \rvert
 \leq
K \eu^{c|t|} |X|
\end{align*}
with 
$c=\max\{ c_\mu ;\: \mu \in \Sp(A') \}$ and
$K=\max\left\{  \sum_{\mu \in \Sp(A')}   \lvert  P^\h_\mu(z)  \rvert;\: z \in \overline{D}(0,r) \right\}$. 
\end{proof}

\subsubsection{Well-posedness}\label{sec:wp}
By gathering the results of the previous two subsubsections, we can prove that the heat-transport system~\eqref{Syst} is well-posed.
We define the Fourier coefficients of $f\in L^2(\set T)^d$ by 
\[\forall n \in \Z, \  \hat{f}(n)=\frac1{2\pi}\int_{\T} f(t) \eu^{-\iu nt} \diff t \in \C^d.\]
We consider the operator $\mathcal{L}$ defined by (\ref{defL}) and (\ref{def:D(L)}). By Bessel-Parseval identity and the fact that $\mathcal L(Xe_n) = n^2E(\iu/n)X e_n$,
\begin{equation}\label{D(L)_BIS}
D(\mathcal{L})=\Big\{ f \in L^2(\T)^d;\  
\sum\limits_{n\in\Z} \left\lvert n^2 E\left(\frac\iu n\right) \hat{f}(n)\right\rvert^2
< \infty \Big\}.
\end{equation}
The goal of this section is to prove the following result.

\begin{prop} \label{Prop:C0_sg}
$-\mathcal{L}$ generates a $C^0$ semi-group of bounded operators on $L^2(\T^d)$.
\end{prop}

This result will ensure well posedness of \eqref{Syst} in the following sense.

\begin{definition}\label{def:solution}
Let $T>0$, $f_0 \in L^2(\T)^{d}$ and $u \in L^2(Q_T)^{d}$.
The solution of \eqref{Syst} is the function $f \in C^0([0,T];L^2(\T)^d)$ defined for $t\in [0,T]$ by
\[f(t)=\eu^{-t\mathcal{L}} f_0+ \int_0^t \eu^{-(t-\tau)\mathcal{L}} u(\tau) \diff\tau.\]
Moreover, $f(t)$ satisfies the estimate
\begin{equation} \label{WP_estim}
\forall 0\leq t\leq T,\ \lVert f(t)\rVert_{L^2(\set T)}\leq C\big(\lVert f_0\rVert_{L^2(\set T)} + \lVert u\rVert_{L^2([0,T]\times\omega)}\big),
\end{equation}
where $C$ depends on $T$ but not on $f_0$ and $u$. We will also note $S(t,f_0,u) \coloneqq f(t)$ this solution.
\end{definition}


\begin{proof}
We deduce from \Cref{Prop:dissip_p} and \Cref{Prop:dissip_h} that for every $x \in [-r,r]\setminus\{0\}$, $t> 0$ and $X \in \C^d$,
\begin{equation} \label{EDO_estim}
\begin{aligned} 
\left\lvert \exp\left( -\frac{t}{x^2} E(\iu x) \right) X \right\rvert
& \leq 
\left\lvert \exp\left( - E(\iu x) \frac{t}{x^2} \right) P^\p(\iu x) X \right\rvert +
\left\lvert \exp\left(- \frac{t}{x^2} E(\iu x) \right) P^\h(\iu x) X \right\rvert
\\ & \leq
K_\p \eu^{-c_\p tx^{-2}} \left\lvert P^\p(\iu x) X \right\rvert 
+
K_\h \eu^{c_\h t} \left\lvert P^\h(\iu x) X \right\rvert
\\ & \leq
K \eu^{c_\h t} \left\lvert  X \right\rvert 
\end{aligned}
\end{equation}
where $K=\max\left\{ K_\p  \left\lvert P^\p(\iu x) \right\rvert 
+ K_\h  \left\lvert P^\h(\iu x)  \right\rvert ;  x \in [-r,r] \right\}$.

For $f \in L^2(\T)^d$ and $t\in[0,\infty)$ we define
\[S(t)=\sum\limits_{n\in\Z} \eu^{-tn^2 E\left( \frac{\iu}{n} \right)}  \hat{f}(n) e_n.\]
By Bessel Parseval equality and (\ref{EDO_estim}) with $x=1/n$, $S(t)$ is a bounded operator on $L^2(\T)^d$, because the number of $n\in\Z$ such that $\frac{1}{n} \notin [-r,r]$ is finite.
The semi-group properties $S(0)=I$ and $S(t+s)=S(t) S(s)$ are clearly satisfied. For $f\in D(\mathcal{L})$, we have, by Bessel Parseval equality
\[
\left\lVert \left(\frac{S(t)-I}{t} + \mathcal{L} \right)f \right\rVert_{L^2(\T)^d}^2\mkern-12mu 
 =
\sum_{n\in\Z} \left\lvert  \left(\dfrac{\eu^{-tn^2 E\left( \frac{\iu}{n} \right)}-I_d}{t} - n^2 E\left(\frac\iu n\right) \right) \hat{f}(n)\right\rvert^2 .
\]
In the right hand side, each term of the series converges to zero when $[t \rightarrow 0]$ and, thanks to (\ref{EDO_estim}), is dominated for every $t \in [0,1]$ and $n>1/r$ by
\[  \left\lvert  
 \left( \int_0^1 \eu^{-t\theta n^2 E\left( \frac{\iu}{n} \right)} \diff\theta - I_d \right)  n^2 E\left(\frac\iu n\right) \hat{f}(n)\right\rvert^2 \leq (K \eu^{c_h} + 1)^2\left\lvert   n^2 E\left(\frac\iu n\right) \hat{f}(n)\right\rvert^2,
 \]
 which can be summed over $n\in\Z$ because $f\in D(\mathcal{L})$, see (\ref{D(L)_BIS}).
 By the dominated convergence theorem, the sum of the series converges to zero.
\end{proof}

\begin{remark}\label{rk:semigroup_Hs}
We can see from this proof that the semi-group $\eu^{-t\mathcal L}$ is strongly continuous on any $H^s(\T)^d$ for any $s\geq 0$, i.e.\ we have
\[
\|\eu^{-t\mathcal L} f_0\|_{H^s(\T)^d} \leq K\eu^{c_\h t}\|f_0\|_{H^s(\T)^d}.
\]
\end{remark}

\subsection{Adjoint system and observability}
\label{Paragraphdual}

The null-controllability of a linear system is equivalent to a dual notion called ``observability''. We have the following general, abstract result (see \cite[Lemma 2.48]{coron_2007}).
\begin{lemma}
\label{LemmaDuality}
Let $H_1$, $H_2$ and $H_3$ be three Hilbert spaces. Let $\Phi_2\colon H_2 \to H_1$ and $\Phi_3\colon H_3 \to H_1$ be continuous linear maps. Then 
\[ \Ima(\Phi_2) \subset \Ima(\Phi_3)\]
if and only if there exists $C>0$ such that 
\[ \forall h_1 \in H_1,\ \norme{\Phi_2^{*} h_1}_{H_2} \leq C \norme{\Phi_3^{*} h_1}_{H_3}.\]
\end{lemma}
 
From \Cref{LemmaDuality}, see \cite[Theorem 2.44]{coron_2007}, we deduce the following result.
\begin{prop}\label{Hum}
Given $T>0$, the system \eqref{Syst} is null-controllable on $\omega$ in time $T$ if and only if there exists $C>0$ such that for every $g_0 \in L^2(\T;\C^{d})$, the solution $g$ to the equation
\begin{equation}
\label{SystAdj}
\left\{
\begin{array}{l l}
\partial_t g -  \Tr{B} \partial_{x}^2 g  - \Tr{A} \partial_x g + \Tr{K} g  = 0&\mathrm{in}\ (0,T)\times\T,\\
g(0,\cdot)=g_{0}& \mathrm{in}\ \T.
\end{array}
\right.
\end{equation}
satisfies
\begin{equation}
\label{ObsHumG}
\norme{g(T,\cdot)}_{L^2(\T;\C^{d})}^2 \leq C \int_{0}^{T} \int_{\omega} |M^* g(t,x)|^{2} \diff t \diff x.
\end{equation}
\end{prop}

Note that the solutions of the adjoint system~\eqref{SystAdj} are of the form\footnote{When we write $E(z)^\ast$, it is to be understood as $(E(z))^\ast$. We will use the same notation for $P_\mu^\h(z)^\ast$ etc.}
\begin{equation}\label{eq:sol_adj}
g(t,x) = \sum_{n\in\set Z} \eu^{- tn^2E\left(\frac\iu n\right)^\ast} \widehat{g}_0(n) \eu^{\iu nx}.
\end{equation}
Moreover, we have a spectral theory for the adjoint system that is similar to Prop.~\ref{th:perturb_hyper}--\ref{th:perturb}. We just have to take the adjoint of each formulas of these Propositions.

\begin{remark}\label{rk:semigroup*_Hs}
As for the semi-group $\eu^{-t\mathcal L}$ (see \Cref{rk:semigroup_Hs}), the dual semi-group $\eu^{-t\mathcal L^*}$ is strongly continuous on any $H^s(\T)^d$ for any $s\geq 0$, i.e.\ we have
\[
\|\eu^{-t\mathcal L^*} g_0\|_{H^s(\T)^d} \leq K'\eu^{c' t}\|g_0\|_{H^s(\T)^d}.
\]
\end{remark}

\section{Obstruction to the null-controllability in small time} \label{Sec:Obst}

The goal of this section is to prove the first point of \Cref{th:main}, by disproving the observability inequality \eqref{ObsHumG} on an appropriate solution of the adjoint system~\eqref{SystAdj}. Intuitively, the lack of null-controllability in small time should come from the transport components. So, the idea is to construct approximate transport solutions. Note that in general, there is no \emph{exact} transport solutions that are supported on a strict subset of $[0,T]\times \set T$ (see Appendix~\ref{app:transport}).

\begin{proof}[Proof of the lack of null-controllability in time $T<T^*$]

\paragraph{Step 1: Construction of approximate transport solutions.} Let $\mu \in \Sp(A')$ with minimum absolute value (i.e.\ $|\mu| = \mu^*$). Let $\chi\in C^\infty(\set T)\setminus\{0\}$ be such that the solution $\eta(t,x)=\chi(x+\mu t)$ of the transport equation $(\partial_t-\mu\partial_x)\eta(t,x) = 0$ on $[0,T]\times\set T$, with initial condition $\eta(0,\cdot) = \chi$, has its support $\Supp(\eta)$ is disjoint from $[0,T]\times \omega$. Such a solution exists because $T<T^*$.

To exploit the spectral asymptotics of the previous section, that are valid in the high-frequency limit, we want a high-frequency version of $\chi$. To that end, for $N\in\set N^*$ we consider the polynomial $P_N(X) = \prod_{j=-N}^N(X-j)$ and $\chi_N = P_N(-\iu \partial_x)\chi$. Since $\chi_N$ is the image of $\chi$ by a differential operator, we have $\Supp(\chi_N)\subset \Supp(\chi)$. If we note $\chi(x) = \sum_{n\in\set Z} a_n \eu^{\iu nx}$ and $\chi_N(x) = \sum_{n\in\set Z} a_n^N \eu^{\iu nx}$, we have $a_n^N = P_N(n) a_n$. In particular, for $|n|\leq N$, $a_n^N = 0$.

In summary, $\chi_N$ satisfies the following properties
\begin{itemize}
\item $\chi_N$ is non-zero,
\item $\chi_N$ has no components along frequencies less than $N$,
\item the support of $\chi_N$ is a subset of the support of $\chi$.
\end{itemize}
In particular, the last property implies that the solution $\eta_N$ of $(\partial_t-\mu\partial_x)\eta_N(t,x) = 0$, with initial condition $\eta_N(0,\cdot) = \chi_N$ is such that $\Supp(\eta_N)$ is disjoint from $[0,T]\times \omega$.

We adopt the notations of \Cref{th:perturb}. Let $\varphi_0\in\Ima(P_\mu^\h(0)^*)\setminus\{0\}$. We define
\begin{align}
\widetilde g_N(t,x) 
  &= \sum_{n\in\set Z} a_n^N \eu^{\iu n(x+\mu t)+tR_\mu^\h(0)^*}\varphi_0 
  = \chi_N(x+\mu t)\eu^{+tR_\mu^\h(0)^*}\varphi_0\label{eq:def_gN_approx}\\
g_N(t,x), 
  &=\sum_{n\in\set Z} a_n^N 
   \eu^{\iu n(x+\mu t) +tR_\mu^\h\iun^*}P_\mu^\h\iun^{\!*}\varphi_0. \label{eq:def_gN}
\end{align}

By \Cref{th:perturb}, $E(z)^*$ acts as $\mu z +z^2R_\mu^\h(z)$ on the range of $P_\mu^\h(z)^*$. So the definition of $g_N$ can be written alternatively as
\[
g_N(t,x) = \sum_{n\in\set Z} a_n^N 
   \eu^{\iu nx-tn^2 E\iun^*}P_\mu^\h\iun^{\!*}\varphi_0.
\]
So, according to the representation of the solutions of the adjoint system (Eq.~\eqref{eq:sol_adj}), $g_N$ solves the parabolic-transport system~\eqref{SystAdj}. On the other hand, $\widetilde g_N$ solves the transport equation $(\partial_t-\mu\partial_x -R_\mu^\h(0)^*)\widetilde g_N = 0$. We will prove that in the limit $N\to +\infty$, $g_N$ is an approximation of $\widetilde g_N$.

\paragraph{Step 2: Approximation of the exact solution by the transport solution.}
According to Parseval's identity, we have for every $t\geq 0$
\[
\lVert g_N(t,\cdot)-\widetilde g_N(t,\cdot)\rVert^2_{L^2(\set T)}
= 2\pi\sum_{n\in\set Z} |a_n^N|^2\left| \left(
  \eu^{tR_\mu^\h(0)^*} - \eu^{tR_\mu^\h\iun^*}P_\mu^\h\iun^{\!*}
 \right)\varphi_0\right|^2.
 \]
Then, according to the holomorphy of $z\mapsto R_\mu^\h(z)$ and $z\mapsto P_\mu^\h(z)$, and the fact that $P_\mu^\h(0)^*\varphi_0 = \varphi_0$, we have locally uniformly with respect to $t\in [0,T^*]$
\[
\lVert g_N(t,\cdot)-\widetilde g_N(t,\cdot)\rVert^2_{L^2(\set T)}
= 2\pi\sum_{n\in\set Z} |a_n^N|^2\bigO\left(\frac1{n^2}\right).
 \]
Now, reminding that for $|n|\leq N$, $a_n^N = 0$, we deduce that
\[
\lVert g_N(t,\cdot)-\widetilde g_N(t,\cdot)\rVert^2_{L^2(\set T)}
= \bigO\left(\frac1{N^2}\right)\sum_{n\in\set Z} |a_n^N|^2.
 \]
Thanks to Parseval's identity we rewrite it as
\begin{equation}
\label{eq:approx_gN}
\lVert g_N(t,\cdot)-\widetilde g_N(t,\cdot)\rVert_{L^2(\set T)} =
\bigO\left(\frac 1N\right)\lVert\chi_N\rVert_{L^2(\set T)}.
\end{equation}

\paragraph{Step 3: Conclusion.}
By the triangle inequality, we have for $0\leq t\leq T$
\[
\lVert g_N(t,\cdot)\rVert_{L^2(\omega)} \leq \lVert \widetilde g_N(t,\cdot)\rVert_{L^2(\omega)}
+\lVert g_N(t,\cdot)-\widetilde g_N(t,\cdot)\rVert_{L^2(\omega)}.
\]
Then, since the support of $\widetilde g_N$ does not intersect $[0,T]\times \omega$, the first term of the right-hand side is zero, and according to the inequality~\eqref{eq:approx_gN}, we have uniformly in $0\leq t\leq T$
\[
\lVert g_N(t,\cdot)\rVert_{L^2(\omega)}^2 = 
 \bigO\left(\frac 1{N^2}\right)\lVert\chi_N\rVert_{L^2(\set T)}^2.
\]
Integrating this estimate for $0\leq t\leq T$, we get the following upper bound on $\lVert g_N\rVert_{L^2([0,T]\times\omega)}$:
\begin{equation}
\label{eq:upper_gN}
\lVert g_N\rVert_{L^2([0,T]\times \omega)}^2 = 
 \bigO\left(\frac 1{N^2}\right)\lVert\chi_N\rVert_{L^2(\set T)}^2.
\end{equation}

To disprove the observability inequality, we also need a lower bound of $\| g_N(T,\cdot) \|_{L^2(\T)}$. According to Parseval's identity, we have 
\[
\lVert g_N(T,\cdot)\rVert_{L^2(\set T)}^2 
= 2\pi\sum_{n\in \set Z} |a_n^N|^2\left| \eu^{TR_\mu^\h\iun^* } P_\mu^\h\iun^{\!*} \varphi_0\right|^2.
\]
Since $\varphi_0$ is in the range of $P_\mu^\h(0)$, for $n$ large enough, we have $|\eu^{TR_\mu^\h\iun^* } P_\mu^\h\iun^{\!*}\varphi_0| \geq c>0$. Then, since $a_n^N = 0$ for $|n|\leq N$, as soon as $N$ is large enough, 
\begin{equation}
\label{eq:lower_gN}
\lVert g_N(T,\cdot)\rVert_{L^2(\set T)}^2 
\geq 2\pi c^2\sum_{n\in \set Z} |a_n^N|^2 = 2\pi c^2\lVert \chi_N\rVert_{L^2(\set T)}^2.
\end{equation}

Comparing the lower bound~\eqref{eq:lower_gN} and the upper bound~\eqref{eq:upper_gN}, we see that the observability inequality~\eqref{ObsHumG} (with $M = \text{identity matrix of size } d$) cannot hold.
\end{proof}

\section{Large time null-controllability} \label{sec:NC}

The goal of this section is to prove the point $(ii)$ of \Cref{th:main}. An adapted decomposition of $L^2(\T)^d$ is introduced in \Cref{subsec:decomp}. The control strategy is presented in \Cref{subsec:CS}.
Projecting the dynamics onto the parabolic/hyperbolic subspaces, the system is decomposed into 2 weakly coupled systems, the first one behaving like a transport equation, the second one like a heat equation. The transport equation is handled  in \Cref{subsec:hyp} by using the methods developped in~\cite{alabau-boussouira_2017}. The parabolic equation is treated in \Cref{subsec:parab} by adapting the Lebeau-Robbiano method~\cite{lebeau_1995} to systems with arbitrary size. The low frequency part is treated by a compactness argument and a unique continuation property in \Cref{ControlLowFreqSec}.

In the whole \Cref{sec:NC}, the parameter $r>0$ is assumed to be small enough so that Propositions \ref{th:perturb_hyper}, \ref{th:perturb_parab}, \ref{th:perturb}, \ref{Prop:dissip_p} and \ref{Prop:dissip_h} hold.

\subsection{An adapted decomposition of $L^2(\T)^{d}$} \label{subsec:decomp}

\begin{prop}
\label{decompL2}
Let $n_0 \in \N$ be such that $\frac{1}{n_0} < r$. We have the following decomposition
\begin{equation}
\label{DecompRd}
L^2(\T)^{d} = F^0 \oplus F^\p \oplus F^\h,
\end{equation}
where 
\begin{align}
\label{defW0}
F^0 &\coloneqq \bigoplus\limits_{|n| \leq n_0} \set C^d e_n,\\
\label{defWp}
F^\p &\coloneqq\bigoplus\limits_{|n| > n_0}  \Ima\left(P^\p\Big(\frac \iu n\Big)\right) e_n,\\
\label{defWh}
F^\h &\coloneqq\bigoplus\limits_{|n| > n_0}  \Ima\left(P^\h\Big(\frac \iu n\Big)\right) e_n.
\end{align}
Moreover the projections $\Pi^{0}$, $\Pi^{\p}$, $\Pi^{\h}$ and $\Pi$ defined by
\[\begin{array}{r@{}c@{}c@{}c@{}c@{}c@{}c@{}c}
L^2(\T)^{d} &{}={}& F^0 & {}\oplus{} & F^\p & {}\oplus{} & F^\h & \\
\Pi^0       &{}={}& I_{F^0} & {}+{} & 0   & {}+{}      & 0  & \\
\Pi^\p       &{}={}& 0   &  {}+{}  & I_{F^\p} & {}+{}    & 0  &  \\
\Pi^\h       &{}={}& 0 & {}+{} & 0 & {}+{} & I_{F^\h} & \\
\Pi\,         &{}={}& 0 & {}+{} & I_{F^\p} & {}+{} & I_{F^\h}  & {}= \Pi^\p+\Pi^\h \end{array}\]
are bounded operators on $L^2(\T)^d$.
\end{prop} 

\begin{proof}
The function $z \in D(0,r) \mapsto P^\p(z)$ is continuous thus there exists $C>0$ such that, for every $z \in \overline{D}(0,1/n_0)$, $|P^\p(z)|\leq C$. Let $f \in L^2(\T)^d$. We deduce from
\begin{equation} \label{Pp_c0}
\sum\limits_{|n| > n_0} \left| P^\p\left(\frac{\iu}{n}\right) \hat{f}(n) \right|^2 \leq C^2
 \sum\limits_{|n| > n_0} |\hat{f}(n)|^2  \leq C^2 \|f\|_{L^2(\T)^d}^2
 \end{equation}
and Bessel-Parseval identity that the series $\sum P^\p\left(\frac{\iu}{n}\right) \hat{f}(n) e_n$ converges in $L^2(\T)^d$.
Using $I_d=P^\p(z)+P^\h(z)$, we get the decomposition
\[f = \sum\limits_{n\in\Z} \hat{f}(n) e_n =
\sum\limits_{|n|\leq n_0} \hat{f}(n) e_n +
\sum\limits_{|n| > n_0} P^\p\left(\frac{\iu}{n}\right) \hat{f}(n) e_n +
\sum\limits_{|n| > n_0} P^\h\left(\frac{\iu}{n}\right) \hat{f}(n) e_n\]
with convergent series in $L^2(\T)^d$.
This proves $L^2(\T)^{d} = F^0 + F^\p + F^\h$. The sum is direct because $(e_n)_{n\in\Z}$ is orthogonal and $\Ima(P^\p(z))\cap\Ima(P^\h(z))=\{0\}$ when $|z|<r$.
The linear mappings $\Pi^0$ and $\Pi$ are orthogonal projections, thus bounded operators on $L^2(\T)^d$. We deduce from Bessel-Parseval identity and \eqref{Pp_c0} that
$\Pi^\p$ is a bounded operator on $L^2(\T)^d$ and so is $\Pi^\h=\Pi-\Pi^\p$.
\end{proof}

The operator $\mathcal{L}$ defined in \eqref{defL} maps $D(\mathcal{L})\cap F^0 = F^0$ into $F^0$ thus we can define an operator $\mathcal{L}^0$ on $F^0$ by
$D(\mathcal{L}^0)=D(\mathcal{L})\cap F^0$ and $\mathcal{L}^0=\mathcal{L}|_{F^0}$. Moreover, $-\mathcal{L}^0$ generates a $C^0$-semi-group of bounded operators on $F^0$ and $\eu^{-t\mathcal{L}^0}=\eu^{-t\mathcal{L}}|_{F^0}$.
For the same reasons, we can define an operator $\mathcal{L}^\p$ on $F^\p$ by $D(\mathcal{L}^\p)=D(\mathcal{L})\cap F^\p$ and $\mathcal{L}^\p=\mathcal{L}|_{F^\p}$, that generates a $C^0$-semi-group of bounded operators on $F^\p$: $\eu^{-t\mathcal{L}^\p}=\eu^{-t\mathcal{L}}|_{F^\p}$.
Finally, we can define an operator $\mathcal{L}^\h$ on $F^\h$ by $D(\mathcal{L}^\h)=D(\mathcal{L})\cap F^\h$ and $\mathcal{L}^\h=\mathcal{L}|_{F^\h}$, that generates a $C^0$-semi-group of bounded operators on $F^\h$: $\eu^{-t\mathcal{L}^\h}=\eu^{-t\mathcal{L}}|_{F^\h}$.

\begin{prop}\label{ExtendSemigroups}
The operator $-\mathcal{L}^0$  generates a $C^0$ group $(\eu^{-t\mathcal{L}^0})_{t\in\R}$ of bounded operators on $F^0$.
The operator $-\mathcal{L}^\h$ generates a $C^0$ group $(\eu^{-t\mathcal{L}^\h})_{t\in\R}$ of bounded operators on $F^\h$
\end{prop}

\begin{proof}
We just need to check that $\eu^{-t\mathcal{L}}$ defines a bounded operator of $F^0$ and $F^\h$ when $t<0$. It is clear for $F^0$ because it has finite dimension. For $F^\h$, one may proceed as in the proof of \Cref{Prop:C0_sg}, noticing that the estimate of \Cref{Prop:dissip_h} is valid for any $t\in\R$.
\end{proof}

For the duality method, we will need the dual decomposition of \eqref{DecompRd}, i.e.
\begin{equation}
\label{duadecomp}
\begin{array}{ll}
&\displaystyle L^2(\T)^{d} = F^0 \oplus \widetilde{F^\p} \oplus \widetilde{F^\h},\\
&\displaystyle  \text{where}\ \widetilde{F^\p} \coloneqq \Ima\big((\Pi^{\p})^{*}\big),\ \widetilde{F^\h} \coloneqq \Ima\left((\Pi^{\h})^{*}\right).
\end{array}
\end{equation}
By using the definitions of $F^\p$ and $F^\h$ in \eqref{defWp} and~\eqref{defWh} and the fact that $(e_n)_{n \in \Z}$ is an Hilbert basis of $L^2(\T)$, we get
\begin{equation}
\label{defWptilde}
\widetilde{F^{\p}} =\bigoplus\limits_{|n| > n_0}  \Ima\left(P^\p\Big(\frac \iu n\Big)^\ast\right) e_n,
\end{equation}
\begin{equation}
\label{defWhtilde}
\widetilde{F^{\h}} =\bigoplus\limits_{|n| > n_0}  \Ima\left(P^\h\Big(\frac \iu n\Big)^\ast\right) e_n.
\end{equation}
Moreover, 
\begin{equation}\label{eq:adj_semigroup}
( \eu^{-t\mathcal L})^\ast f = \eu^{-t\mathcal L^\ast} f= \sum_{n\in\set Z} \eu^{-tn^2 E\left(\frac\iu n\right)^\ast} \widehat{f}(n) e_n
\end{equation}
and the spaces $F^0$, $\widetilde{F^\p}$ and $\widetilde{F^\h}$ are stable by $\eu^{t\mathcal L^\ast}$.

\subsection{Control strategy} \label{subsec:CS}

Let $T^{*}$ be as in \eqref{eq:T} and $T, T'$ be such that
\begin{equation}
T^\ast  < T' < T. 
\label{defTT'}
\end{equation}
In this section, we consider controls $u$ of the form 
\begin{equation}
\label{DecompH}
u \coloneqq \Tr{(u_{\h}, u_{\p})} \in \C^{d_1} \times \C^{d_2},
\end{equation}
where 
\begin{equation}
\Supp(u_{\h}) \subset [0,T']\times\overline \omega, \qquad \Supp(u_{\p}) \subset [T',T] \times \overline \omega,
\label{suppHhHp}
\end{equation}
\begin{equation*}
u_{\h} \in L^2((0,T')\times\T)^{d_1}, \qquad u_{\p} \in L^2((T',T)\times\T)^{d_2}.
\end{equation*}
The control $u_{\h}$ is intended to control the hyperbolic component of the system and the control $u_{\p}$ the parabolic component.

The control strategy for system \eqref{Syst} consists in
\begin{itemize}
\item first proving the null controllability in time $T$ in a subspace of $L^2(\T)^d$ with finite codimension,  
\item then using a unique continuation argument, to get the full null controllability.
\end{itemize}
The first step of this strategy is given by the following statement.
\begin{prop}
\label{PropControlHighFreq}
There exists a closed subspace $\mathcal{G}$ of $L^2(\T)^{d}$ with finite codimension and a continuous operator 
\begin{equation*}
\mathcal U \colon\begin{array}[t]{@{}c@{}l}
\mathcal{G} &\to  L^2((0,T')\times\omega)^{d_1} \times C^\infty_c((T',T)\times\omega)^{d_2}\\
f_0 &\mapsto (u_{\h},u_{\p}),
\end{array}
\end{equation*}
that associates with each $f_0 \in \mathcal{G} $ a pair of controls $\mathcal U f_{0}=(u_{\h}, u_{\p})$ such that 
\begin{equation}
\label{PropK}
\forall f_0 \in \mathcal{G},\ \Pi S(T;f_0, \mathcal Uf_0) = 0.
\end{equation}
\end{prop}
By “continuous operator”, we mean that, for every $s\in\N$, the map $\mathcal{U}\colon \mathcal{G} \mapsto L^2((0,T')\times\omega)^{d_1}\times H_0^s((T',T)\times \omega)^{d_2}$ is continuous: there exists $C_s>0$ such that
\[\forall f_0 \in \mathcal{G}, \quad
\|u_{\h}\|_{L^2((0,T')\times\omega)^{d_1}} + \|u_{\p}\|_{H_0^s((T',T)\times\omega)^{d_2}} \leq C_s \|f_0\|_{L^2(\T)^d}.\]
The proof strategy of \Cref{PropControlHighFreq} consists in splitting the problem in two parts:
\begin{itemize}
\item for any initial data $f_0$ and parabolic control $u_{\p}$, steer the hyperbolic high frequences to zero at time $T$ (\Cref{LemHyp}),
\item for any initial data $f_0$ and hyperbolic control $u_{\h}$, steer the parabolic high frequences to zero at time $T$ (\Cref{LemPar}).
\end{itemize}

\begin{prop}
\label{LemHyp}
If $n_0$ (in Eq.~(\ref{defW0}--\ref{defWp})) is large enough, there exists a continuous operator 
\begin{equation*}
\mathcal U^\h \colon\begin{array}[t]{@{}c@{}l}
 L^2(\T)^{d} \times L^2((T',T)\times\omega)^{d_2} & \rightarrow  L^2((0,T')\times\omega)^{d_1}\\
 (f_0,u_{\p}) &\mapsto u_{\h},
\end{array}
\end{equation*}
such that for every $(f_0,u_{\p})\in  L^2(\T)^{d} \times L^2((T',T)\times\omega)^{d_2} $,
\begin{equation*}
\Pi^{\h} S(T;f_0, (\mathcal U^\h (f_0,u_{\p}),u_{\p})) = 0.
\end{equation*}
\end{prop}
\begin{prop}
\label{LemPar}
If $n_0$ is large enough, there exists a continuous operator 
\begin{equation*}
\mathcal U^\p \colon \begin{array}[t]{@{}c@{}l}
  L^2(\T)^{d} \times L^2((0,T')\times\omega)^{d_1}  & \rightarrow  C^\infty_c((T',T)\times\omega)^{d_2} \\
 (f_0,u_{\h}) &\mapsto u_{\p},
\end{array}
\end{equation*}
such that for every $(f_0,u_{\h})\in  L^2(\T)^{d} \times L^2((0,T')\times\omega)^{d_1}$,
\begin{equation*}
\Pi^{\p} S(T;f_0, (u_{\h},\mathcal U^\p(f_0,u_{\h})) = 0.
\end{equation*}
\end{prop}

Admitting that \Cref{LemHyp} and \Cref{LemPar} hold, we can now prove \Cref{PropControlHighFreq}. 

\begin{proof}
We observe that the relation $\Pi S(T;f_0, (u_{\h},u_{\p})) = 0$ holds
if the two following equations are simultaneously satisfied
\begin{equation}
\begin{split}
\label{SystK}
 u_{\h} & = \mathcal U^\h(f_0, u_{\p}) = \mathcal U^\h_{1}(f_0) + \mathcal U^\h_{2}(u_{\p}),\\
u_{\p} &  = \mathcal U^\p(f_0, u_{\h}) = \mathcal U^\p_{1}(f_0) + \mathcal U^\p_{2}(u_{\h}).
\end{split}
\end{equation}
If we set
\begin{equation*}
C\coloneqq \mathcal U^\p_{1} + \mathcal U^\p_{2} \mathcal U^\h_{1} : L^2(\T)^{d} \rightarrow C^\infty_c((T',T)\times\T)^{d_2},
\end{equation*}
then solving system \eqref{SystK} is equivalent to 
\begin{equation}
\label{ResoudreC}
\text{find}\ u_{\p} \in C^\infty_c((T',T)\times\T)^{d_2},\ \text{such}\ \text{that}\  Cf_0 = (I-\mathcal U^\p_{2}\mathcal U^\h_{2})u_{\p}.
\end{equation}
The operator $\mathcal U^\p_{2}\mathcal U^\h_{2}$ is compact on $L^2((T',T)\times\T)^{d_2}$ because it takes values in   $C^\infty_c((T',T)\times \T)^{d_2}$. Thus, by Fredhlom's alternative (see~\cite[Thm. 6.6]{brezis_2011}), there exist $N \in \N$ and $l_1, \dots, l_N$ continuous linear forms on $L^2((T',T)\times\T)^{d_2}$ such that the equation \eqref{ResoudreC} has a solution $u_{\p} \in L^2((T',T)\times\T)^{d_2}$ if and only if 
\begin{equation}
\label{ResoudreFL}
\forall j \in \{1, \dots, N\},\ l_j(C(f_0)) = 0.
\end{equation}
Under these conditions \eqref{ResoudreFL}, the equation \eqref{ResoudreC} has a solution $u_{\p} = L(f_0)$ given by a continuous map $L\colon\mathcal{G} \rightarrow L^2((T',T)\times\T)^{d_2}$ defined on the closed vector subspace of $L^2(\T)^d$ defined by
\begin{equation}
\mathcal{G} \coloneqq \{ f_0\in L^2(\T)^d\ ;\  l_j(Cf_0) = 0,\ 1\leq j \leq N\}.
\label{defmathcalW}
\end{equation} 
Then $L(f_0)=u_{\p} =  \mathcal{U}_2^{\p}\mathcal{U}_2^{\h} u_{\p} + C f_0$ belongs to $C^\infty_c((T',T)\times\omega)$.
We get the conclusion with
\[ \forall f_0 \in \mathcal{G},\ \mathcal{U}(f_0) \coloneqq (\mathcal{U}^{\h}(f_0, L( f_0)) ,L( f_0 )).\qedhere\]
\end{proof}

\Cref{LemHyp} is proved in Section~\ref{subsec:hyp}.
\Cref{LemPar} is proved in Section~\ref{subsec:parab}.
The unique continuation argument to control the low frequencies is presented in \Cref{ControlLowFreqSec}.

\subsection{Control of the hyperbolic high frequencies} \label{subsec:hyp}
The goal of this subsection is to prove \Cref{LemHyp}. We remind that $T > T' >T^\ast$ and that the control $u = (u_\h,u_\p)$ satisfies \eqref{suppHhHp}.

\subsubsection{Reduction to an exact controllability problem}
\label{SecReduceHyp}

The goal of this paragraph is to transform the null-controllability problem of \Cref{LemHyp} into an exact controllability problem associated with an hyperbolic system. Precisely, we will get \Cref{LemHyp} as a corollary of the following result.

\begin{prop} \label{LemHyp_BIS}
If $n_0$ (in Eq.~(\ref{defW0}--\ref{defWp})) is large enough, then, for every ${T'}>T^*$, there exists a continuous operator 
\begin{equation*}
\underline{\mathcal U}_{T'}^\h \colon\begin{array}[t]{@{}c@{}l}
 F^{\h}  & \rightarrow  L^2((0,T')\times\omega)^{d_1}\\
 f_{T'} &\mapsto u_{\h},
\end{array}
\end{equation*}
such that for every $f_{T'} \in  F^{\h}$,
\begin{equation*}
\Pi^{\h} S \left( {T'}; 0 , (\underline{\mathcal U}_{T'}^\h(f_{T'}),0) \right) = f_{T'}.
\end{equation*}
\end{prop}

\Cref{LemHyp_BIS} will be proved in Section~\ref{subsec:Hyp_NC}. Now, we prove \Cref{LemHyp} thanks to \Cref{LemHyp_BIS}.

\begin{proof}[Proof of \Cref{LemHyp}]
Let  $(f_0,u_{\p})\in  L^2(\T)^{d} \times L^2((T',T)\times\omega)^{d_2} $. We have to find $u_{\h} \in L^2((0,T')\times\omega)^{d_1}$ such that
\[\Pi^{\h} S(T;f_0, (u_{\h},u_{\p})) = 0,\]
or, equivalently,
\begin{equation}
\label{HypTransf1}
\Pi^{\h} S(T;0, (u_{\h},0))= -\Pi^{\h} S(T;f_0, (0,u_{\p})).
\end{equation}
According to the well-posedness of the system~\eqref{Syst} and the continuity of the projection $\Pi^\h$ (Definition~\ref{def:solution} and Proposition~\ref{decompL2}), the linear map
\begin{equation}
\label{Continuous}
(f_0, u_{\p}) \mapsto -\Pi^{\h} S(T;f_0, (0,u_{\p})),
\end{equation}
is continuous from $L^2(\T)^{d} \times L^2((T',T)\times\omega)^{d_2}$ into $F^{\h}$, equipped with the $L^2(\T)^d$-norm. Since $u_{\h}$ is supported in $(0,T')\times\omega$ by \eqref{suppHhHp}, we have
\begin{equation}
\label{HypTransf2}
\Pi^{\h} S(T;0, (u_{\h},0))= \eu^{-(T-T')\mathcal{L}^{\h}}\Pi^{\h} S(T';0, (u_{\h},0)).
\end{equation}
As pointed out in \Cref{ExtendSemigroups}, $\eu^{t\mathcal{L}^{\h}}$ is well-defined for all $t \in \R$. Therefore, by using \eqref{Continuous} and \eqref{HypTransf2}, \eqref{HypTransf1} is equivalent to 
\begin{equation}
\label{HypTransf3}
\Pi^{\h} S(T';0, (u_{\h},0))= -\eu^{(T-T')\mathcal{L}^{h}}\Pi^{\h} S(T;f_0, (0,u_{\p})) \in F^{\h}.
\end{equation}
We get the conclusion with
\[\mathcal{U}^\h(f_0,u_p) = \underline{\mathcal{U}}_{T'}^{\h} \left(-\eu^{(T-T')\mathcal{L}^{h}}\Pi^{\h} S(T;f_0, (0,u_{\p})) \right).\qedhere\]
\end{proof}

\subsubsection{Exact controllability of the hyperbolic part} \label{subsec:Hyp_NC}

The goal of this section is to prove \Cref{LemHyp_BIS}.
By the Hilbert Uniqueness Method, \Cref{LemHyp_BIS} is equivalent to the following observability inequality (it is an adaptation of~\cite[Thm. 2.42]{coron_2007}).
\begin{prop}
\label{ObsHyp}
If $n_0$ is large enough, there exists a constant $C>0$ such that for every $g_0 \in \widetilde{F^{\h}}$, the solution $g$ of \eqref{SystAdj} satisfies
\begin{equation}
\label{inobsHyp}
\norme{g_0}_{L^2(\T)^{d}}^{2} \leq C \int_{0}^{T'} \int_{\omega} |g_{1}(t,x)|^{2} \diff t \diff x,
\end{equation}
where $g_{1}$ denotes the first $d_1$ components of $g$.
\end{prop}

\begin{proof}
Let $g_0\in  \widetilde{F^{\h}}$. By using the definition~\eqref{defWhtilde} of $\widetilde{F^\h}$, $g_0$ decomposes as follows
\begin{equation}
\label{DecompVarphiThyp}
g_0  = \sum\limits_{\mu\in\Sp(A')}\sum\limits_{|n| > n_0} {P_\mu^\h}\Big(\frac \iu n\Big)^* \widehat{g}_0(n) e_n.
\end{equation}
Then, the solution $g$ of \eqref{SystAdj} is 
\begin{equation}
\label{DecompVarphiTHypJ}
g(t) = \sum\limits_{\mu \in \Sp(A')} G_\mu(t) 
\quad \text{ where } \quad
G_\mu(t) = \sum\limits_{|n| > n_0} \eu^{-t n^2 {E}\left(\frac \iu n\right)^*}P_\mu^\h\Big(\frac \iu n\Big)^* \widehat{g}_0(n) e_n.
\end{equation}
Let $\mu\in\Sp(A')$.

\paragraph{Step 1: We prove the existence of $C_1=C_1(T')>0$, independent of $g_0$, such that}
\begin{equation}
\label{HF_Hyp_Step1}
\lVert{{G_\mu}(0,\cdot)}\rVert_{L^2(\T)^{d}} \leq C_1 \left( \|G_\mu\|_{L^2(q_{T'})^d} + \|g_0\|_{H^{-1}(\T)^d} \right)
\end{equation}
where $q_{T'} = (0,T')\times\omega$ and
\begin{equation} \label{normeH-1}
\|g_0\|_{H^{-1}(\T)^d} = \left( \sum\limits_{|n| > n_0} \frac{|\widehat{g}_0(n)|^{2}}{n^2} \right)^{1/2}.
\end{equation}
By using~\ref{eigenprojection} and~\ref{defAsymHyp} of \Cref{th:perturb}, we have
\[ \eu^{-t n^2 {E}\left(\frac \iu n\right)^*}{P_\mu^\h}\Big(\frac \iu n\Big)^* = \eu^{-tn^2\left(\mu \frac{\iu}{n} + \left(\frac \iu n\right)^{2} R_\mu^\h\left(\frac \iu n\right)\right)^*} {P_\mu^\h}\Big(\frac \iu n\Big)^* = \eu^{t \mu \iu n + t {R_\mu^\h}\left(\frac \iu n\right)^*}{P_\mu^\h}\Big(\frac \iu n\Big)^*,\]
which leads to
\begin{equation}
\label{SystAdjSAsymp}
 \partial_t G_\mu -  \mu \partial_x G_\mu  - {R_\mu^{\h}}(0)^* G_\mu  = S_\mu \quad \text{in}\ (0,T')\times\T,
\end{equation}
where
\begin{equation}\label{Smu}
S_\mu(t)=\sum\limits_{|n|>n_0} \left(R_\mu^\h\left( \frac{\iu}{n} \right)^* - R_\mu^\h(0)^*  \right) \eu^{t \mu \iu n + t {R_\mu^\h}\left(\frac \iu n\right)^*}{P_\mu^\h}\Big(\frac \iu n\Big)^* \widehat{g}_0(n) e_n.
\end{equation}
By regularity of $z\mapsto R_\mu^\h(z)$, Bessel-Parseval identity and (\ref{normeH-1}) there exists $C=C(T')>0$, independent of $g_0$, such that
\begin{equation}
\label{defsourcecomp}
 \norme{S_\mu}_{L^\infty((0,T'),L^2(\T)^{d})} \leq C \| g_0\|_{H^{-1}(\T)^d}.
\end{equation}
By (\ref{SystAdjSAsymp}), the function $\widetilde{G}_\mu$ defined by
\begin{equation} \label{def:gmutilde}
\widetilde{G}_{\mu}(t,x) =\eu^{t R_{\mu}^{\h}(0)^*} G_{\mu}(t,x)
\end{equation}
solves
\begin{equation} \label{syst_gmutilde}
\left\{
\begin{array}{l l}
 \partial_t \widetilde{G}_\mu -  \mu \partial_x \widetilde{G}_\mu    = \eu^{t R_{\mu}^{\h}(0)^*} S_\mu&\text{in}\ (0,T')\times\T,\\
\widetilde{G}_\mu(0,\cdot)={G_\mu}(0,\cdot)& \text{in}\ \T.
\end{array}\right.
\end{equation}
We introduce the solution $G_\mu^\flat$ of
\begin{equation} \label{eq_gmudiese}
\left\{
\begin{array}{ll}
 \partial_t G_\mu^\flat -  \mu \partial_x G_\mu^\flat    = 0 &\text{in}\ (0,T')\times\T,\\
G_\mu^\flat (0,\cdot)= G_\mu(0,\cdot)& \text{in}\ \T.
\end{array} \right.
\end{equation}
Using the Duhamel formula for system (\ref{syst_gmutilde}) and the estimate (\ref{defsourcecomp}), we obtain
\begin{equation} \label{dmutilde-gmudiese}
\| \widetilde{G}_\mu - G_\mu^\flat \|_{L^\infty((0,T'),L^2(\T)^d)} \leq C \| \eu^{t R_{\mu}^{\h}(0)^*} S_\mu \|_{L^1((0,T'),L^2(\T)^d)} \leq C
\| g_0 \|_{H^{-1}(\T)^d}
\end{equation}
where $C=C(T')>0$ is independent of $g_0$.
The time $T_{\mu} \coloneqq \ell(\omega)/|\mu|$ is the minimal time for the observability of the system (\ref{eq_gmudiese}) on $\omega$ (see for instance~\cite[Theorem 2.2]{alabau-boussouira_2017}). Indeed, for any $T''>T_\mu$,
\[ \T \subset\{ x - \mu t;\: (t,x) \in [0,T'']\times\omega\} . \]
Since $T'>T_\mu$, there exists $C=C(T',\omega)>0$, independent of $g_0$, such that
\[\| G_\mu(0,\cdot) \|_{L^2(\T)^d} \leq C \| G_\mu^\flat \|_{L^2(q_{T'})^d}.\]
By the triangular inequality, (\ref{def:gmutilde}) and (\ref{dmutilde-gmudiese}),  we deduce that
\[
\| G_\mu(0,\cdot) \|_{L^2(\T)^d} 
\leq C \left( \| \widetilde{G}_\mu \|_{L^2(q_{T'})^d} +  \| \widetilde{G}_\mu - G_\mu^\flat \|_{L^2(q_{T'})^d} \right) \\
\leq C \left( \|G_\mu\|_{L^2(q_{T'})^d}+ \| g_0 \|_{H^{-1}(\T)^d} 
\right)\]
which ends the first step.

\paragraph{Step 2: We prove the existence of $C_2=C_2(T',\omega)>0$, independent of $g_0$, such that}
\begin{equation} \label{HF_Hyp_Step2}
\lVert{{G_\mu}(0,\cdot)}\rVert_{L^2(\T)^{d}} \leq C_2 \left( \|P^\h_\mu(0)^* g \|_{L^2(q_{T'})^d} + \| g_0\|_{H^{-1}(\T)^d} \right).
\end{equation}
Taking into account that the projection $P^\h_\lambda(z)$ commutes wih $E(z)$ we deduce from (\ref{DecompVarphiTHypJ}) that for any $\lambda \in \Sp(A')$,
\[G_\lambda(t) = \sum\limits_{|n| > n_0} P_\lambda^\h\Big(\frac \iu n\Big)^* \eu^{-t n^2 {E}\left(\frac \iu n\right)^*} P_\lambda^\h\Big(\frac \iu n\Big)^* \widehat{g}_0(n) e_n \]
thus,
\begin{equation} \label{Gmu_Phmu_decom}
\begin{aligned} 
&G_{\mu}(t)-P^\h_\mu(0)^* g(t)\\
&\quad =\sum\limits_{|n|>n_0} \left( P_\mu^\h\Big(\frac \iu n\Big)^* - P^\h_\mu(0)^* \right) \eu^{-t n^2 {E}\left(\frac \iu n\right)^*} P_\mu^\h\Big(\frac \iu n\Big)^*\widehat{g}_0(n) e_n 
\\  &
\qquad- \sum\limits_{\lambda \in \Sp(A') \setminus \{\mu\}} \sum\limits_{|n| > n_0} P^\h_\mu(0)^* \left( P_\lambda^\h\Big(\frac \iu n\Big)^* - P_\lambda^\h (0)^* \right) \eu^{-t n^2 {E}\left(\frac \iu n\right)^*}P_\lambda^\h\Big(\frac \iu n\Big)^* \widehat{g}_0(n) e_n
\end{aligned}
\end{equation}
because, for $\lambda \neq \mu$, $P^\h_\mu(0)^* P_\lambda^\h (0)^*=0$. By using the regularity of $z\mapsto P^\h_\lambda(z)$, Bessel-Parseval identity and (\ref{normeH-1}), we obtain $C=C(T')>0$ independent of $g_0$ such that
\[\|G_{\mu} -P^\h_\mu(0)^* g \|_{L^\infty((0,T'),L^2(\T)^d)} \leq C \| g_0\|_{H^{-1}(\T)^d}.\]
We deduce from Step 1, the triangular inequality and the previous estimate that
\begin{align*}
\lVert{{G_\mu}(0,\cdot)}\rVert_{L^2(\T)^{d}} 
& \leq C \left( \|G_\mu\|_{L^2(q_{T'})} + \| g_0\|_{H^{-1}(\T)^d} \right)
\\ & \leq C \left( \|P^\h_\mu(0)^* g\|_{L^2(q_{T'})} +  \|G_\mu - P^\h_\mu(0)^* g  \|_{L^2(q_{T'})} + \| g_0\|_{H^{-1}(\T)^d} \right)
\\ & \leq C \left( \|P^\h_\mu(0)^* g\|_{L^2(q_{T'})}+ \| g_0\|_{H^{-1}(\T)^d} \right),
\end{align*}
which ends Step 2.

\paragraph{Step 3: Conclusion.} 
For every $\mu \in \Sp(A')$, we have 
$P^\h_\mu(0)^*=P^\h_\mu(0)^* P^\h(0)^*$ thus
\[\|P^\h_\mu(0)^* g\|_{L^2(q_{T'})} \leq |P^\h_\mu(0)^*| \| P^\h(0)^* g\|_{L^2(q_{T'})} \leq C \|g_1\|_{L^2(q_{T'})}.\]u
Using (\ref{DecompVarphiTHypJ}), the triangular inequality, Step 2 and the previous inequality, we obtain
\begin{equation} \label{inobsHypRed4}
\|g_0\|_{L^2(\T)^d}  \leq \sum\limits_{\mu \in \Sp(A')} \|G_\mu(0,\cdot)\|_{L^2(\T)^d} 
 \leq C \left( \|g_1\|_{L^2(q_{T'})^d}+ \| g_0\|_{H^{-1}(\T)^d} \right).
\end{equation}
From this estimate and the compact embedding $L^2(\T) \hookrightarrow H^{-1}(\T)$, a classical com\-pact\-ness-uniqueness argument gives the observability inequality \eqref{inobsHyp} (see for instance~\cite[Lemma 2.1 and Remark 2.2]{duprez_2018b}).

Indeed, by Peetre's lemma (see~\cite[Lemma 3]{peetre_1961}), we have from \eqref{inobsHypRed4} that 
\[ N_{T'} \coloneqq \{g_{0} \in \widetilde{F^{\h}};\ g_{1} = 0\ \text{in}\ (0,T')\times\omega\},\]
has finite-dimension. Moreover, from~\cite[Lemma 4]{peetre_1961}, to prove~\eqref{inobsHyp}, we only need to show that $N_{T'}$ is reduced to zero. First, by definition, we remark that $N_{T'}$ decreases as $T'$ increases. By a small perturbation of $T'$, we may therefore assume that $N_T = N_{T'}$ for $T-T'$ small, in which case $N_{T'}$ is stable by $\eu^{-t {\mathcal{L}^*}^\h}$ where ${\mathcal{L}^*}^\h$ is the restriction of ${\mathcal{L}^*}$ to $\widetilde{F^\h}$. Then, if $N_{T'}$ is not reduced to zero, it contains an eigenfunction of ${\mathcal{L}^*}^\h$, i.e.\ a function of the form $X e_n$ where $X \in \C^d$, $|n|>n_0$ and $X =P^\h\left(\frac{\iu}{n}\right) X$. By definition of $N_{T'}$, the first components of that eigenfunction vanishes on $\omega$ i.e.\ $X_1=0$, or equivalently $P^\h(0)X=0$. Thus
\[|X| = \left| \left( P^\h\left(\frac{\iu}{n}\right)  - P^\h (0) \right) X \right| \leq \frac{C}{|n|} |X|\]
where $C>0$ does not depend on $n$. For a large enough choice of $n_0$, this is impossible. 
\end{proof}

\subsection{Control of the parabolic high frequencies} \label{subsec:parab}
The goal of this subsection is to prove \Cref{LemPar}.  We recall that $T$ and $T'$ are chosen such that $T^*<T'<T$ and the control $u$ is such that~\eqref{DecompH} and~\eqref{suppHhHp} hold.

The strategy is the following one: identify the equation satisfied by the last $d_2$ components of the parabolic equation~\eqref{SystAdj} with the help of the asymptotics of \Cref{th:perturb}, then construct smooth controls by adapting the Lebeau-Robbiano's method to systems.

In this section, for every vector $\varphi\in\set C^d$, we will note $\varphi_1$ its first $d_1$ components and $\varphi_2$ its last $d_2$ components.

\subsubsection{Reduction to a null-controllability problem}\label{sec:reduc_parab}

The goal of this paragraph is to transform the null-controllability problem of \Cref{LemPar} into a null-controllability problem associated to a parabolic system. Precisely, we will prove that \Cref{LemPar} is a consequence of the following result.

\begin{prop}
\label{LemPar_Bis}
If $n_0$ is large enough, then for every $T>0$, there exists a continuous operator 
\begin{equation*}
\underline{\mathcal{U}}_T^\p \colon \begin{array}[t]{@{}c@{}l}
  F^\p  & \rightarrow  C^\infty_c((0,T)\times\omega)^{d_2} \\
  f_0 &\mapsto u_{\p},
\end{array}
\end{equation*}
such that for every $f_0 \in F^\p$,
\begin{equation*}
\Pi^{\p} S(T; f_0, (0, \underline{\mathcal{U}}_T^\p(f_0))) = 0.
\end{equation*}
\end{prop}

\Cref{LemPar_Bis} will be proved thanks to an adaptation of Lebeau-Robbiano's method in Section~\ref{Subsec:LR}, after two sections of preliminary results. Now we prove \Cref{LemPar} thanks to \Cref{LemPar_Bis}.

\begin{proof}[Proof of \Cref{LemPar}]
Let $(f_0,u_{\h}) \in L^2(\T)^d\times L^2((0,T')\times \omega)^{d_1}$. We have to find $u_{\p}\in C^\infty_c((T',T)\times\omega)^{d_2}$ such that 
\begin{equation}
\label{HypTransf1Par}
\Pi^{\p} S(T;f_0, (u_{\h},u_{\p}))= 0,
\end{equation}
or equivalently,
\begin{equation}
\label{HypTransf2Par}
\Pi^{\p} S(T;0, (0,u_{\p}))= -\Pi^{\p} S(T;f_0, (u_{\h},0)).
\end{equation}
In view of the support of the controls (Eq.~\eqref{suppHhHp}), the equality \eqref{HypTransf2Par} is equivalent to
\begin{equation}
\label{HypTransf3Par}
\Pi^{\p} S(T-T';0, (0,u_{\p}(\cdot+T')))= - \eu^{-(T-T')\mathcal{L}^{\p}} \Pi^{\p} S(T';f_0, (u_{\h},0)),
\end{equation}
or
\begin{equation}
\label{HypTransf3Par2}
\Pi^{\p} S\Big( T-T'; \Pi^{\p} S(T';f_0, (u_{\h},0))  , (0,u_{\p}(\cdot+T')) \Big)= 0.
\end{equation}
By using Definition~\ref{def:solution} and Proposition~\ref{decompL2}, we see that the mapping $(f_0, u_{\h}) \mapsto \Pi^{\p} S(T';f_0, (u_{\h},0))$ is continuous from $L^2(\T)^d \times L^2((0,T')\times \omega)^{d_1}$ into $F^{\p}$ . Thus we get the conclusion with
\[\forall t \in (T',T),\
\mathcal{U}^\p(f_0,u_h)(t)=\underline{\mathcal{U}}_{(T-T')}^\p\Big(\Pi^{\p} S(T';f_0, (u_{\h},0)) \Big)(t-T').\qedhere\]
\end{proof}

\subsubsection{Equation satisfied by the parabolic components of the free system}

We begin by proving that if $g$ is in $\widetilde{F^\p}$ then we can compute the first $d_1$ components of $g$ from the last $d_2$. This will allow us to write an uncoupled equation for these components.

\begin{prop}
\label{Prop:DefG}
If $z$ is small enough, there exists a matrix $G(z)$ such that for every $\varphi\in\set C^d$,  
\[\varphi\in\Ima(P^\p(z)^\ast) \longeq \varphi_1 = G(z)\varphi_2.\]
Moreover, $G$ is holomorphic in $z$ and $G(0) = 0$.
\end{prop}
\begin{proof}
We write 
\[P^\p(z)^\ast = \begin{pmatrix} p_{11}(z)&p_{12}(z)\\p_{21}(z)&p_{22}(z)\end{pmatrix}.\]
Since $P^\p(z)^\ast$ is a projection, $\varphi$ is in $\Ima (P^\p(z)^\ast)$ if and only if
\begin{equation*}
 \left\{\begin{aligned}p_{11}(z)\varphi_1 + p_{12}(z)\varphi_2 &= \varphi_1\\
 p_{21}(z)\varphi_1 + p_{22}(z)\varphi_2 &= \varphi_2. \end{aligned}\right.
\end{equation*}

In particular, if $\varphi\in\Ima(P^\p(z)^\ast)$, then $(I_{d_1}-p_{11}(z))\varphi_1 = p_{12}(z) \varphi_2$. And since $P^\p(0)^\ast = \big(\begin{smallmatrix}0&0\\0&I_{d_2}\end{smallmatrix} \big)$ (see Proposition~\ref{th:perturb_parab}), $p_{11}(0) = 0$, and so, if $z$ is small enough, $|p_{11}(z)|<1$ and $I_{d_1} - p_{11}(z)$ is invertible.

In that case, $\varphi_1 = (I_{d_1} - p_{11}(z))^{-1} p_{12}(z) \varphi_2$. This proves that the map 
\[\varphi \in \Ima(P^\p(z)^\ast) \mapsto \varphi_2 \in \C^{d_2}\]
is one-to-one. But the rank of $P^\p(z)^\ast$ does not depend on $z$ (Remark~\ref{RmkRk}), and so it is always $d_2$. So the previous map is bijective. We note $G(z)$ the first $d_1$ component of its inverse. Note that we have $G(z) = (I_{d_1} - p_{11}(z))^{-1} p_{12}(z)$. Then, if $\varphi\in\Ima(P^\p(z)^\ast)$, we have
\[\varphi = (\varphi_1,\varphi_2) = (G(z)\varphi_2,\varphi_2).\]

To prove the converse, note that the inverse of $\varphi\in\Ima(P^\p(z)^\ast)\mapsto \varphi_2$ is $\varphi_2\in\C^{d_2}\mapsto (G(z)\varphi_2,\varphi_2)$.
\end{proof}

Increasing $n_0$ if necessary, we may assume that for $|n|> n_0$, $G(\iu/n)$ is well-defined. Then, we define the (bounded) operator $G$ from $L^2(\set T,\set C^{d_2})$ to $L^2(\set T, \set C^{d_1})$ by
\begin{equation}\label{eq:G}
G\left( \sum_{n\in \set Z} \varphi_{n,2} e_n\right)  = \sum_{|n|>n_0} G\left(\frac \iu n\right) \varphi_{n,2} e_n.
\end{equation}
Then, according to the definition of $\widetilde{F^\p}$, we have the following corollary that allows us to compute the first $d_1$ components from the last $d_2$.
\begin{corollary}\label{th:G}
For every $g\in (F^0)^\perp$ (the space of functions with no components along frequencies less than $n_0$), we have the equivalence $g\in \widetilde{F^\p} \eq g_1 = Gg_2$.
\end{corollary}

The Corollary~\ref{th:G} makes it easy to write an equation on the last $d_2$ components of the adjoint system~\eqref{SystAdj} if the initial condition is in $\widetilde{F^\p}$.
\begin{prop}\label{LemEquationg2}
We define the operator  $\mathfrak{D}$ by
\begin{equation}\label{eq:d'}
D(\mathfrak{D})= H^2(\T)^{d_2}, \quad \mathfrak{D} = \Tr{D} \partial_x^2 + \Tr{A}_{22}\partial_x -  \Tr{K}_{22} + \Tr{A}_{12} \partial_x G -  \Tr{K}_{12}G.
\end{equation}
Let $g_0 \in \widetilde{F^\p}$ and $g(t) = \eu^{-t\mathcal L^\ast} g_0$. Then, for all $t\geq 0$, $g_1(t) = G g_2(t)$ and $g_2$ satisfies the following equation
\begin{equation}\label{eq:g_2}
\partial_t g_2(t,x) - \mathfrak{D} g_2(t,x) = 0\qquad \mathrm{in}\ (0,T)\times\T.
\end{equation}
\end{prop}

\begin{proof}
The function $g$ satisfies the system
\[(\partial_t - \Tr{B}\partial_x^2 - \Tr{A} \partial_x  + \Tr{K})g(t,x) = 0\qquad \mathrm{in}\ (0,T)\times\T.\]
If we take the last $d_2$ components of this system, we get, in $(0,T)\times\T$,
 \begin{equation}\label{eq:syst_g2}
 \left( \partial_t  - \Tr{D}\partial_x^2  - \Tr{A}_{22}\partial_x  + \Tr{K}_{22}  \right) g_2(t,x) - \left(  \Tr{A}_{12}\partial_x  - \Tr{K}_{12} \right) g_1(t,x) = 0.
 \end{equation}
 But for all $t \in [0,T]$, $g(t,\cdot)\in \widetilde{F^\p}$, so, according to Corollary~\ref{th:G}, $g_1(t) = Gg_2(t)$. Substituting this inside the equation~\eqref{eq:syst_g2} gives the stated equation~\eqref{eq:g_2}.
\end{proof}

\subsubsection{Smooth control of a finite number of parabolic vectorial components}

For $N>n_0$ we introduce
\begin{equation} \label{def:FpN}
F^\p_N \coloneqq \bigoplus\limits_{n_0 < |n| \leq N}  \Ima\left(P^\p\Big(\frac \iu n\Big)\right) e_n,
\end{equation}
\[F^\p_{>N} \coloneqq\bigoplus\limits_{|n|>N  }  \Ima\left(P^\p\Big(\frac \iu n\Big)\right) e_n.\]
and the projection $\Pi^\p_N$ defined by
\[\begin{array}{r@{}c@{}c@{}c@{}c@{}c@{}c@{}c@{}c}
L^2(\T)^d & {}={} & F^0 & {}\oplus{} & F^\p_N & {}\oplus{} & F^\p_{>N} & {}\oplus{} & F^\h \\
\Pi^\p_N   &{} = {}& 0   &{} + {}& I_{F^\p_N} &{} +{}   & 0 & {}+{} & 0
   \end{array}\]
which is a bounded operator on $L^2(\T)^d$ (compostion of the bounded operator $\Pi^\p$ with an orthogonal projection).
The  goal of this section is to prove the following result.

\begin{prop} \label{Prop:FNPF}
There exists $\mathcal{C}>0$ such that, for every $T \in(0,1]$ and $N>n_0$, there exists a linear map\footnote{The space $C_0^\infty((0,T)\times \omega)$ means that the function is supported on $[0,T]\times K$ where $K$ is a compact subset of $\omega$, and
all the derivatives vanish on $\omega$ at time $t=0$ and $t=T$.}
\[\mathcal{K}_{T,N}\colon F^\p \rightarrow C^\infty_0((0,T) \times \omega)\]
such that, for every $f_0 \in F^\p$ and $s \in\N$
\[\Pi^\p_N   S \Big(T ; f_0 , \left(0,\mathcal{K}_{T,N}(f_0)\right) \Big) = 0,\]
\[\left\| \mathcal{K}_{T,N}(f_0) \right\|_{H_0^s((0,T)\times\T)} \leq \frac{\mathcal{C} }{ T^{s+1} } N^{2s} \eu^{\mathcal{C} N} \|f_0\|_{L^2(\T)^d}.\]
\end{prop}

\begin{proof}
Let $f_0 \in F^\p$. Throughout this proof, we will note $E_2(n)$ the $d_2\times d_2$ matrices defined by
\[\forall |n|>n_0, \  E_2(n)\coloneqq \Tr{D} - \frac{\iu}{n} \Tr{A}_{22} + \frac{1}{n^2} \Tr{K}_{22} - \left( \frac{\iu}{n} \Tr{A}_{12} - \frac{1}{n^2} \Tr{K}_{12} \right) G \left( \frac{\iu}{n} \right).\]

\paragraph{Step 1:} We prove that $u_2 \in C^\infty_0((0,T)\times\omega)$ satisfies $\Pi^\p_N S(T;f_0,(0,u_2))=0$ if and only if $u_2$ solves the following moments problem in $\C^{d_2}$
\begin{multline} \label{moment}
\forall n_0 < |n| \leq N, \  
\int_0^T \int_{\omega} \eu^{-n^2(T-t)E_2(n)^*} u_2(t,x)\eu^{-\iu nx} \diff x \diff t= F_n \\  
\text{ where } F_n = - \eu^{-n^2 T E_2(n)^*} \left( G {\left( \frac{\iu}{n} \right)\!}^\ast \widehat{f}_{01}(n) + \widehat{f}_{02}(n) \right)
\end{multline}
and $E_2(n)^*=\Tr{\overline{E_2(n)}}$.

We first recall that, if $P$ is a projection operator on $\R^d$ and $x\in\Ima(P)$, then 
\[ \left( x=0 \right) \Leftrightarrow \left(\forall z\in\Ima(P^*), \langle x ,  z \rangle=0 \right)\]
because $|x|^2 = \langle x , x \rangle = \langle Px , x \rangle = \langle x , P^* x \rangle$.

As a consequence, the relation $\Pi^\p_N S(T;f_0,(0,u_2))=0$ is equivalent to
\begin{equation} \label{NC-p_HF}
\forall g_T \in \widetilde{F^\p_N}, \  
\left\langle  S(T;f_0,(0,u_2)) , g_T \right\rangle = 0
\end{equation}
where $\langle \cdot , \cdot \rangle$ is the scalar product of $L^2(\T,\C^d)$ and
\[\widetilde{F^\p_N} \coloneqq \bigoplus\limits_{n_0 < |n| \leq N}  \Ima\left(P^\p\Big(\frac \iu n{\Big)\!}^\ast\right) e_n.\]

For $g_T \in \widetilde{F^\p_N}$, we denote by $g(t)=\eu^{-\mathcal{L}^*(T-t)}g_T$ the solution of the adjoint system~\eqref{SystAdj}. Then, by \Cref{LemEquationg2}, $g=(g_1,g_2)$, where $g_1=G(g_2)$ and
\[\left\langle S(T;f_0,(0,u_2)) ,  g_T  \right\rangle = \left\langle f_0, g(0)  \right\rangle + \int_0^T \int_{\omega} \left\langle u_2(t,x) ,  g_2(t,x) \right\rangle \diff x \diff t.\]
where the first 2 scalar products are in $L^2(\T)^{d}$ and the last one is in $\C^{d_2}$.
By Corollary~\ref{th:G}, the assertion~\eqref{NC-p_HF} is equivalent to 
\[\forall g_{2}^T \in L^2(\T,\C^{d_2}), \  
\int_0^T \int_{\omega} \langle u_2(t,x) ,  g_2(t,x) \rangle \diff x \diff t = -
\left\langle f_0 ,  \left(G(g_2^0),g_2^0\right)  \right\rangle,\]
where $g_2(t)=\eu^{-\mathfrak{D}(T-t)} g_2^T$ and $g_2^0=g_2(0)$. 
By considering $g_2^T = X e_n$ with $X \in \C^{d_2}$ and $n_0 < |n|\leq N$, we obtain
\[g_2(t)=\eu^{-n^2(T-t)E_2(n)} X e_n \text{  and  } 
G(g_2^0)=G\left(\frac{\iu}{n}\right)\eu^{-n^2 T E_2(n)} X e_n.\]
The previous property is equivalent to
\begin{align*}
\forall n_0 < |n|\leq N,\  \forall X \in \C^{d_2}, \ 
 \int_0^T \int_{\omega} \langle u_2(t,x) , \eu^{-n^2(T-t)E_2(n)} X   \rangle \eu^{-\iu nx} \diff x \diff t \\
 = -
\langle f_{01} , G(\iu /n)\eu^{-n^2 T E_2(n)} X e_n  \rangle - \langle f_{02} , \eu^{-n^2 T E_2(n)} X e_n \rangle
\end{align*}
or, equivalently,
\begin{align*}
\forall n_0 < |n|\leq N,\ \forall X \in \C^{d_2}, \ 
\left\langle \int_0^T \int_{\omega}  \eu^{-n^2(T-t)E_2(n)^*} u_2(t,x) \eu^{-\iu nx} \diff x \diff t  ,  X   \right\rangle\\
 = -
\left\langle \eu^{-n^2 T E_2(n)^*} G(\iu /n)^* \widehat{f}_{01}(n) + \eu^{-n^2 T E_2(n)^*} \widehat{f}_{02}(n) , X \right\rangle
\end{align*}
which proves~\eqref{moment}.

\paragraph{Step 2: Solving the moment problem.} We look for a solution $u_2 \in C^\infty_0((0,T)\times \omega)$ of the moment problem~\eqref{moment} of the form
\begin{equation} \label{form_u2}
u_2(t,x)=\rho(t,x) v_2(t,x)
\end{equation}
where $v_2 \in C^\infty((0,T)\times\T)^{d_2}$ and $\rho \in C^\infty_0((0,T)\times\omega)$ is a scalar function with an appropriate support. More precisely, let 
\begin{itemize}
\item $\widehat{\omega}$ be an open subset such that $\widehat{\omega} \subset \subset \omega$ and $\rho_2 \in C^\infty_c(\omega,\mathbb{R}_+)$ such that $\rho_2=1$ on $\widehat{\omega}$,
\item $\rho_1 \in C^\infty([0,1],\R_+)$ such that $\rho_1(0)=\rho_1(1)=0$ and 
\begin{equation} \label{def:rho1}
\exists C_0 >0, \forall \gamma>0, \quad \int_0^1 \rho_1(\tau) \eu^{-\gamma \tau} \diff \tau \geq \frac{1}{C_0} \eu^{-C_0\ssqrt{\gamma}}.
\end{equation}
\end{itemize}

For instance, we may consider $\rho_1$ such that $\rho_1(\tau)=\rho_1(1-\tau)=\eu^{-\frac{1}{\tau}}$ for $\tau\in(0,1/4)$. Indeed, for every $\gamma>0$,
the change of variable $s=\ssqrt{\gamma} \tau$ gives
\[\int_0^1 \rho_1(\tau) \eu^{-\gamma \tau} \diff \tau \geq \frac1{\ssqrt\gamma}\int_{0}^{\ssqrt{\gamma}/4} \eu^{-\ssqrt{\gamma}\phi(s)} \diff s\]
where $\phi(s)=\frac{1}{s}+s$. The function $\phi$ takes its minimal value at $s_*=1$ and $\phi''(1)=2>0$ thus, by Laplace's method (see \cite[Chapitre 9, Théorème VI.1]{queffelec_2013}),
\[\int_{0}^{2} \eu^{-\ssqrt{\gamma}\phi(s)} \diff s \underset{\gamma \rightarrow \infty}{\sim} \frac{\sqrt{\pi}}{\sqrt[4]{\smash[b]\gamma}} \eu^{-2\ssqrt{\gamma}}.\]
which proves \eqref{def:rho1} for a large enough constant $C_0$.

Then we choose $\rho(t,x) = \rho_1((T-t)/T)\rho_2(x)$. We also look for $v_2$ of the form
\begin{equation} \label{form_v2}
v_2(t,x)=\sum\limits_{n_0<|k|\leq N} \eu^{-k^2(T-t)E_2(k)} V_k \eu^{\iu kx}\, \text{ where } V_k \in \C^{d_2}.
\end{equation}

The construction of $v_2$ will use the following algebraic result.
\begin{lemma} \label{Prop:ANinv}
There exists $\mathcal{C}>0$ such that, for every $N>n_0$ and $T\in(0,1]$
the matrix $A$ in $\C^{(2(N-n_0)d_2)\times(2(N-n_0)d_2)}$, defined by blocks
$A=(A_{n,k})_{\smash[b]{\substack{n_0 < |n|\leq N\\ n_0 < |k| \leq N}}}$ by
\[A_{n,k}=\int_0^T \int_{\omega} \eu^{-n^2(T-t)E_2(n)^*} \eu^{-k^2(T-t)E_2(k)} \eu^{\iu(k-n)x} \rho(t,x) \diff x \diff t\, \in \C^{d_2\times d_2},\]
is invertible and 
\[\forall F\in\C^{2(N-n_0)d_2}, \ |A^{-1} F|  \leq \frac{\mathcal{C}}{T} \eu^{\mathcal{C} N}  |F|,\]
where $|\cdot|$ is the hermitian norm on $\C^{2(N-n_0)d_2}$.
\end{lemma}

\begin{remark}
For instance, when $N=n_0+2$, then $A$ is given by 
\[A=\begin{pmatrix}
A_{-n_0-2,-n_0-2} & A_{-n_0-2,-n_0-1} & A_{-n_0-2,n_0+1} & A_{-n_0-2,n_0+2} \\
A_{-n_0-1,-n_0-2} & A_{-n_0-1,-n_0-1} & A_{-n_0-1,n_0+1} & A_{-n_0-1,n_0+2} \\
A_{n_0+1,-n_0-2} & A_{n_0+1,-n_0-1} & A_{n_0+1,n_0+1} & A_{n_0+1,n_0+2} 
\\
A_{n_0+2,-n_0-2} & A_{n_0+2,-n_0-1} & A_{n_0+2,n_0+1} & A_{n_0+2,n_0+2} 
\end{pmatrix}.\]
For $X \in \C^{4d_2}$ with block decomposition
\[X=\begin{pmatrix}
X_{-n_0-2} \\ X_{-n_0-1} \\ X_{n_0+1} \\ X_{n_0+2}
\end{pmatrix}\]
where $X_k \in \C^{d_2}$ for every $n_0 < |k| \leq n_0+2$, 
we have
\[AX=\begin{pmatrix}
\sum\limits_{n_0<|k|\leq n_0+2} A_{-n_0-2,k} X_{k} \\
\sum\limits_{n_0<|k|\leq n_0+2} A_{-n_0-1,k} X_{k} \\
\sum\limits_{n_0<|k|\leq n_0+2} A_{n_0+1,k} X_{k} \\
\sum\limits_{n_0<|k|\leq n_0+2} A_{n_0+2,k} X_{k} 
\end{pmatrix}.\]
Thus $\langle X , AX \rangle = \sum_{n_0 < |n|,|k|\leq n_0+2} X_n^* A_{n,k} X_k$.
\end{remark}

\begin{proof}[Proof of Lemma~\ref{Prop:ANinv}]
The proof relies on the following spectral inequality, due to Lebeau and Robbiano (see~\cite{lebeau_1995} and also~\cite[Thm. 5.4]{lerousseau_2012}): 
\begin{equation}\label{eq:spectral}
\exists C_1>0,\ \forall N \in \N,\ \forall (a_n)_{n\in\Z} \in \C^\Z, \ 
\sum_{n=-N}^{+N} |a_n|^2 \leq C_1 \eu^{C_1 N} \int_{\widehat{\omega}} \left|\sum_{n=-N}^{+N} a_n \eu^{\iu nx}\right|^2 \diff x.
\end{equation}
By summing the components, the same inequality holds when  $a_n$ is a vector, $a_n \in \C^{d_2}$, and $|\cdot|$ denotes the hermitian norm on $\C^{d_2}$.

Let $N>n_0$ and $X \in \C^{2(N-n_0)d_2}$ written by blocks $X=(X_k)_{n_0<|k|\leq N}$ with $X_k \in \C^{d_2}$. Then, by using the definition of $A$, $\rho$, the properties of $\rho_2$ and the above spectral inequality in vectorial form, we obtain
\begin{align*}
\langle AX,X\rangle
& = \sum\limits_{n_0 < |n|,|k|\leq N} X_n^* A_{n,k} X_k \\
& = \int_0^T \int_{\omega} \Big| \sum\limits_{n_0 < |k|\leq N} 
\eu^{-k^2(T-t)E_2(k)} X_k \eu^{\iu kx} \Big|^2 \rho(t,x) \diff x \diff t
\\ & \geq
 \int_0^T \int_{\widehat{\omega}} \Big| \sum\limits_{n_0 < |k|\leq N} 
\eu^{-k^2(T-t)E_2(k)} X_k \eu^{\iu kx} \Big|^2 \rho_1\left( \frac{T-t}{T} \right) \diff x \diff t
\\ & \geq \frac{\eu^{-C_1 N}}{C_1} \int_0^T \sum\limits_{n_0 < |k|\leq N}  \left| \eu^{-k^2(T-t)E_2(k)} X_k \right|^2  \rho_1 \left( \frac{T-t}{T} \right) \diff t.
\end{align*}
There exists $c>0$ such that, for every $|k|>n_0$, $|E_2(k)| \leq c$. Then, 
\[\forall |k| >n_0,\  \tau>0,\   Y\in\C^{d_2},\   | \eu^{E_2(k)\tau}Y| \leq  \eu^{c \tau} |Y|.\]
Then, by considering $\tau=k^2(T-t)$ and $Y=\exp \left(-k^2(T-t)E_2(k)\right) X_k$, we obtain
\[\forall |k| >n_0,\ t \in (0,T), \  
\big| \eu^{-k^2(T-t)E_2(k)} X_k \big| \geq \eu^{-ck^2(T-t)} |X_k| .\]
Therefore, by using the change of variable $\tau=\frac{T-t}{T}$ and~\eqref{def:rho1}, we get
\begin{align*}
\langle AX,X\rangle
 & \geq
\frac{T \eu^{-C_1 N}}{C_1}  \sum\limits_{n_0 < |k|\leq N}  | X_k |^2 \int_0^T \eu^{-2ck^2T \tau} \rho_1 \left( \tau \right) \diff \tau
\\ & \geq
\frac{T \eu^{-C_1 N}}{C_1 C_0}  \sum\limits_{n_0 < |k|\leq N}  | X_k |^2 \eu^{-C_0 k\sqrt{2cT}}
\\ & \geq
\frac{T}{C_1 C_0} \eu^{- (C_1+C_0 \sqrt{2cT})N}   | X |^2.
\end{align*}
The above relation, valid for any $X\in\C^{2(N-n_0)d_2}$ proves that any eigenvalue of $A$ is positive, thus $A$ is invertible.
Moreover, for any $F \in \C^{2(N-n_0)d_2} \setminus \{0\}$, the vector $X=A^{-1}F$ satisfies
\[\frac{T}{C_1 C_0} \eu^{- (C_1+C_0 \sqrt{2cT})N}   | X |^2 
\leq \langle AX , X \rangle = \langle F , X \rangle \leq |F| |X|.\]
Thus
\[|X| \leq \frac{C_1 C_0}{T} \eu^{ (C_1+C_0 \sqrt{2cT})N} |F|.\]
This gives the conclusion with
$\mathcal{C}=\max\left\{ C_1 C_0 ;\ C_1+C_0 \sqrt{2c}  \right\}$.
\end{proof}

Now, let us come back to the proof of Proposition~\ref{Prop:FNPF}.
For such a control of the form given by equations~\eqref{form_u2} and~\eqref{form_v2}, the moment problem~\eqref{moment} writes
\[\forall n_0 < |n|\leq N, \ 
\sum\limits_{n_0 < |k|\leq N} A_{n,k} V_k = F_n\]
or equivalently $AV=F$ with the notations of \Cref{Prop:ANinv}.
Thus, it is sufficient to take $V=A^{-1}F$.
By the definition of $F$ in~(\ref{moment}), and Bessel-Parseval identity there exists $C_2>0$ independent of $(T,N)$ such that
\[|F| = \Big( \sum_{n_0<|n|\leq N} |F_n|^2 \Big)^{1/2} \leq C_2 \|f_0\|_{L^2(\T)^d}.\] Thus, by \Cref{Prop:ANinv}
\begin{equation}\label{est_V}
|V| = \Big( \sum_{n_0<|k|\leq N} |V_k|^2 \Big)^{1/2} \leq  \frac{C_2 \mathcal{C} }{T} \eu^{\mathcal{C}N} \|f_0 \|_{L^2(\T)^d}.
\end{equation}

\paragraph{Step 3: Estimates on $u_2$.} Let $s\in\N^*$. By~\eqref{form_u2} and the definition of $\rho$,
there exists $C=C(\rho,s)>0$ such that 
\begin{equation} \label{estim:u2/v2}
\| u_2 \|_{H^s((0,T)\times\omega)} \leq \frac{C}{T^s} \|v_2\|_{H^s((0,T)\times\T)}.
\end{equation}
For any $s_1, s_2 \in \N$ such that $s_1+s_2 \leq s$ we have,
\[\partial_t^{s_1} \partial_x^{s_2} v_2 (t,x) = 
\sum\limits_{n_0<|k|\leq N} k^{2s_1} E_2(k)^{s_1} \eu^{-k^2(T-t)E_2(k)} V_k (\iu k)^{s_2} \eu^{\iu kx}.\]
By Bessel-Parseval identity, we have
\begin{align*}
\left\| \partial_t^{s_1} \partial_x^{s_2} v_2 \right\|_{L^2((0,T)\times\T)}^2
& = \int_0^T \sum\limits_{n_0<|k|\leq N} \left| k^{2s_1+s_2}  E_2(k)^{s_1} \eu^{-k^2(T-t)E_2(k)} V_k \right|^2 \diff t\\
& \leq  C  \int_0^T \sum\limits_{n_0<|k|\leq N} k^{4s} \left| \eu^{-k^2(T-t)E_2(k)} V_k\right|^2   dt
\end{align*}
By working as in the proof of \Cref{Prop:dissip_p}, we obtain, for $n_0$ large enough, positive constants $K_p, c_p>0$ such that
\begin{align*}
\left\| \partial_t^{s_1} \partial_x^{s_2} v_2 \right\|_{L^2((0,T)\times\T)}^2
& \leq C   \sum\limits_{n_0<|k|\leq N} k^{4s} K_p^2 \int_0^T  \eu^{-2c_pk^2(T-t)} \diff t\, |V_k|^2 \\
& \leq \frac{C K_p^2}{2c_p}  \sum\limits_{n_0<|k|\leq N}  k^{4s-2}  |V_k|^2 \leq \frac{C K_p^2}{c_p} N^{4s-1} |V|^2
\end{align*}
By~\eqref{est_V},
\[\left\|\partial_t^{s_1} \partial_x^{s_2} v_2 \right\|_{L^2((0,T)\times\T)}
\leq \sqrt{\frac{C}{c_p}} K_p N^{2s-1/2} \frac{C_2 \mathcal{C} }{T} \eu^{\mathcal{C}N} \|f_0 \|_{L^2(\T)^d}. \]
This provides a constant $C>0$ independant of $(T,N)$ such that
\[\left\| v_2 \right\|_{H^s((0,T)\times\T)}
\leq  \frac{C}{T} N^{2s-1/2} \eu^{\mathcal{C}N} \|f_0 \|_{L^2(\T)^d} \]
and~\eqref{estim:u2/v2} gives the expected estimate on $u$ in $H^s$.
\end{proof}

\subsubsection{Lebeau-Robbiano's method} \label{Subsec:LR}

The goal of this section is to prove \Cref{LemPar_Bis}. Let $T>0$. 
We fix $\delta \in (0,T/2)$ and $\rho \in (0,1)$.
For $\ell \in \set N^*$, we set $N_\ell=2^\ell$, $T_{\ell}=A2^{-\rho \ell}$ where $A>0$ is such that $2\sum_{\ell=1}^\infty T_\ell = T - 2\delta$.
Let $f_0 \in F^\p$. We define
\[\left\lbrace\begin{aligned}
& f_1=\eu^{-\delta \mathcal{L}^\p} f_0, \\
& g_\ell = \Pi^\p S( T_\ell ; f_\ell , u_\ell) \text{ where } u_\ell = \left(0,K_{T_\ell,N_\ell}(f_\ell)\right), \\
& f_{\ell+1}=\eu^{-T_\ell \mathcal{L}^\p} g_\ell,
\end{aligned}\right.\]
where $K_{T_\ell,N_\ell}$ is the control operator introduced in \Cref{Prop:FNPF}. By construction $\Pi^\p_{N_\ell} g_\ell=0$ and therefore, by \Cref{Prop:dissip_p}
\begin{multline*}
\|f_{\ell+1}\|_{L^2(\T)^d}^2 = \| \eu^{-T_\ell \mathcal{L}^\p} g_\ell \|_{L^2(\T)^d}^2
= \sum_{|n|>N_\ell} \left\lvert  \eu^{-n^2 E(\iu/n)T_\ell}  \widehat{g}_\ell(n) \right\rvert^2 \\
 \leq \sum_{|n|>N_\ell} K_p^2 \eu^{-2 n^2 c_p T_\ell} \left\lvert  \widehat{g}_\ell(n) \right\rvert^2
 \leq K_p^2 \eu^{-2 c_p N_\ell^2 T_\ell} \|g_\ell\|_{L^2(\T)^d}^2.
\end{multline*}
By the semi-group property proved in \Cref{Prop:C0_sg}, there exists positive constants $K$ and $c$ such that
\[\forall f \in L^2(\T)^d, t\geq 0 \quad \| \eu^{-t\mathcal{L}} f \|_{L^2(\T)^d}\leq K \eu^{ct} \|f\|_{L^2(\T)^d}.\]
Then, according to the triangle inequality and Cauchy-Schwarz inequality,
\begin{align*}
\|g_\ell\|_{L^2(\T)^d} 
\leq \| S( T_\ell ; f_\ell , u_\ell) \|
& \leq K \eu^{cT_\ell} \|f_\ell\|_{L^2(\T)^d} + \int_0^{T_\ell} K \eu^{c(T_\ell-t)} \|u_\ell(t)\|_{L^2(\T)} \diff t \\
& \leq K \eu^{cT_\ell} \left( \|f_\ell\|_{L^2(\T)^d} +  \ssqrt{T_\ell} \|u_\ell\|_{L^2((0,T_\ell)\times\omega)} \right),
\end{align*}
and by \Cref{Prop:FNPF}
\begin{equation}\label{Control:EstimationL2}\| u_\ell \|_{L^2((0,T_\ell)\times\omega)} \leq \frac{\mathcal{C}}{T_\ell} \eu^{\mathcal{C} N_\ell} \|f_\ell\|_{L^2(\T)^d}.\end{equation}
Thus
\[\|g_\ell\|_{L^2(\T)^d} \leq K \eu^{cT_\ell} \left( 1 +  \frac{\mathcal{C}}{\sqrt{T_\ell}} \eu^{\mathcal{C} N_\ell}\right)
\|f_\ell\|_{L^2(\T)^d}.\]
By setting 
\[m_\ell = K_p \eu^{- c_p N_\ell^2 T_\ell}  K \eu^{cT_\ell} \left( 1 +  \frac{\mathcal{C}}{\sqrt{T_\ell}} \eu^{\mathcal{C} N_\ell}\right),\]
we get
\[\|f_{\ell+1}\|_{L^2(\T)^d} \leq m_\ell \|f_\ell \|_{L^2(\T)^d}.\]
It is easy to see that there exists $C_1, C_2>0$ such that $m_\ell \leq C_1 \eu^{-C_2 2^{(2-\rho)\ell}}$. Thus $\|f_\ell\|_{L^2(\T)^d} \rightarrow 0$ and more precisely there exists positive constants $C_3, C_4>0$ such that
\[\|f_\ell\|_{L^2(\T)^d} \leq C_3 \exp \left( -C_4 2^{(2-\rho)\ell} \right) \|f_0\|_{L^2(\T)^d}.\]
Moreover, from \eqref{Control:EstimationL2},
\begin{equation}\label{controle_L2}
\sum_{\ell=1}^\infty \| u_\ell \|_{L^2((0,T_\ell)\times\omega)}^2 \leq
\mathcal{C} \sum_{\ell=1}^\infty \frac{\eu^{\mathcal{C}N_\ell}}{T_\ell} C_3 \exp(-C_4 2^{(2-\rho)\ell}) \|f_0\|_{L^2(\T)^d} <\infty.
\end{equation}

We set $a_0=\delta$, $a_2=\delta+2T_1$, \ldots , $a_\ell=a_{\ell-1}+2T_\ell$. We have $a_\ell \rightarrow (T-\delta)$ as $\ell \rightarrow \infty$. Then, for any $f_0 \in F^\p$, we define the control
\[\underline{\mathcal{U}}_T^{\p}(f_0)(t,x)=\left\lbrace\begin{array}{ll}
 K_{T_\ell,N_\ell}(f_\ell)(t-a_{\ell-1})\quad & \text{ for } a_{\ell-1}\leq t \leq a_{\ell-1}+T_\ell, \\
 0 & \text{ for } a_{\ell-1}+T_\ell\leq t \leq a_{\ell-1}+2T_\ell=a_{\ell}, 
 \\ 
 0 & \text{ for } T-\delta \leq t \leq T.
 \end{array}\right.\]
Then, $\underline{\mathcal{U}}_T^{\p}(f_0) \in C^\infty_0((\delta,T-\delta)\times \omega)^{d_2}$ because all its derivatives vanish at times $t=a_\ell$. Thus $\underline{\mathcal{U}}_T^{\p}(f_0) \in C^\infty_c((0,T)\times\omega)^{d_2}$.

By~\eqref{controle_L2}, $\underline{\mathcal{U}}_T^{\p}(f_0) \in L^2((0,T)\times\omega)^d$ thus
$S(T-\delta;f_0,\underline{\mathcal{U}}_T^{\p}(f_0))$ is the limit,  in $L^2(\T)^d$, as $\ell \rightarrow \infty$, of the sequence $S(a_\ell;f_0,\underline{\mathcal{U}}_T^{\p}(f_0))$. As a consequence,
$\Pi^\p S(T-\delta;f_0,\underline{\mathcal{U}}_T^{\p}(f_0))$ is the limit in $L^2(\T)$ of the sequence $\Pi^\p S(a_\ell;f_0,\underline{\mathcal{U}}_T^{\p}(f_0))=f_{\ell+1}$.
Finally,
\[\Pi^\p S(T;f_0,\underline{\mathcal{U}}_T^{\p}(f_0))=\Pi^\p S(T-\delta;f_0,\underline{\mathcal{U}}_T^{\p}(f_0))=0.\]
By \Cref{Prop:FNPF}, for any $s\in\mathbb{N}^*$, 
\[\left\| \underline{\mathcal{U}}_T^{\p}(f_0) \right\|_{H^s((0,T)\times\omega)}
\leq \underbrace{\sum\limits_{\ell=1}^\infty \frac{\mathcal{C}}{T_\ell^{s+1}} N_\ell^{2s} \eu^{\mathcal{C} N_\ell} C_3 \exp \left( -C_4 2^{(2-\rho)\ell} \right)}_{< \infty}
\|f_0\|_{L^2(\T)^d} .\]
This concludes the proof of \Cref{LemPar_Bis}. \qed

\subsection{Control of the low frequencies}
\label{ControlLowFreqSec}

The goal of this subsection is to prove \Cref{th:main}. 
Let $T > T^{*}$ where $T^{*}$ is defined in~\eqref{eq:T}. Then, there exists $T'>0$ such that \eqref{defTT'} holds. Let $\mathcal{G}$ and $\mathcal{U}$ be as in Proposition \ref{PropControlHighFreq}. 

Without loss of generality, we may assume that $F_0\subset \mathcal G$ by the following procedure. Let $W$ be a complement of $\mathcal G\cap F^0$ in $F^0$. Then $W$ is a complement of $\mathcal G$ in\footnote{If $f\in \mathcal G+F^0$, we write it as $f_{\mathcal G} + f_{F_0}$, and in turn we decompose $f_{F_0}$ along the sum $F_0 =\mathcal G\cap F_0\oplus W$: $f_{F_0} =  f_{\mathcal G\cap F_0}+f_W\in \mathcal G+W$. So $f = (f_{\mathcal G} + f_{ \mathcal G \cap F_0}) + f_{W}$. This proves that $\mathcal G+F_0 = \mathcal G+W$. Moreover, if $f\in \mathcal G\cap W$, since $W\subset F_0$, we have $f\in \mathcal G\cap F_0\cap W$, which is $\{0\}$. So the sum $\mathcal G+W$ is direct.} $\mathcal G+F^0$, and we extend $\mathcal U$ to $\mathcal G\oplus W$ by setting $\mathcal U(f_0) = 0$ for every $f_0\in W$.

Implicitly, $\mathcal{G}$ is equipped with the topology of the $L^2(\T)^d$-norm. The operator $S$ is defined in Definition~\ref{def:solution}.

We introduce the vector subspace of $L^2(\T)^d$ defined by
\begin{align*}
\mathcal{F}_{T} = \big\{ f_0 \in L^2(\T)^{d} ;\ & \exists u \in L^2((0,T')\times\omega)^{d_1} \times C^\infty_c ((T',T)\times\omega)^{d_2} / S(T;f_0, u)=0
\big\}.
\end{align*}

\paragraph{Step 1: We prove that $\mathcal{F}_{T}$ is a closed subspace of $L^2(\T)^d$ with finite codimension.} 

For $f_0 \in \mathcal{G}$, 
the function $S(T;f_0,\mathcal{U}f_0)$ belongs to $F^0$, 
thus 
\begin{equation} \label{def:Kf00}
\mathcal{K}(f_0)\coloneqq -\eu^{T \mathcal{L}^0} S(T;f_0,\mathcal{U}f_0)
\end{equation}
is well defined in $F^0$ by \Cref{ExtendSemigroups}. Then, $\mathcal{K}$ is a compact operator on $\mathcal{G}$ because it has finite rank.
By the Fredholm alternative, 
$(I+\mathcal{K})(\mathcal{G})$ is a closed subspace of $\mathcal{G}$ and
there exists a closed subspace $\mathcal{G}'$ of $\mathcal{G}$, with finite codimension in $\mathcal{G}$, such that $(I+\mathcal{K})$ is a bijection from $\mathcal{G}'$ to $(I+\mathcal{K})(\mathcal{G})$.
Note that $\mathcal{G}'$ is also a closed subspace with finite codimension in $L^2(\T)^d$.

For any  $f_0 \in \mathcal{G}'$, by using that $\mathcal{K}(f_0) \in F^0$ and (\ref{def:Kf00}), we obtain
\[S\left(T,\mathcal{K}(f_0),0\right)=
\eu^{-T\mathcal{L}} \mathcal{K}(f_0) = 
\eu^{-T\mathcal{L}^0} \mathcal{K}(f_0) = 
- S\left(T,f_0,\mathcal{U} f_0 \right)\]
thus
\[S\left(T,f_0+\mathcal{K}(f_0),\mathcal{U} f_0\right)=S\left(T,f_0,\mathcal{U} f_0 \right)+ S\left(T,\mathcal{K}(f_0),0\right)=0. \]

This proves that $\mathcal{F}_{T}$ contains $(I+\mathcal{K})(\mathcal{G}')$,
which is a closed subspace with finite codimension in $L^2(\T)^d$. Therefore, there exists a finite dimensional subspace $F_\sharp$ of $L^2(\T)^d$ such that 
$\mathcal{F}_{T} = (I+\mathcal{K})(\mathcal{G}') \oplus F_\sharp$. This gives the conclusion of Step 1.

\paragraph{Step 2: We prove that, up to a possibly smaller choice of $T>T^\ast$, there exists $\delta>0$ such that $\mathcal{F}_{T'}=\mathcal{F}_{T}$ for every $T'\in[T,T+\delta]$.}
When $0<T'<T''$, by extending controls defined on $(0,T')$ by zero on $(T',T'')$, we see that $\mathcal{F}_{T'} \subset \mathcal{F}_{T''}$. Thus, the map $T' \mapsto \codim(\mathcal{F}_{T'})$ is decreasing and takes integer values. As a consequence the discontinuities on $(T^\ast,T+1]$ are isolated. If $T$ is not such a discontinuity point, then there exists $\delta>0$ such that $\codim(\mathcal{F}_{T'})=\codim(\mathcal{F}_{T})$ for every $T'\in[T,T+\delta]$. In case $T$ is such a discontinuity point, one may replace $T$ by a smaller value, still such that $T>T^\ast$, for which this holds.

\paragraph{Step 3: We prove that $\left( \eu^{-t {\mathcal{L}^*}} \mathcal{F}_{T}^{\perp} \right)^\perp \subset \mathcal{F}_{T}$ for every $t\in(0,\delta)$.} Let $t \in (0,\delta)$ and $g_0 \in L^2(\T)^d$ be such that $\langle g_0 , \eu^{-t {\mathcal{L}^*}} f_0 \rangle = 0$ for every $f_0 \in \mathcal{F}_{T}^{\perp}$. Then $\langle \eu^{-t {\mathcal{L}}} g_0 ,  f_0 \rangle = 0$ for every $f_0 \in \mathcal{F}_{T}^{\perp}$, i.e.
$\eu^{-t {\mathcal{L}}} g_0 \in (\mathcal{F}_{T}^{\perp})^\perp$.
By Step 1, $\mathcal{F}_{T}$ is a closed subspace of $L^2(\T)^d$ thus $(\mathcal{F}_{T}^{\perp})^\perp = \mathcal{F}_{T}$. Therefore $\eu^{-t {\mathcal{L}}} g_0 \in \mathcal{F}_{T}$. By definition of $\mathcal{F}_T$, this implies that $g_0 \in \mathcal{F}_{T+t}$. By Step 2, we get $g_0 \in \mathcal{F}_T$, which ends the proof of Step 3.

\paragraph{Step 4: We prove that $\mathcal{F}_{T}^{\perp}$ is left invariant by $\eu^{-t {\mathcal{L}^*}}$, i.e.\ $\mathcal{F}_{T}^{\perp}=\eu^{-t {\mathcal{L}^*}} \mathcal{F}_{T}^{\perp}$ for every $t>0$.}
The subspace $\eu^{-t {\mathcal{L}^*}} \mathcal{F}_{T}^{\perp}$ is closed in $L^2(\T)^d$ because it has finite dimension. Thus 
$\left( \left( \eu^{-t {\mathcal{L}^*}} \mathcal{F}_{T}^{\perp} \right)^\perp \right)^\perp=\eu^{-t {\mathcal{L}^*}} \mathcal{F}_{T}^{\perp}$
 and we deduce from Step 3 that, for every $t\in(0,\delta)$, $\mathcal{F}_{T}^\perp \subset \eu^{-t {\mathcal{L}^*}} \mathcal{F}_{T}^{\perp}$.
Taking into account that $\dim( \eu^{-t {\mathcal{L}^*}} \mathcal{F}_{T}^{\perp} ) \leq \dim(\mathcal{F}_{T}^{\perp})$, we obtain $\mathcal{F}_{T}^\perp = \eu^{-t {\mathcal{L}^*}} \mathcal{F}_{T}^{\perp}$ for every $t\in(0,\delta)$. By the semi-group property, this equality holds for every $t >0$.

\paragraph{Step 5: We prove the existence of $N\in\N$ such that any 
$f_0 \in \mathcal{F}_{T}^{\perp}$ can be written}
\begin{equation}
f_0  = \sum_{k \leq N} \varphi_k e_k \quad \text{ with } \varphi_k \in \C^d.
\label{sommewfin}
\end{equation}
Let ${S}(t)^*$ be the restriction of the semigroup $\eu^{t {\mathcal{L}^*}}$ to $\mathcal{F}_{T}^{\perp}$, i.e.\ ${S}(t)^* = \eu^{-t {\mathcal{L}^*}}|_{\mathcal{F}_{T}^{\perp}}$. Then ${S}(t)^*=\eu^{tM}$ where $M$ is a matrix such that ${\mathcal{L}^*} f_0 = M f_0$ for every $f_0 \in \mathcal{F}_{T}^{\perp}$. But then $\ker(M-\overline{\lambda})^{j} = \ker({\mathcal{L}^*}-\overline{\lambda})^{j} \cap \mathcal{F}_{T}^\perp$. The Kernel decomposition theorem applied to $M$, and the structure of the generalized eigenspaces of $\mathcal{L}^*$ gives the conclusion of Step 4.

\paragraph{Step 6: We prove that any element of $L^2(\T)^{d}$ can be steered to $\mathcal{F}_{T}$ in an arbitrary short time}, i.e.\ for every $\varepsilon >0$ and $f_0 \in L^2(\T)^{d}$, there exists $u \in L^2((0,T')\times\omega)^{d_1} \times C^\infty_c ((T',T)\times\omega)^{d_2}$ such that $S(\varepsilon;f_0, u) \in \mathcal{F}_{T}$.
By the Hilbert Uniqueness Method, it is sufficient to prove an observability inequality for ${S}(t)^*$. By using the finite-dimensionality of $\mathcal{F}_{T}^{\perp}$, it is equivalent to prove that the following unique continuation property holds: if $f(t,\cdot)=\eu^{tM} f_0$ with $f=0$ in $(0,\varepsilon)\times \omega$, then $f_0 = 0$. By using the spectral inequality of Lebeau-Robbiano, i.e.~\eqref{eq:spectral} and~\eqref{sommewfin}, we readily get the result. 

\paragraph{Step 7: Conclusion.} Step 5 implies the controllability of the system in  any time $\tau>T$. As $T$ is an arbitrary time such that $T>T^\ast$, this concludes the null-controllability in any time $T>T^\ast$. \hfill\qed
\medskip

By a duality argument, we obtain the following result, that will be used in the next sections.

\begin{corollary}
\label{CorHUM}
For every $T>T^*$ and $s\in\mathbb{N}$, there exists $C_{T,s}>0$ such that,
for every $g_0 \in L^2(\T)^d$ the solution $g(t)=\eu^{-t\mathcal{L}^*} g_0=(g_1,g_2)(t)$ of the adjoint system \eqref{SystAdj} satisfies
\[\|g(T)\|_{L^2(\T)^d} \leq C_{T,s} \left( \|g_1\|_{L^2(q_T)^{d_1}} + \|g_2\|_{H^{-s}(q_T)^{d_2}} \right),\]
where $q_T=(0,T)\times\omega$.
\end{corollary}
We will use the following standard lemma that gives a canonical isometry between $H^{-s}(\Omega)$ and $H_0^s(\Omega)$.
\begin{lemma}\label{th:lemma_dual_Hs}
Let $\Omega$ be an open subset of $\R^d$ or a compact manifold (possibly with boundary). Let $s\geq 0$ and $\iota_s\colon H^s_0(\Omega)\to L^2(\Omega)$ be the inclusion map.\puncfootnote{We recall that $H^s_0(\Omega)$ is the closure of $C_c^\infty(\Omega)$ for the $H^s$-norm, and that $H^{-s}(\Omega)$ is the dual of $H^s_0(\Omega)$ with respect to the pivot space $L^2(\Omega)$.} The map $\iota_s^*\colon L^2(\Omega) \to H^s_0(\Omega)$ extends to a bijective isometry from $H^{s}_0(\Omega)$ to $H^{-s}(\Omega)$. 
\end{lemma}
\begin{proof}
The map $\iota_s^*$ is defined on $L^2(\Omega)$ by 
\begin{equation}\label{eq:def_iota*}
\forall f\in L^2(\Omega),\ \forall v\in H^s_0(\Omega),\ \langle \iota_s^*f,v\rangle_{H^s_0} = \langle f,v\rangle_{L^2}.\end{equation}
Thus, for evey $f\in L^2(\Omega)$,
\[
|\iota_s^*f|_{H^s_0} = \sup_{|v|_{H^s_0} = 1} \langle \iota_s^*f,v\rangle_{H^s_0} = \sup_{|v|_{H^s_0}=1} \langle f,v\rangle_{L^2} = |f|_{H^{-s}},\]
where we used the definition of $H^{-s}(\Omega)$ as the dual of $H^s_0(\Omega)$ with respect to the pivot space $L^2(\Omega)$ (see for instance~\cite[Sec.~2.9]{tucsnak_2009}). Since $L^2(\Omega)$ is dense in $H^{-s}(\Omega)^d$, this proves that $\iota_s^*$ extends by continuity to $H^{-s}(\Omega)$.

This extension is an isometry from $H^{-s}(\Omega)$ onto its range. As such it is injective and its range is closed. To prove it is bijective, we check that its range is dense, i.e.\ that its orthogonal is zero.

If $g_0\in H^s_0(\Omega)$ is orthogonal to $\Ima(\iota_s^*)$, then, according to the definition of $\iota_s^*$ (Eq.~\eqref{eq:def_iota*}) $g_0$ is orthogonal in $L^2(\Omega)$ to $H^s_0(\Omega)$. But $H^s_0(\Omega)$ is dense in $L^2(\Omega)$, so $g = 0$. Thus $\Ima(\iota_s^*)^\perp = \{0\}$.
\end{proof}

\begin{proof}[Proof of \Cref{CorHUM}]
We apply the duality \Cref{LemmaDuality} with
\[ \Phi_2 \colon f_0 \in L^2(\T)^d \mapsto f(T,\cdot) \in L^2(\T)^d,\]
where $f$ is the solution to the system \eqref{Syst} with initial data $f_0$ and control $u=0$, and
\[ \Phi_3 \colon u=(u_1,u_2)  \in L^2(q_T)^{d_1} \times H_0^{s}(q_T)^{d_2} \mapsto f(T,\cdot) \in L^2(\T)^d,\]
where $f$ is the solution to the system \eqref{Syst} with initial data $f_0=0$ and control $u$. 
The null-controllability result proved above is equivalent to the inclusion $\Ima(\Phi_2)\subset \Ima(\Phi_3)$, thus to the existence of $C>0$ such that for every $g_T\in L^2(\T)^d$,
\begin{equation}\label{eq:obs_lemma2}
\lVert\Phi_2^*(g_T)\rVert_{L^2(\T)^d}\leq C\lVert \Phi_3^*(g_T)\rVert_{L^2(q_T)^{d_1}\times H^s_0(q_T)^{d_2}}.
\end{equation}

We compute the adjoint operators of $\Phi_2$ and $\Phi_3$ thanks to the duality relation between the solution $f$ of \eqref{Syst} and the solution $\varphi(\cdot)= g(T-\cdot)$ of the adjoint system~\eqref{SystAdj}:
\begin{align} \label{dualite}
& \langle f(T) , \varphi(T) \rangle_{L^2(\T)^d}  = \langle f(0), \varphi(0) \rangle_{L^2(\T)^d} + \int_{0}^{T}\int_{\omega} \langle u(t,x) , \varphi(t,x) \rangle \diff t \diff x \\
= & \langle f(0), \varphi(0) \rangle_{L^2(\T)^d} + \int_{0}^{T}\int_{\omega} \langle u_1(t,x) , \varphi_1(t,x) \rangle + \langle u_2(t,x) , \varphi_2(t,x) \rangle \diff t \diff x.
\end{align}
First, we have $\Phi_2^*(g_T) = (\eu^{-T\mathcal L})^* g_T = \eu^{-T\mathcal L^*}g_T$.
To compute $\Phi_3^*$, we introduce the input-output operator $\mathcal F_T \colon u \in L^2(q_T)^d \mapsto f(T,.) \in L^2(\T)^d$, where $f$ is the solution of~\eqref{Syst} with initial condition $f_0 = 0$ and right-hand side $u$. By (\ref{dualite}), $\mathcal F_T^*(g_T)$ is the restriction of $\eu^{(t-T)\mathcal L^*}g_T$  to $[0,T]\times \omega$. 
We have $\Phi_3 = \mathcal F_T \circ (I,\iota_s)$, where $(I,\iota_s)$ stands for the inclusion map $L^2(q_T)^{d_1}\times H^s_0(q_T)^{d_2}\to L^2(q_T)^d$. Thus, according to \Cref{th:lemma_dual_Hs}, the right-hand side of the inequality~\eqref{eq:obs_lemma2} is
\[
\lVert (I,\iota_s^*) \circ \mathcal F_T^* (g_T)\rVert_{L^2(q_T)^{d_1}\times H^s_0(q_T)^{d_2}} = \lVert \mathcal F_T^* (g_T)\rVert_{L^2(q_T)^{d_1}\times H^{-s}(q_T)^{d_2}},
\]
which gives the conclusion.
\end{proof}

\section{Hyperbolic control: coupling of order zero}
\label{sec:cont_parab_simple}

The goal of this section is to prove the following result on the system
\begin{equation} \label{Syst_bloc_hyp_K}
\mkern-3mu\left\{\!
\begin{array}{ll}
\left( \partial_t +A'\partial_x+K_{11}\right)f_1+\left(A_{12}\partial_x+K_{12}\right)f_2= u_1  1_\omega
&\text{in}\ (0,T)\times \T,\!\\
\left( \partial_t -\partial_x^2+K_{22}\right)f_2 + K_{21} f_1 = 0 &\text{in}\ (0,T)\times \T.\!\\
\end{array} \right.
\end{equation}

\begin{theorem} \label{th:main_2_facile}
We assume~\eqref{dd1d2}--\eqref{h:A1}, $D=I_{d_2}$ $m=d_1$, $M_1=I_{d_1}$, $M_2=0$, $A_{21}=0$ and $A_{22}=0$. Let $T^\ast$ be defined by~\eqref{eq:T}.
The following statements are equivalent.
\begin{itemize}
\item The system (\ref{Syst_bloc_hyp_K}) is null controllable in any time $T>T^{\ast}$.
\item The couple of matrices $(K_{22},K_{21})$ satisfies the Kalman rank condition:
\begin{equation} \label{Kalman_Thm2_facile}
\Span \{ K_{22}^j K_{21} X_1 ; X_1 \in \C^{d_1}, 0 \leq j \leq d_2-1 \} = \C^{d_2}.
\end{equation}
\end{itemize}
\end{theorem}

The interest of this theorem is that its proof is essentially the same as the one of Theorems \ref{th:main_2} and \ref{th:main_3} (that will be done in the next sections) but it is less technical. 

In \Cref{SectionRequirementKalman}, we prove that the Kalman condition (\ref{Kalman_Thm2_facile}) is necessary for the null controllability of System (\ref{Syst_bloc_hyp_K}). In \Cref{subsec:K_S_K}, we prove that the Kalman condition (\ref{Kalman_Thm2_facile}) is sufficient for the null controllability of System (\ref{Syst_bloc_hyp_K}), first in the case $d_1=1$ (i.e.\ with one hyperbolic line in the system) where the cascade structure (or Brunovski form) is easy to handle, then in the general case $d_1>1$ which is more delicate to write.

\subsection{The Kalman condition is necessary}
\label{SectionRequirementKalman}

If the null controllability property for (\ref{Syst_bloc_hyp_K}) holds, then, by
considering the Fourier components of the solution and the control, we obtain the null controllability, for every $n\in\Z\setminus\{0\}$, of the system
\[\left\lbrace\begin{aligned}
& X_1(t)'+ (\iu nA'+K_{11})X_1(t) + (\iu nA_{12}+K_{12})X_1(t) = v_1(t), \\
& X_2'(t)+(n^2 I_{d_2}+K_{22})X_2(t)+ K_{21} X_1(t)=0,
\end{aligned}\right.\]
with state $X(t)=(X_1,X_2)(t)\in\C^{d_1}\times\C^{d_2}$ and control $v_1 \in L^2(0,T)^{d_1}$. This requires the null controllability of the control system
\[X_2'(t)+(n^2 I_{d_2}+ K_{22})X_2(t) + K_{21} X_1(t)=0,\]
with state $X_2(t) \in \C^{d_2}$ and control $X_1 \in L^2(0,T)^{d_1}$, i.e.\ the Kalman rank condition (see for instance \cite[Thm.~1.16]{coron_2007})
\[\Span\{ (n^2 I_{d_2}+K_{22})^j K_{21} v_1 ; v_1 \in \C^{d_1},\ j \in \{0,\dots,d_2-1\}\}=\C^{d_2},\]
that can equivalently be written in the form (\ref{Kalman_Thm2_facile}).

\subsection{The Kalman condition is sufficient}
\label{subsec:K_S_K}

In this section, we explain how to complete the proof of Theorem~\ref{th:main} to prove that the Kalman rank condition \eqref{Kalman_Thm2_facile} implies the null controllability of \eqref{Syst_bloc_hyp_K} in time $T>T^\ast$, in \Cref{th:main_2_facile}.

First, we treat the case $d_1 = 1$ then we generalize to the case $d_1 >1$.

From now and until end of this subsection, $C$ will denote positive constants which will vary from line to line. For $1 \leq i \leq 2$ and $1 \leq j \leq d_i$, we denote by $v_{i}^j$ the $j$-th component of a vector $v_i \in \C^{d_i}$.

\subsubsection{The case of one hyperbolic component: $d_1 = 1$}

By using the Hamilton-Cayley's theorem, we know that there exist $c_0, \dots, c_{d_2-1} \in \R$ such that
\begin{equation}
\label{HamiltonCayley}
K_{22}^{d_2} = c_0 I_{d_2} + c_1 K_{22} + \dots + c_{d_2-1} K_{22}^{d_2-1}.
\end{equation}
By using the Kalman condition \eqref{Kalman_Thm2_facile}, the matrix $P$ defined as follows
\begin{equation}
\label{defPK}
 P \coloneqq (K_{21}, K_{22} K_{21}, \dots, K_{22}^{d_2-1} K_{21}),
 \end{equation}
is invertible. We set
\begin{equation}
\label{eq:cascadeK}
 \widehat{K_{22}}\coloneqq  \begin{pmatrix} 0 & \dots & \dots & 0 & c_0\\
1 & 0 &\dots & \vdots & c_1\\
0 & \ddots & \ddots & \vdots & c_2\\
\vdots & \ddots & \ddots & 0 & \vdots\\
0 & \dots &0 & 1 & c_{d_2-1}
\end{pmatrix}\ \text{and}\ \widehat{K_{21}}\coloneqq \begin{pmatrix} 1\\0\\ \vdots\\ 0 \end{pmatrix}.
\end{equation}
From \eqref{HamiltonCayley}, \eqref{defPK}, \eqref{eq:cascadeK}, we check that we have the following relations
\[K_{22}P = P \widehat{K_{22}}\ \text{and}\  K_{21} = P \widehat{K_{21}},\ \text{i.e.}\ \widehat{K_{22}}= P^{-1} K_{22} P\ \text{and}\ {\widehat K}_{21}= P^{-1} K_{21}.\]
\indent The function $w =(w_1,w_2) = (f_1,P^{-1}f_2)$ solves 
\begin{equation} \label{Syst_bloc_hyp_K_w}
\mkern-3mu\left\{\!
\begin{array}{ll}
\left( \partial_t +A'\partial_x+K_{11}\right)w_1+\left(A_{12}P\partial_x+K_{12}P\right)w_2= u_1  1_\omega
&\text{in}\ (0,T)\times \T,\!\\
\left( \partial_t - \partial_x^2+\widehat{K_{22}}\right)w_2 + \widehat{K_{21}}w_1 = 0 &\mathrm{in}\ (0,T)\times \T,\!\\
(w_1,w_2)(0,\cdot)=(w_{01},w_{02}) & \text{in}\ \T.
\end{array} \right.
\end{equation}

The system \eqref{Syst_bloc_hyp_K_w} is a ``cascade system''. Indeed, roughly speaking the control $u_1$ directly controls the component $w_1$, the component $w_1$ indirectly controls the component $w_{2}^1$ in the second equation through the coupling term $w_1$, the component $w_{2}^1$ indirectly controls the component $w_{2}^2$ in the third equation through the coupling term $w_2^1$, \ldots\ the component $w_{2}^{d_2-1}$ indirectly controls the component $w_{2}^{d_2}$ in the last equation through the coupling term $w_2^{d_2-1}$.

The adjoint system of \eqref{Syst_bloc_hyp_K_w} is
\begin{equation} \label{Syst_bloc_hyp_K_Adj}
\mkern-3mu\left\{\!
\begin{array}{ll}
\left( \partial_t -\Tr{A'}\partial_x+\Tr K_{11}\right)g_1 +\Tr{\widehat{K_{21}}} g_2= 0
&\text{in}\ (0,T)\times \T,\!\\
\left( \partial_t - \partial_x^2+\Tr{\widehat{K_{22}}}\right)g_2 +\left(-\Tr{(A_{12}P)}\partial_x+\Tr{(K_{12}P)}\right)g_1  = 0 &\mathrm{in}\ (0,T)\times \T,\!\\
(g_1,g_2)(0,\cdot)=(g_{01},g_{02}) & \text{in}\ \T.
\end{array} \right.
\end{equation}
From \Cref{CorHUM}, we know that for every $g_0 \in L^2(\T)^d$, the solution $g$ of \eqref{Syst_bloc_hyp_K_Adj} satisfies
\begin{equation} \label{IO_H_Obs1K}
\| g(T,\cdot) \|_{L^2(\T)} \leq C \left( \| g_1\|_{L^2(q_T)} + \|g_2\|_{H^{-2d_2+1}(q_T)} \right) .
\end{equation}
By using the fact that ${\widehat K}_{22}$ is a companion matrix, see \eqref{eq:cascadeK}, we have that for every $i \in \{2, \dots, d_2\}$, the $i$-th equation of \eqref{Syst_bloc_hyp_K_Adj} is
\begin{equation*}
 \partial_t g_2^{i-1} - \partial_x^2 g_2^{i-1} + g_2^{i} + b_{i-1} \partial_x g_1+ a_{i-1} g_1 =0,\ \text{with}\ (a_{i-1}, b_{i-1}) \in \R^2
 \end{equation*}
Then we deduce
\begin{equation}
\|g_2^{i}\|_{H^{-2i + 1 }(q_T)} \leq C \left(\|g_2^{i-1}\|_{H^{-2(i-1)+1}(q_T)} + \| g_1\|_{L^2(q_T)} \right).
\label{Obs_CascadeK}
\end{equation}
Here, we have used in particular that
\[ \| (\partial_t  - \partial_x^{2} )g_2^{i-1}\|_{H^{-2 i +1 }(q_T)} \leq C \|g_2^{i-1}\|_{H^{-2(i-1)+1}(q_T)}\]
and 
\[ \| b_{i-1} \partial_x g_1+ a_{i-1} g_1 \| _{H^{-2 i +1 }(q_T)} \leq C \|g_1\|_{L^2(q_T)}.\]
Then, we deduce from \eqref{IO_H_Obs1K} and \eqref{Obs_CascadeK} that 
\begin{equation} \label{IO_H_Obs2K}
\| g(T,\cdot) \|_{L^2(\T)} \leq C \left( \| g_1\|_{L^2(q_T)} + \|g_2^1\|_{H^{-1}(q_T)} \right) .
\end{equation}
By using the fact that ${\widehat K}_{21}$ is the first vector of the canonical basis of $\R^{d_2}$, see \eqref{eq:cascadeK}, the first equation of \eqref{Syst_bloc_hyp_K_Adj} is
\[ \partial_t g_1 - A' \partial_{x} g_1 + K_{11} g_1 + g_2^{1} = 0.\]
Then, we obtain that
\begin{equation}
\label{Obs_FirstVectorK}
\|g_2^1\|_{H^{-1}(q_T)} \leq C \| g_1\|_{L^2(q_T)}.
\end{equation}
So, we deduce from \eqref{IO_H_Obs2K} and \eqref{Obs_FirstVectorK} the observability inequality 
\[\| g(T,\cdot) \|_{L^2(\T)^{d}}^2 \leq C \int_0^T \int_{\omega} |g_1(t,x)|^2 \diff x \diff t, \]
in the case $d_1=1$. This concludes the proof of \Cref{th:main_2_facile} in the case $d_1=1$ by duality.

\subsubsection{The case of several hyperbolic components: $d_1 >1$}
\label{SectionKd1>1}

In this section, we deal with the general problem of null-controllability of~\eqref{Syst_bloc_hyp_K}. To this aim, we introduce $K_{21}^i \in \R^{d_2}$ the $i$-th column of the matrix $K_{21}$ ($1 \leq i \leq d_1$), i.e.
\[ K_{21} = \left(K_{21}^1 | K_{21}^2 | \dots | K_{21}^{d_1} \right),\] 
From the Kalman rank condition~\eqref{Kalman_Thm2_facile}, we construct an adapted basis of $\C^{d_2}$.
\begin{lemma}
\label{AdaptedBasisLemma}
There exist $r \in \{1, \dots, d_2\}$ and sequences $(l_j)_{1 \leq j \leq r} \subset \{1, 2, \dots, d_1\}$ and $(s_j)_{1 \leq j \leq r} \subset \{1, 2, \dots, d_2\}$ with $\sum_{j=1}^r s_j =d_2$, such that
\[ \mathcal{B} = \bigcup\limits_{j=1}^{r} \left\{K_{21}^{l_j}, K_{22} K_{21}^{l_j}, \dots, K_{22}^{s_j-1} K_{21}^{l_j} \right\}\]
is a basis of $\C^{d_2}$. Moreover, for every $j$, with $1 \leq j \leq r$, there exist $\alpha_{k,s_j}^i \in \R$ ($1 \leq i \leq j$, $1 \leq k \leq s_j$) such that
\begin{equation}
K_{22}^{s_j} K_{21}^{l_j} = \sum\limits_{i=1}^j \left( \alpha_{1, s_j}^i K_{21}^{l_i} + \alpha_{2, s_j}^i K_{22} K_{21}^{l_i} + \dots + \alpha_{s_i,s_j}^i K_{22}^{s_i-1} K_{21}^{l_i}\right).
\label{eqLemmaBasisK}
\end{equation}
\end{lemma}
For a proof of this lemma, see \cite[Lemma 3.1]{ammarkhodja_2009a}.

Let $\mathcal{B}$ the basis of $\C^{d_2}$ provided by \Cref{AdaptedBasisLemma} and $P$ be the matrix whose columns are the elements of $\mathcal{B}$, i.e.
\[ P \coloneqq \left(K_{21}^{l_1} | K_{22} K_{21}^{l_1} | \dots | K_{22}^{s_1-1} K_{21}^{l_1}| \dots | K_{21}^{l_r} | \dots | K_{22}^{s_r-1} K_{21}^{l_r} \right).\]
Let us observe that the basis $\mathcal{B}$ has been constructed in such a way that \eqref{eqLemmaBasisK} is satisfied.

Let the matrices $C_{ii} \in \R^{s_i \times s_i}$ and $C_{ij} \in \R^{s_i \times s_j}$, $1 \leq i < j \leq r$, be defined by
\begin{equation}
C_{ii} =  \begin{pmatrix} 0 & 0 & 0 & \dots & \alpha_{1, s_i}^i\\
1 & 0 & 0 & \dots & \alpha_{2, s_i}^i\\
0 & 1 & 0 & \dots & \alpha_{3, s_i}^i \\
\vdots & \vdots & \ddots & \ddots & \vdots\\
0 & 0 & \dots & 1 & \alpha_{s_i,s_i}^i
\end{pmatrix}\ \text{and}\ 
C_{ij} = \begin{pmatrix} 0 & \dots & 0 & \alpha_{1, s_j}^i\\
\vdots & \ddots & \vdots & \alpha_{2, s_j}^i\\
0 &  \dots & 0 & \alpha_{s_i,s_j}^i
\end{pmatrix}.
\label{defblockC_iK}
\end{equation}
We set 
\begin{equation}
\label{eq:cascadeK_d1>1}
\widehat{K_{22}}\coloneqq  \begin{pmatrix} C_{11} & C_{12} & \dots & C_{1r}\\
0 & C_{22} &\dots &C_{2r}\\
\vdots & \vdots & \ddots & \vdots \\
0 & 0 &\dots & C_{rr}\\
\end{pmatrix}\ \text{and}\
\widehat{K_{21}}\coloneqq P^{-1} K_{21}.
\end{equation}

From \eqref{eqLemmaBasisK}, \eqref{eq:cascadeK_d1>1} and \eqref{defblockC_iK}, by denoting $P_ i \coloneqq \left(K_{21}^{l_i} | K_{22} K_{21}^{l_i} | \dots | K_{22}^{s_i-1} K_{21}^{l_i}\right)$, we have
\begin{align*}
K_{22} P_i &= \left(K_{22}K_{21}^{l_i} | K_{22}^2 K_{21}^{l_i} | \dots | K_{22}^{s_i} K_{21}^{l_i}\right)\\
&= \left(K_{22}K_{21}^{l_i} \Big| K_{22}^2 K_{21}^{l_i}  \Big| \dots  \Big| \sum_{k=1}^i \left( \alpha_{1, s_i}^k K_{21}^{l_k} + \alpha_{2, s_i}^k K_{22} K_{21}^{l_k} + \dots + \alpha_{s_k,s_i}^k K_{22}^{s_k-1} K_{21}^{l_k}\right) \right)\\
&= \left(0  \Big| \dots  \Big| 0  \Big|  \sum_{k=1}^{i-1} \left( \alpha_{1, s_i}^k K_{21}^{l_k} + \alpha_{2, s_i}^k K_{22} K_{21}^{l_k} + \dots + \alpha_{s_k,s_i}^k K_{22}^{s_k-1} K_{21}^{l_k}\right) \right)\\
&\quad + \left(K_{22}K_{21}^{l_i}  \Big| K_{22}^2 K_{21}^{l_i}  \Big| \dots  \Big| \left( \alpha_{1, s_i}^i K_{21}^{l_i} + \alpha_{2, s_i}^i K_{22} K_{21}^{l_i} + \dots + \alpha_{s_i,s_i}^i K_{22}^{s_i-1} K_{21}^{l_k}\right) \right) \\
&= P_1 C_{1i} + P_{2} C_{2i} + \dots + P_{i} C_{ii}.
\end{align*}
Then, we obtain
\begin{equation}
 K_{22} P = P\widehat{K_{22}}\ \text{and}\ P e_{S_i} = K_{21}^{l_i},\ 1 \leq i \leq r,
\label{RelationK22PK21}
\end{equation}
where $e_{S_i}$ is the $S_i$-element of the canonical basis of $\R^n$ with $S_i = 1 + \sum_{j=1}^{i-1} s_j$. In the following, we will also use the notation $S_{r+1}\coloneqq d_2+1$.

We argue as in the previous subsection. We perform the same change of variable $w =(w_1,w_2) = (f_1,P^{-1}f_2)$, we consider the solution $g$ of the adjoint system \begin{equation} \label{Syst_bloc_hyp_K_Adjd1>1}
\mkern-3mu\left\{\!
\begin{array}{ll}
\left( \partial_t -\Tr{A'}\partial_x+\Tr K_{11}\right)g_1 +\Tr{\widehat{K_{21}}} g_2= 0
&\text{in}\ (0,T)\times \T,\!\\
\left( \partial_t - \partial_x^2+\Tr{\widehat{K_{22}}}\right)g_2 +\left(-\Tr{(A_{12}P)}\partial_x+\Tr{(K_{12}P)}\right)g_1  = 0 &\mathrm{in}\ (0,T)\times \T,\!\\
(g_1,g_2)(0,\cdot)=(g_{01},g_{02}) & \text{in}\ \T.
\end{array} \right.
\end{equation}
From \Cref{CorHUM}, we recall that the solution $g$ of~\eqref{Syst_bloc_hyp_K_Adjd1>1} satisfies
\begin{equation} \label{IO_H-1-d1>1K}
\| g(T,\cdot) \|_{L^2(\T)} \leq C \left( \| g_1\|_{L^2(q_T)} + \|g_2\|_{H^{-2m+1}(q_T)} \right),\ \text{with}\ m=\max_{1 \leq i \leq r}s_i.
\end{equation}
Then, we use the coupling terms in the system \eqref{Syst_bloc_hyp_K_Adjd1>1} in order to get rid of the term $\|g_2\|_{H^{-2m+1}(q_T)}^2$ in the right hand side of the inequality \eqref{IO_H-1-d1>1K}.

From the cascade form of the matrix $\widehat{K_{22}}$, see \eqref{eq:cascadeK_d1>1}, more precisely from the cascade form of the block matrix $C_{ii}$ and the form of the matrices $C_{1,i}, \dots, C_{i-1,i}$,  see \eqref{defblockC_iK}, the equations of the adjoint system \eqref{Syst_bloc_hyp_K_Adjd1>1} are
\begin{align}
\label{equationadjointBloc}
&\forall i \in \{1, \dots, r\},\ \forall j \in \{S_i, \dots, S_{i+1}-2\},\notag\\
&\partial_t g_2^{j} - \partial_x^2 g_2^{j} + g_2^{j+1} + \sum_{k=1}^{d_1} b_{i,j}^k \partial_x g_1^k + a_{i,j}^k g_1^k =0,\qquad ( a_{i,j}^k,b_{i,j}^k) \in \R^2 .
\end{align}
To simplify, we will denote by $H^k$, $k \in \Z^{-}$, the space $H^{k}(q_T)$.

We deduce successively from \eqref{equationadjointBloc} with $j=S_{i+1}-2, S_{i+1}-3, \dots, S_{i+1}-2-(s_{i}-2) = S_i$, the following estimates
\begin{align*}
\norme{g_2^{S_{i+1}-1}}_{H^{-2s_i+1}(q_T)} & \leq C \left(\norme{g_2^{S_{i+1}-2}}_{H^{-2(s_i-1)+1}(q_T)} + \norme{g_1}_{L^2(q_T)}\right)\notag\\
& \leq C \left(\norme{g_2^{S_{i+1}-3}}_{H^{-2(s_i-2)+1}(q_T)} + \norme{g_1}_{L^2(q_T)}\right) \notag\\
& \leq \dots\label{cascadeg2RK}\\
& \leq C \left(\norme{g_2^{S_{i+1} -2 -(s_i-2)}}_{H^{-1}(q_T)} + \norme{g_1}_{L^2(q_T)}\right). \notag
 \end{align*}
So, we have for every $i \in \{1, \dots, r\}$ and $j \in \{S_i+1, \dots, S_{i+1}-1\}$,
\begin{equation}
\label{cascadeg2RKBis}
\norme{g_2^{j}}_{H^{-2m+1}(q_T)} \leq  C\left( \norme{g_2^{S_i}}_{H^{-1}(q_T)} + \norme{g_1}_{L^2(q_T)}\right).
\end{equation}
Then, by using \eqref{eq:cascadeK_d1>1} and \eqref{RelationK22PK21}, we have $\widehat{K_{21}}^{l_i}=P^{-1} K_{21}^{l_i} = e_{S_i}$. Consequently, the $l_i$-th equation of the adjoint system \eqref{Syst_bloc_hyp_K_Adjd1>1} is 
\begin{equation*}
\partial_t g_1^{l_i} + \sum\limits_{k=1}^{d_1} a_{l_i,k} \partial_x g_1^{k} + b_{l_i,k} g_1^{k} + g_2^{S_i} = 0,\qquad (a_{l_i,k}, b_{l_i,k}) \in \R^2.
\end{equation*}
Then, we obtain
\begin{equation}
\label{g_2S_RK}
\norme{g_2^{S_i}}_{H^{-1}(q_T)} \leq C \norme{g_1}_{L^2(q_T)}.
\end{equation}
By gathering \eqref{cascadeg2RKBis} and \eqref{g_2S_RK}, we obtain 
\begin{equation}
\label{Obsg2Rg1K}
\forall i \in \{1, \dots, r\},\ \forall j \in \{S_i, \dots, S_{i+1}-1\},\  \norme{g_2^{j}}_{H^{-2m+1}(q_T)} \leq C \norme{g_1}_{L^2(q_T)}.
\end{equation}
By using that $\{S_1, \dots, S_{2}-1, S_2, \dots, S_{3}-1, \dots, S_{r}, \dots, S_{r+1}-1\} = \{1, \dots, d_2\}$, we finally deduce from \eqref{Obsg2Rg1K} and \eqref{IO_H-1-d1>1K}
the observability inequality\[\| g(T,\cdot) \|_{L^2(\T)^{d}}^2 \leq C \int_0^T \int_{\omega} |g_1(t,x)|^2 \diff x \diff t. \]
This concludes the proof of \Cref{th:main_2_facile} in the case $d_1>1$ by duality.

\section{Hyperbolic control: coupling of order one} \label{Sec:Hyp}

The goal of this section is to prove Theorem \ref{th:main_2}.
The requirement of the Kalman rank condition \eqref{Kalman_Thm2}  for null-controllability is an adaptation of the proof given in \Cref{SectionRequirementKalman}. Now, we explain how to complete the proof of Theorem~\ref{th:main} to prove that the Kalman condition is sufficient for null controllability. We set
\begin{equation}
\label{defInvQuant}
\bold{F_2} \coloneqq L^2(\T)^{d_1} \times L^2_\m(\T)^{d_2} = \left\{ f_0=(f_{01},f_{02}) \in L^2(\T)^d\ ;\ \int_{\T} f_{02}(x) \diff x = 0\right\}.
\end{equation} 
\indent We only give the proof in the case $d_1=1$. The case $d_1 >1$ is an easy adaptation of the case $d_1=1$ and the arguments already presented for coupling terms of order zero in \Cref{SectionKd1>1}.

\subsection{A special observability inequality}

The goal of this section is to prove the following observability inequality. 
\begin{prop}\label{PropObs1A}
There exists $C>0$ such that for every $g_0 \in \bold{F_2}$, the solution of the adjoint system \eqref{SystAdj} satisfies
\begin{equation} \label{IO_H_Obs1A}
\| g(T,\cdot) \|_{L^2(\T)}^2 \leq C \left( \| g_1\|_{L^2(q_T)}^2 + \|\partial_{x}^{d_2} g_2\|_{H^{-2d_2+1}(q_T)}^2 \right) .
\end{equation}
\end{prop}
In order to prove \Cref{PropObs1A}, by a duality argument, it is sufficient to establish the following null-controllability result.
\begin{prop}\label{PropNullControldx}
For every $f_0 \in \bold{F_2}$, there exists $u \in L^2(q_T)^{d_1}\times(H_0^{2d_2-1}(q_T))^{d_2}$ such that $S(T, f_0, (u_{\h}, \partial_{x}^{d_2} u_{\p})) = 0$.
\end{prop}
\begin{proof}[Proof of the equivalence between \Cref{PropObs1A} and \Cref{PropNullControldx}]
We apply \Cref{LemmaDuality} with 
\[ \Phi_2 \colon f_0 \in \mathbf{F_2} \mapsto f(T,\cdot) \in \mathbf{F_2},\]
where $f$ is the solution to the system \eqref{Syst} with initial data $f_0$ and control $u=0$, and
\[ \Phi_3 \colon u=(u_1,u_2) \in L^2(q_T)^{d_1} \times H_0^{2d_2-1}(q_T)^{d_2} \mapsto f(T,\cdot) \in \mathbf{F_2},\]
where $f$ is the solution to the system \eqref{Syst} with initial data $f_0=0$ and control $(u_1,\partial_x^{d_2} u_2)$. Note that by integrating the second equation of the system~\eqref{Syst_bloc}, we see that a control of the form $(u_1,\partial_x^{d_2}u_2)$ cannot change the mean of the parabolic component. This justifies that $\Phi_2$ and $\Phi_3$ do indeed take values in $\mathbf{F_2}$.

The null-controllability result of \Cref{PropNullControldx} is equivalent to the existence of $C>0$ such that for every $g_T\in L^2(\T)^d$,
\begin{equation}\label{eq:obs_lemma2hyp}
\lVert\Phi_2^*(g_T)\rVert_{\mathbf{F_2}}\leq C\lVert \Phi_3^*(g_T)\rVert_{L^2(q_T)^{d_1}\times H^{2d_2-1}_0(q_T)^{d_2}}.
\end{equation}


We have $\Phi_2^*(g_T) = (\eu^{-T\mathcal L})^* g_T = \eu^{-T\mathcal L^*}g_T$. We claim that the right-hand side of the inequality~\eqref{eq:obs_lemma2hyp} satisfies
\begin{equation}\label{eq:dual_hyper_dx_rhs}
\lVert\Phi_3^*(g_T)\rVert_{L^2(q_T)^{d_1}\times H^{2d_2-1}_0(q_T)^{d_2}} = \lVert (g_1, (-1)^{d_2}\partial_{x}^{d_2} g_2) \rVert_{L^2(q_T)^{d_1}\times H^{-2d_2+1}(q_T)^{d_2}},
\end{equation}
where  $g = \eu^{-(T-t)\mathcal L^*}g_T$. This will prove that the inequality~\eqref{eq:obs_lemma2hyp} is exactly the observability inequality \eqref{IO_H_Obs1A}. 

We write $\Phi_3$ as 
\[\Phi_3 = \mathcal F_T \circ (I, \partial_x^{d_2}) \circ (I,\iota_{2d_2-1}),\]
where $\mathcal F_T\colon L^2(\T)^d \to L^2(\T)^d$ is the input-output operator introduced in the proof of~\Cref{CorHUM}, $\partial_x^{d_2}$ is seen as an unbounded operator on $L^2(\T)^{d_2}$ with domain $H^{d_2}(\T)^{d_2}$, and $\iota_{2d_2-1}\colon H^{2d_2-1}(\T)^{d_2}\to L^2(\T)^{d_2}$ is the inclusion map (see \Cref{th:lemma_dual_Hs}). Note that while $\Phi_3$ written this way looks like an unbounded operator (because $\partial_x^{d_2}$ is), we have $\Ima(\iota_{2d_2-1})\subset D(\partial_x^{d_2})$, so that the composition of operators above is indeed a continuous operator. So, we have
\[
\Phi_3^* = (I,\iota_{2d_2-1}^*)\circ(I, (\partial_x^{d_2})^*) \circ \mathcal F_T^* =(I,\iota_{2d_2-1}^*)\circ(I,(-1)^{d_2} \partial_x^{d_2}) \circ \mathcal F_T^*.
\]
Since $\iota_{2d_2-1}^*$ is an isometry between $H^{2d_2-1}_0$ and $H^{-2d_2+1}$ (see \Cref{th:lemma_dual_Hs}), this proves the relation~\eqref{eq:dual_hyper_dx_rhs}.
\end{proof}

First, we show that the null-controllability result of \Cref{PropNullControldx} is true at the high-frequency level, i.e.\ we prove the following adaptation of \Cref{PropControlHighFreq}. 
\begin{prop}
\label{PropControlHighFreqd_x}
There exists a closed subspace $\mathcal{G}^{\blacklozenge}$ of $L^2(\T)^{d}$ with finite codimension and a continuous operator 
\begin{equation*}
\mathcal{U}^{\blacklozenge} \colon\begin{array}[t]{@{}c@{}l}
\mathcal{G}^{\blacklozenge} &\to  L^2((0,T')\times\omega)^{d_1} \times C^\infty_c((T',T)\times\omega)^{d_2}\\
f_0 &\mapsto (u_{\h},u_{\p}),
\end{array}
\end{equation*}
that associates with each $f_0 \in \mathcal{G}^{\blacklozenge} $ a pair of controls $\mathcal{U}^{\blacklozenge} f_{0}=(u_{\h},u_{\p})$ such that 
\begin{equation}
\label{PropKd_x}
\forall f_0 \in \mathcal{G}^{\blacklozenge},\ \Pi S(T;f_0, (u_{\h},\partial_{x}^{d_2}u_{\p})) = 0.
\end{equation}
\end{prop}
In order to prove \Cref{PropControlHighFreqd_x}, it is enough to prove \Cref{LemPar} with parabolic control of the form $\partial_{x}^{d_2} u_{\p}$. Thus, by using \Cref{sec:reduc_parab}, it is sufficient to show the following adaptation of \Cref{LemPar_Bis}.
\begin{prop}
\label{LemPar_Bisd_x}
If $n_0$ is large enough, then for every $T>0$, there exists a continuous operator 
\begin{equation*}
\underline{\mathcal{U}}_T^{\p,\blacklozenge} \colon \begin{array}[t]{@{}c@{}l}
  F^\p  & \rightarrow  C^\infty_c((0,T)\times\omega)^{d_2} \\
  f_0 &\mapsto u_{\p},
\end{array}
\end{equation*}
that associates with each $f_0 \in F^{\p} $ a control $\underline{\mathcal{U}}_T^{\p,\blacklozenge} f_{0}=u_{\p}$ such that 
\begin{equation*}
\Pi^{\p} S(T; f_0, (0, \partial_{x}^{d_2}u_{\p})) = 0.
\end{equation*}
\end{prop}
\begin{proof}
Let $f_0 \in F^{\p}$ and $f_0^{\blacklozenge}$ be such that $\partial_{x}^{d_2}f_0^{\blacklozenge} = f_0$. Note that $f_0^{\blacklozenge}$ is well-defined because $\int_{\T} f_0(x) dx = 0$. We know from \Cref{LemPar_Bis} that there exists $u_{\p} \in C^\infty_c((0,T)\times\omega)^{d_2}$ such that the solution $f^{\blacklozenge}$ of \eqref{Syst} with initial data $f_0^{\blacklozenge}$ and control $(0,u_{\p})$ satisfies
\[ \Pi^{\p} f^{\blacklozenge}(T,\cdot)=0.\]
Then, by setting $f\coloneqq \partial_{x}^{d_2} f^{\blacklozenge}$ and by applying $\partial_{x}^{d_2}$ to the system \eqref{Syst} satisfied by $f^{\blacklozenge}$, we deduce that $f$ is the solution of \eqref{Syst} with initial data $f_0$ and control $(0, \partial_x^{d_2}u_{\p})$ and satisfies 
\[ \Pi^{\p} f(T,\cdot)=0,\]
because $\partial_{x}^{d_2}$ and $\Pi^{\p}$ commute.

We get the conclusion of the proof of \Cref{LemPar_Bisd_x} with the continuous operator $\underline{\mathcal{U}}_T^{\p,\blacklozenge}(f_0) =  \underline{\mathcal{U}}_T^\p\left(f_0^{\blacklozenge}\right)$ where $\underline{\mathcal{U}}_T^\p$ is the operator defined in \Cref{LemPar_Bis}.
\end{proof}

Secondly, we have to show that the null-controllability result of \Cref{PropNullControldx} is true at the low frequency-level, as we have already shown for \Cref{th:main} in \Cref{ControlLowFreqSec}. All the steps of \Cref{ControlLowFreqSec} remain unchanged except the Step $6$. Indeed, the unique continuation argument transforms into: if $f(t,\cdot) = \eu^{tM} f_0$ with $(f_1, \partial_{x}^{d_2} f_2) = (0,0)$ in $(0,\varepsilon)\times\omega$ then $(f_{01}, \partial_{x}^{d_2} f_{02})= (0,0)$ thanks to the spectral inequality of Lebeau-Robbiano \eqref{eq:spectral}, that is to say, $f_0 = 0$ because $\int_{\T} f_{02}(x) \diff x =0$.

This concludes the proof of \Cref{PropNullControldx} thus the proof of \Cref{PropObs1A}. \qed

\subsection{The case of one hyperbolic component: $d_1 = 1$}
\label{SectionControlHyperbolicd_1=1}

By the Hamilton-Cayley's theorem, there exist $c_0, \dots, c_{d_2-1} \in \R$ such that
\[ A_{22}^{d_2} = c_0 I_{d_2} + c_1 A_{22} + \dots + c_{d_2-1} A_{22}^{d_2-1}.\]
By using the Kalman condition \eqref{Kalman_Thm2}, the matrix $P$ defined as follows
\[ P \coloneqq (A_{21}, A_{22} A_{21}, \dots, A_{22}^{d_2-1} A_{21}),\]
is invertible. By setting
\begin{equation}
\label{eq:cascade}
 \widehat{A_{22}} \coloneqq  \begin{pmatrix} 0 & \dots & \dots & 0 & c_0\\
1 & 0 &\dots & \vdots & c_1\\
0 & \ddots & \ddots & \vdots & c_2\\
\vdots & \ddots & \ddots & 0 & \vdots\\
0 & \dots &0 & 1 & c_{d_2-1}
\end{pmatrix}\ \text{and}\ \widehat{A_{21}} \coloneqq \begin{pmatrix} 1\\0\\ \vdots\\ 0 \end{pmatrix},
\end{equation}
we check that we have the following relations
\[A_{22}P = P \widehat{A_{22}}\ \text{and}\  A_{21} = P \widehat{A_{21}},\ \text{i.e.}\ \widehat{A_{22}} = P^{-1} A_{22} P\ \text{and}\ \widehat{A_{21}} = P^{-1} A_{21}.\]
Then, by setting $w =(w_1,w_2) = (f_1,P^{-1}f_2)$, we have
\begin{equation} \label{Syst_bloc_hyp_A_w}
\mkern-3mu\left\{\!
\begin{array}{ll}
\left( \partial_t +A'\partial_x+K_{11}\right)w_1+\left(A_{12}P\partial_x+K_{12}P\right)w_2= u_1  1_\omega
&\text{in}\ (0,T)\times \T,\!\\
\left( \partial_t - \partial_x^2+\widehat{A_{22}}\partial_x\right)w_2 + \widehat{A_{21}}\partial_x w_1 = 0 &\mathrm{in}\ (0,T)\times \T,\!\\
(w_1,w_2)(0,\cdot)=(w_{01},w_{02}) & \text{in}\ \T.
\end{array} \right.
\end{equation}
The system \eqref{Syst_bloc_hyp_A_w} is a “cascade system” with coupling terms of order one in the spatial variable.

The adjoint system of \eqref{Syst_bloc_hyp_A_w} is
\begin{equation} \label{Syst_bloc_hyp_A_Adj}
\mkern-3mu\left\{\!
\begin{array}{ll}
\left( \partial_t -\Tr{A'}\partial_x+\Tr{K_{11}}\right)g_1 - \Tr{\widehat{A_{21}}}\partial_x g_2= 0
&\text{in}\ (0,T)\times \T,\!\\
\left( \partial_t - \partial_x^2-\Tr{\widehat{A_{22}}}\partial_x\right)g_2 +\left(-\Tr{(A_{12}P)}\partial_x+\Tr{(K_{12}P)}\right)g_1  = 0 &\mathrm{in}\ (0,T)\times \T,\!\\
(g_1,g_2)(0,\cdot)=(g_{01},g_{02}) & \text{in}\ \T.
\end{array} \right.
\end{equation}
\normalsize
We know from \Cref{PropObs1A} that the solution $g$ of \eqref{Syst_bloc_hyp_A_Adj} satisfies
\begin{equation} \label{IO_H_Obs1ASC}
\| g(T,\cdot) \|_{L^2(\T)} \leq C \left( \| g_1\|_{L^2(q_T)} + \|\partial_{x}^{d_2} g_2\|_{H^{-2d_2+1}(q_T)} \right) .
\end{equation}
By using the fact that $\widehat{A_{22}}$ is a companion matrix, see \eqref{eq:cascadeK}, for every $i \in \{2, \dots, d_2\}$, the $i$-th equation of \eqref{Syst_bloc_hyp_A_Adj} is
\begin{equation}
 \partial_t g_2^{i-1} - \partial_x^2 g_2^{i-1} + \partial_{x} g_2^{i} + b_{i-1} \partial_x g_1+ a_{i-1} g_1 =0,\qquad (a_{i-1}, b_{i-1}) \in \R^2.
 \label{ithequationA}
 \end{equation}
Then, by applying $\partial_{x}^{i-1}$ to \eqref{ithequationA} with $i \in \{2, \dots, d_2\}$, we get that there exists $C>0$ such that
\begin{align}
\|\partial_{x}^{i} g_2^{i}\|_{H^{-2i + 1 }(q_T)} &\leq C \left(\|(\partial_t - \partial_{x}^2)\partial_x^{i-1} g_2^{i-1}\|_{H^{-2i+1}(q_T)} + \| (b_{i-1} \partial_x^{i}+ a_{i-1}\partial_{x}^{i-1}) g_1\|_{H^{-2i+1}(q_T)} \right),\notag
\end{align}
therefore we have
\begin{align}
\|\partial_{x}^{i} g_2^{i}\|_{H^{-2i + 1 }(q_T)} &\leq C \left(\|\partial_x^{i-1} g_2^{i-1}\|_{H^{-2(i-1)+1}(q_T)} + \| g_1\|_{L^2(q_T)} \right).
\label{Obs_CascadeA}
\end{align}
Then, we deduce from \eqref{IO_H_Obs1ASC} and \eqref{Obs_CascadeA} that 
\begin{align} 
\| g(T,\cdot) \|_{L^2(\T)} & \leq C \left( \| g_1\|_{L^2(q_T)} + \|\partial_{x}^{d_2} g_2\|_{H^{-2d_2+1}(q_T)} \right)\notag\\
& \leq C \left(\| g_1\|_{L^2(q_T)} + \sum_{i=1}^{d_2}\|\partial_{x}^{i} g_2^{i}\|_{H^{-2i + 1 }(q_T)} \right)\notag\\
&\leq C \left( \| g_1\|_{L^2(q_T)} + \|\partial_{x} g_2^1\|_{H^{-1}(q_T)}\right). \label{IO_H_Obs2A}
\end{align}
By using the fact that $\widehat{A_{21}}$ is the first vector of the canonical basis of $\R^{d_2}$, see \eqref{eq:cascade}, the first equation of \eqref{Syst_bloc_hyp_A_Adj} is
\[ \partial_t g_1 - A' \partial_{x} g_1 + K_{11} g_1 + \partial_{x} g_2^{1} = 0.\]
Then, we obtain that
\begin{equation}
\label{Obs_FirstVectorA}
\|\partial_{x} g_2^1\|_{H^{-1}(q_T)} \leq C \| g_1\|_{L^2(q_T)}.
\end{equation}
So, we deduce from \eqref{IO_H_Obs2A} and \eqref{Obs_FirstVectorA} the observability inequality
\[\| g(T,\cdot) \|_{L^2(\T)} \leq C  \| g_1\|_{L^2(q_T)}, \]
which concludes the proof of \Cref{th:main_2} in the case $d_1=1$. \qed

\section{Parabolic control} \label{Sec:Par}

The goal of this section is to prove Theorem \ref{th:main_3} and to illustrate the necessity of a regularity assumption on the initial condition.

\subsection{A regularity assumption is necessary}

We consider for $\lambda\in\mathbb{R}^*$ the system
\begin{equation} \label{c-ex_parab}
\left\lbrace \begin{array}{ll}
\partial_t \widetilde{f}_1 + \lambda \partial_x \widetilde{f}_1 + \partial_x \widetilde{f}_2 = 0, &\ \text{ in } (0,T)\times\T, \\
\partial_t \widetilde{f}_2 - \partial_x^2 \widetilde{f}_2 + \lambda \partial_x \widetilde{f}_2 = v(t,x), &\ \text{ in } (0,T)\times\T,
\end{array}\right.
\end{equation}
i.e.\ $\omega=\T$, $d=2$, $m=1$, 
\[A=\begin{pmatrix}
\lambda & 1 \\ 0 & \lambda 
\end{pmatrix},\quad
A'=(\lambda), \quad
B=\begin{pmatrix}
0 & 0 \\ 0 & 1
\end{pmatrix},\quad
M=\begin{pmatrix}
0  \\ 1 
\end{pmatrix},\]
that satisfies (\ref{h:D}),(\ref{h:A1}) and the Kalman condition~\eqref{Kalman_Thm3} because $A_{12}=1$.
By Theorem \ref{th:main_3}, any initial condition $f_0=(f_{01},f_{02})\in H^{2}_\m \times H^2(\T)$ is null controllable.
The following statement illustrates that
\begin{itemize}
\item a regularity assumption on $f_{01}$ is necessary for the null controllability
\item the one given by Theorem \ref{th:main_3} is sufficient but may not be necessary.
\end{itemize}

\begin{prop}
An initial condition $f_0=(f_{01},f_{02}) \in L^2_\m(\T)\times L^2(\T)$ 
is null controllable with $v \in L^2((0,T)\times\T)$ if and only if $f_{01} \in H^1(\T)$.
\end{prop}
\begin{remark}
Similar problems of regularity between initial data and control have already been noticed in the context of transport systems, see \cite[Remark 5]{alabau-boussouira_2017}.
\end{remark}
\begin{proof}
In the proof, we use the notation $Q_T = (0,T)\times\Omega$.

The change of variable
\[\widetilde{f}_j(t,x)=f_j(t,x-\lambda t),\qquad v(t,x)=u(t,x-\lambda t)\]
leads to
\begin{equation} \label{c-ex_parabVar}
\left\lbrace \begin{array}{ll}
\partial_t f_1 - \partial_x f_2 = 0, &\ \text{ in } (0,T)\times\T, \\
\partial_t f_2 - \partial_x^2 f_2 = u(t,x), &\ \text{ in } (0,T)\times\T. 
\end{array}\right.
\end{equation}
The null controllability of $(\widetilde{f}_1,\widetilde{f}_2)$ with control $v \in L^2(Q_T)$ is equivalent to the null controllability of $(f_1,f_2)$ with control $u \in L^2(Q_T)$. On Fourier components, equation \eqref{c-ex_parabVar} gives the ordinary differential equations
\begin{equation} \label{c-ex_parab2}
\left\lbrace \begin{array}{ll}
\frac{\diff}{\diff t} \widehat{f}_1(t,n) = \iu n \widehat{f}_2(t,n), \\
\frac{\diff}{\diff t} \widehat{f}_2(t,n)=-n^2 \widehat{f}_2(t,n) + \widehat{u}(t,n).
\end{array}\right.
\end{equation}
Let $f_0=(f_{01},f_{02}) \in L_{\m}^2(\T) \times L^2(\T)$. The solution writes
\[\widehat{f}_2(t,n)=\widehat{f}_{02}(n) \eu^{-n^2 t} + \int_0^t \eu^{-n^2(t-\tau)}\widehat{u}(\tau,n) \diff\tau,\]
\begin{align*}
\widehat{f}_1(t,n) 
& =\widehat{f}_{10}(n) + \iu n \int_0^t \widehat{f}_2(\tau,n) \diff \tau \\
& = \widehat{f}_{01}(n) + \frac{\iu}{n}(1-\eu^{-n^2 T})\widehat{f}_{02}(n) + \iu n \int_0^t \hat u(\tau,n)  \frac{1-\eu^{-n^2(t-\tau)}}{n^2} \diff\tau.
\end{align*}
thus the relation $f(T)=0$ is equivalent to the moment problem
\begin{equation} \label{moment_pb}
\begin{array}{ll}
\int_0^T \eu^{-n^2(T-\tau)}\widehat{u}(\tau,n) \diff \tau=-\widehat{f}_{02}(n) \eu^{-n^2 T}, \quad & \forall n \in \Z, \\ 
\int_0^T \widehat{u}(\tau,n)  \diff \tau = \iu n \widehat{f}_{01}(n) - \widehat{f}_{02}(n), \quad & \forall n \in \Z\setminus\{0\}.
\end{array}
\end{equation}
Note that the assumption $\int_{\T} f_{01} = 0$ implies $\int_{\T} f_1(t) =0$ for every $t>0$ thus the null controllability of this component does not require any condition on the control $u$.
  
 \paragraph{Necessary condition:} We assume $f_0=(f_{01},f_{02})$ null controllable with a control $u \in L^2(Q_T)$ and we prove that $f_{01} \in H^1(\T)$.
 By the Bessel-Parseval equality and Cauchy-Schwarz inequality,
\begin{align*}
\|u\|_{L^2(Q_T)}^2 & = \sum\limits_{n\in\Z} 
\int_0^T | \widehat{u}(t,n) |^2 \diff t
\\
 & \geq \sum\limits_{n\in\Z} \frac{1}{T}
\left| \int_0^T  \widehat{u}(t,n)  \diff t \right|^2
\\ & \geq
\sum\limits_{n\in\Z} \frac{1}{T} \left| \iu n \widehat{f}_{01}(n) - \widehat{f}_{02}(n) \right|^2 = \frac{1}{T} \|\partial_x f_{01} - f_{02}\|_{L^2(\T)}^2 
 \end{align*}
 thus $f_{01} \in H^1(\T)$.
 
 \paragraph{Sufficient condition:} We assume $f_0=(f_{01},f_{02}) \in H^1_\m \times L^2(\T)$ and we construct a control $u \in L^2((0,T)\times\T)$ that steers this initial condition to $0$. 
 
 Let $G_n$ be the Grammian matrix, in $L^2(0,T)$, of the family $(w_{1,n},w_{2,n})$  where $w_{1,n}\colon\tau \mapsto n \eu^{-n^2(T-\tau)}$ and $w_{2,n}: \tau \mapsto 1$, i.e.\ $(G_n)_{i,j}=\int_0^T w_{i,n}(\tau) w_{j,n}(\tau) \diff\tau$ for every $1 \leq i,j \leq 2$. 
 Then $G_n$ is invertible for every $n\in\Z\setminus\{0\}$ 
 (because it is the Grammian matrix of a linearly independent family)
    and, when $|n| \rightarrow \infty$,
 \[G_n \sim \left( \begin{array}{cc} 
 1/2 & 1/n \\ 1/n & T
  \end{array}\right)\]
thus there exists $C>0$ such that, for every $n\in\mathbb{Z} \setminus\{0\}$, $\| G_n^{-1} \| \leq C$. We take 
 \[u(\tau,x)= - \frac{1}{T} \widehat{f}_{02}(0) + \sum_{n\in\mathbb{Z}\setminus\{0\}} \left( \alpha_n w_{1,n}(\tau) + \beta_n w_{2,n}(\tau) \right) \eu^{\iu nx}\]
 where 
 \begin{equation} \label{def:alpha_beta_n}
 \left(\begin{array}{c} \alpha_n \\ \beta_n \end{array}\right)\coloneqq G_n^{-1}
 \left(\begin{array}{c} 
 -n \widehat{f}_{02}(n) \eu^{-n^2 T}
 \\ 
 \iu n \widehat{f}_{01}(n) - \widehat{f}_{02}(n)
 \end{array} \right).
 \end{equation}
 Then, by Bessel-Parseval equality, we have for various positive constants $C$ depending on $T$,
 \begin{align*}
 \|u\|_{L^2(Q_T)}^2 
 & =\frac{1}{T} |\widehat{f}_{02}(0)|^2 + \int_0^T \sum_{n\in\Z\setminus\{0\}}| \alpha_n w_{1,n}(t) + \beta_n w_{2,n}(t) |^2 dt \\
 & \leq \frac{1}{T} |\widehat{f}_{02}(0)|^2 + C\sum_{n\in\Z\setminus\{0\}} \left( |\alpha_n|^2 + |\beta_n|^2  \right) \\
 & \leq \frac{1}{T} |\widehat{f}_{02}(0)|^2 + C \sum_{n\in\Z\setminus\{0\}} \left( |n \widehat{f}_{02}(n) \eu^{-n^2 T}|^2 + |\iu n \widehat{f}_{01}(n) - \widehat{f}_{02}(n)|^2  \right) \\
& \leq C \left( \|f_{01}\|_{H^1(\T)}^2 + \|f_{02}\|_{L^2(\T)}^2 \right) < \infty.
  \end{align*}
Thus $u\in L^2(Q_T)$. Note that the moment problem (\ref{moment_pb})
can equivalently be written 
\begin{equation} \label{moment_pb_bis}
\begin{array}{ll}
\int_0^T \widehat{u}(\tau,n) \diff \tau=-  \widehat{f}_{02}(0)\,, \\
\int_0^T w_{1,n}(\tau) \widehat{u}(\tau,n) \diff \tau=- n \widehat{f}_{02}(n) \eu^{-n^2 T}, \quad & \forall n \in \mathbb{Z}\setminus\{0\}, \\ 
\int_0^T w_{2,n}(\tau) \widehat{u}(\tau,n)  \diff \tau = \iu n \widehat{f}_{01}(n) - \widehat{f}_{02}(n), \quad & \forall n \in \mathbb{Z}\setminus\{0\}.
\end{array}
\end{equation}
Thus, by (\ref{def:alpha_beta_n}), $u$ solves (\ref{moment_pb}).
  \end{proof}

\subsection{Proof of Theorem \ref{th:main_3}}

The  Kalman rank condition \eqref{Kalman_Thm3} is a necessary condition for null-controllability of \eqref{Syst_bloc_parab_A} by the same arguments as in \Cref{SectionRequirementKalman}. Thus we only  explain how to complete the proof of Theorem~\ref{th:main} to prove that it is a sufficient condition for null-controllability of \eqref{Syst_bloc_parab_A}. We introduce the space
\begin{equation}
\label{defInvQuantPar}
\bold{F_1} \coloneqq H^{d_1+1}_\m(\T)^{d_1} \times H^{d_1+1}(\T)^{d_2},
\end{equation}
equipped with the scalar product of $H^{d_1+1}(\T)^{d}$
and the space
\begin{equation}
\label{defInvQuantParAdj}
\widetilde{\bold{F_1}} \coloneqq L^{2}_\m(\T)^{d_1} \times L^2(\T)^{d_2},
\end{equation}
equipped with the scalar product of $L^2(\T)^d$.

The null-controllability of the system \eqref{Syst_bloc_parab_A} in $\bold{F_1}$ with control of the form $(0,u_2)\in\{0\}^{d_1}\times L^2(q_T)^{d_2}$ is equivalent to the following observability inequality: for every $T>T^\ast$, there exists $C>0$ such that, for every $g_0 \in \widetilde{\bold{F_1}}$, the solution of the adjoint system~\eqref{SystAdj} satisfies
\begin{equation} 
\label{IO_hyp_NewAPar}
\| g(T,\cdot) \|_{H^{-(d_1+1)}(\T)^{d}}^2 \leq C \int_0^T \int_{\omega} |g_2(t,x)|^2 \diff x \diff t.
\end{equation}
where $g_2(t,x)\in\C^{d_2}$ is made of the last $d_2$ components of $g(t,x)$.
\begin{proof}[Proof of the equivalence between the null-controllability in $\bold{F_1}$ and the observability inequality \eqref{IO_hyp_NewAPar}] 
We apply the duality \Cref{LemmaDuality} with
\[ \Phi_2 \colon f_0 \in \bold{F_1} \mapsto \eu^{-T \mathcal{L}} f_0  \in \widetilde{\bold{F_1}},\]
\[ \Phi_3 \colon u_2 \in L^2(q_T)^{d_2} \mapsto S(T;0,(0,u_2)) \in \widetilde{\bold{F_1}}.\]
Note that the mean value of the $d_1$ first components is indeed zero.
The null-controllability result in $\bold{F_1}$ is equivalent to the inclusion $\Ima(\Phi_2)\subset \Ima(\Phi_3)$, thus to the existence of a constant $C>0$ such that for every $g_T\in \widetilde{\bold{F_1}}$
\begin{equation}\label{eq:obs_lemma2_Par}
\lVert\Phi_2^*(g_T)\rVert_{H^{d_1+1}(\T)^{d}}\leq C\lVert \Phi_3^*(g_T)\rVert_{L^2(q_T)^{d_2}}.
\end{equation}
We compute the adjoint operators of $\Phi_2$ and $\Phi_3$ thanks to the duality relation between the solution $f$ of \eqref{Syst} and the solution $\varphi(\cdot)= g(T-\cdot)$ of the adjoint system~\eqref{SystAdj}:
\begin{equation} \label{dualite_3}
\langle f(T) , \varphi(T) \rangle_{L^2(\T)^d} = \langle f(0), \varphi(0) \rangle_{L^2(\T)^d} + \int_{0}^{T}\int_{\omega}  \langle u_2(t,x) , \varphi_2(t,x) \rangle \diff t \diff x.
\end{equation}
First, $\Phi_3^*(g_T)$ is the restriction of the $d_2$-last components of $\eu^{(t-T)\mathcal L^*}g_T$  to $[0,T]\times \omega$. Then, by (\ref{dualite_3}) and \Cref{th:lemma_dual_Hs} (working as in the proof of \Cref{CorHUM}), the left-hand side of \eqref{eq:obs_lemma2_Par} is
\[
\lVert\Phi_2^*(g_T)\rVert_{H^{d_1+1}(\T)^{d}} = \|\eu^{-T\mathcal{L}^*} g_T\|_{H^{-(d_1+1)}(\T)^{d}}.
\]

Thus the inequality \eqref{eq:obs_lemma2_Par} is indeed the observability inequality \eqref{IO_hyp_NewAPar}.
\end{proof}
By using the strategy developed in \Cref{Sec:Hyp}, we claim that, in the case $d_2=1$, it is sufficient to prove the following result in order to prove the observability inequality~\eqref{IO_hyp_NewAPar}.
\begin{prop}\label{PropObs1APar}
For every $T>T^\ast$, there exists $C>0$ such that for every $g_0 \in \widetilde{\bold{F_1}}$, the solution $g$ of the adjoint system \eqref{SystAdj} satisfies
\begin{equation} \label{IO_H_Obs1APar}
\| g(T,\cdot) \|_{H^{-(d_1+1)}(\T)^{d}}^2 \leq C \left( \|\partial_x^{d_1} g_1\|_{H^{-(d_1+1)}(q_T)}^2 + \|g_2\|_{L^2(q_T)}^2 \right) .
\end{equation}
\end{prop}
The observability inequality \eqref{IO_H_Obs1APar} has to be compared to the observability inequality \eqref{IO_H_Obs1A} in \Cref{Sec:Hyp}. Roughly speaking, the term $\|\partial_x^{d_1} g_1\|_{H^{-(d_1+1)}(q_T)}$ comes from the fact that we will perform $(d_1-1)$ steps of elimination, each of them “costs” one derivative (instead of two in \Cref{SectionControlHyperbolicd_1=1}) because we will use transport equations which are of order one in time and space (instead of parabolic equations which are of order two in space variable). The last step of elimination “costs” two derivatives because we will use a heat equation which is of order one in time and two in space. This explains the number $(d_1-1)+2 = d_1+1$ derivatives. By adapting the arguments of \Cref{SectionKd1>1}, we can also treat the case $d_2>1$.\\
\indent In order to prove \Cref{PropObs1APar}, by duality (a simple adaptation of the proof that \Cref{PropObs1A} and \Cref{PropNullControldx} are equivalent), it is sufficient to establish the following null-controllability result. 
\begin{prop}\label{PropNullControldxPar}
For every $f_0 \in \bold{F_1}$, there exists $u=(u_{\h},u_{\p}) \in (H_0^{2d_1+1}(q_T))^{d_1} \times L^2(q_T)^{d_2}$ such that $S(T, f_0, (\partial_x^{d_1} u_{\h},u_{\p})) = 0$.
\end{prop}

The proof of this result is an adaptation of the proof of \Cref{th:main}:
\begin{itemize}
\item we prove that parabolic high frequencies are null-controllable,
\item we prove that hyperbolic high frequencies are null-controllable,
\item we combine these two propositions to prove that high frequencies are null-controllable,
\item we finally deal with low frequencies.
\end{itemize}

For the first point, we just need a special case of the corresponding result that was used in the proof of \Cref{th:main}, i.e.\ \Cref{LemPar}.
\begin{prop}
\label{LemPar_BIS1}
If $n_0$ is large enough, there exists a continuous operator 
\begin{equation*}
\mathcal U^{\p,\sharp} \colon \begin{array}[t]{@{}c@{}l}
  \bold{F_1} \times H^{2d_1+1}_0((0,T')\times\omega)^{d_1}  & {}\to{}  C^\infty_c((T',T)\times\omega)^{d_2} \\
 (f_0,u_{\h}) &{}\mapsto u_{\p},
\end{array}
\end{equation*}
that associates with any $(f_0,u_{\h})\in \bold{F_1} \times H^{2d_1+1}_0((0,T')\times\omega)^{d_1} $ a control $u_{\p}=\mathcal U^{\p,\sharp}(f_0,u_{\h})$ such that
\begin{equation*}
\Pi^{\p} S(T;f_0, (\partial_x^{d_1} u_{\h},u_{\p})) = 0.
\end{equation*}
\end{prop}
\begin{proof}
Proposition \ref{LemPar_BIS1} is a consequence of Proposition \ref{LemPar} because $\bold{F_1} \times H^{2d_1+1}_0((T',T)\times\omega)^{d_2}$ is continuously embedded in $L^2(\T)^d \times L^2((T',T)\times \omega)^{d_2}$ and $\partial_x^{d_1} u_{\h} \in L^2((0,T')\times\omega)^{d_1}$ for every $u_{\h} \in H^{2d_1+1}_0((0,T')\times\omega)^{d_1}$. 
\end{proof}

For the second point, we will prove the following adaptation of \Cref{LemHyp}.
\begin{prop}
\label{LemHyp_BIS1}
If $n_0$ is large enough, there exists a continuous operator 
\begin{equation*}
\mathcal U^{\h,\sharp} \colon\begin{array}[t]{@{}c@{}l}
 \bold{F_1} \times H_0^{2d_1+1}((T',T)\times\omega)^{d_2} & \rightarrow  H^{2d_1+1}_0((0,T')\times\omega)^{d_1}\\
 (f_0,u_{\p}) &\mapsto u_{\h},
\end{array}
\end{equation*}
that associates with any $(f_0,u_{\p})\in  \bold{F_1} \times H^{2d_1+1}_0((T',T)\times\omega)^{d_2} $ a control $u_{\h}=\mathcal U^{\h,\sharp}(f_0,u_{\p})$ such that
\begin{equation}\label{eq:LemHyp_BIS1}
\Pi^{\h} S(T;f_0, (\partial_x^{d_1} u_{\h},u_{\p})) = 0.
\end{equation}
\end{prop}
While the ideas of the proof are the same as for \Cref{LemHyp}, the proof of this Proposition is technically more delicate, as we have to build regular controls, and, on the observability side, deal with the (slightly impractical) $H^s_0$ and $H^{-s}$ norms. We postpone the proof to the next subsection. For now, let us assume \Cref{LemHyp_BIS1} holds true, and finish the proof of \Cref{th:main_3}.

We now combine Propositions~\ref{LemPar_BIS1} and~\ref{LemHyp_BIS1} with the Fredholm alternative, as in the proof of \Cref{PropControlHighFreq}, to prove that high frequencies are null-controllable. That is to say, we get the following adaptation of \Cref{PropControlHighFreq}. 
\begin{prop}
\label{PropControlHighFreqd_xPar}
There exists a closed subspace $\mathcal{G}^{\sharp}$ of $\bold{F}_1$ with finite codimension and a continuous operator 
\begin{equation*}
\mathcal U^{\sharp} \colon\begin{array}[t]{@{}c@{}l}
\mathcal{G}^{\sharp} &{}\to  H_0^{2d_1+1}(q_T)^{d_1}  \times H_0^{2d_1+1}(q_T)^{d_2}\\
f_0 &{}\mapsto (u_{\h},u_{\p}),
\end{array}
\end{equation*}
that associates with each $f_0 \in \mathcal{G}^{\sharp}$ a pair of controls $\mathcal U^{\sharp} f_{0}=( u_{\h}, u_{\p})$ such that 
\begin{equation}
\label{PropKd_xPar}
\forall f_0 \in \mathcal{G}^{\sharp},\ \Pi S(T;f_0, ( \partial_{x}^{d_1}u_{\h}, u_{\p})) = 0.
\end{equation}
\end{prop}

The last step consists in showing that the null-controllability result of \Cref{PropNullControldxPar} is true at the low frequency-level, as we have already shown for \Cref{th:main} in \Cref{ControlLowFreqSec}. All the steps of \Cref{ControlLowFreqSec} remain unchanged except the Step 6. Indeed, the unique continuation argument transforms into: if $f(t,\cdot) = \eu^{tM} f_0$ with $(\partial_x^{d_1} f_1, f_2) = (0,0)$ in $(0,\varepsilon)\times\omega$ then $(\partial_x^{d_1}  f_{01},f_{02})= (0,0)$ thanks to the spectral inequality of Lebeau-Robbiano \eqref{eq:spectral}, that is to say, $f_0 = 0$ because $\int_{\T} f_{01}(x) \diff x =0$. 

This concludes the proof of \Cref{PropNullControldxPar} thus the proof of \Cref{PropObs1APar}. \qed

\subsection{Proof of \Cref{LemHyp_BIS1}}
The proof of \Cref{LemHyp_BIS1} is an adaptation of the one of \Cref{LemHyp}, with the following changes:
\begin{itemize}
\item we deal with the fact that we want a control of the form $(\partial_x^{d_1}u_\h,0)$,
\item we adapt the duality argument to take into account the regularity of the controls that we want (it involves some $H^{-s}$ norms),
\item we adapt all the inequalities to replace the relevant $L^2$ norms by $H^{-s}$ norms,
\item to build regular controls of the simple transport equation $\partial_t f+ \mu\partial_x f=0$, we use~\cite{alabau-boussouira_2017}. 
\end{itemize}

\paragraph{Step 1: reduction to an exact controllability problem.} We claim that in order to prove \Cref{LemHyp_BIS1}, we only have to prove the following exact controllability result.

\begin{prop}
\label{LemHyp_BISPar_bis}
If $n_0$ is large enough, then for every ${T'}>T^{*}$, there exists a continuous operator 
\begin{equation*}
{\underline{\mathcal{U}}}_{T'}^{\h,\sharp} \colon \begin{array}[t]{@{}c@{}l}
  F^\h \cap H^{2d_1+1}(\T)^{d}  & {}\to H_0^{2d_1+1}(q_{T'})^{d_1} \\
  f_{T'} &{}\mapsto  u_{\h},
\end{array}
\end{equation*}
that associates with any $f_{T'} \in F^\h \cap H^{2d_1+1}(\T)^{d}$, a control $\underline{\mathcal{U}}_{T'}^{\h,\sharp} (f_{T'}) = u_{\h}$ such that 
\begin{equation*}
\Pi^{\h} S({T'}; 0,  (u_{\h}, 0)) = f_{T'}.
\end{equation*}
\end{prop}

Indeed, by the choice of support in time of the controls, and by the reversibility of $\eu^{-t\mathcal L^\h}$ (see \Cref{SecReduceHyp} for the details), the relation~\eqref{eq:LemHyp_BIS1} is equivalent to
\begin{equation*}
\Pi^\h(S(T';0,(\partial_x^{d_1} u_\h,0))) = -\eu^{(T-T')\mathcal L^\h}\Pi^\h S(T;f_0,(0,u_\p)).
\end{equation*}
Note that functions in $F^\h$ have zero mean (see the definition of $F^\h$ Eq.~\eqref{defWh}). Thus, $\partial_x^{d_1}$ is invertible on $F^\h$, and its inverse $\partial_x^{-d_1}$ is, on the Fourier side, the multiplication by $(\iu n)^{-d_1}$. Moreover, the operator $\partial_x$ commute with $\Pi^\h$ and the semi-group $\eu^{-t\mathcal L}$. So the relation~\eqref{eq:LemHyp_BIS1} is equivalent to
\begin{equation}\label{eq:LemHyp_ter}
\Pi^\h(S(T';0,(u_\h,0))) = - \partial_x^{-d_1}\eu^{(T-T')\mathcal L^\h}\Pi^\h S(T;f_0,(0,u_\p))\eqqcolon K(f_0,u_\p).
\end{equation}
So, if \Cref{LemHyp_BISPar_bis} holds, we may choose (assuming it makes sense)
\[
u_\h \coloneqq \underline{\mathcal{U}}_{T'}^{\h,\sharp}(K(f_0,u_\p)).
\]
Thus, to end this first step, we just have to check that the right-hand side $K(f_0,u_\p)$ of~\eqref{eq:LemHyp_ter} is indeed in $F^\h\cap H^{2d_1+1}(\T)^d$.

The projection $\Pi^\h$ has range $F^\h$, and $\eu^{t\mathcal L^\h}$ sends $F^\h$ to itself, as do $\partial_x^{-d_1}$. So $K(f_0,u_\p)$ belongs to $F^\h$.

The group $\eu^{t\mathcal L^\h}$ sends every $H^s(\T)^d$ into itself (see \Cref{rk:semigroup_Hs}). Since $\Pi^\h$ is just the multiplication on the Fourier side by $P^\h(\iu/n)$, the operator $\Pi^\h$ also sends every $H^s(\T)^d$ into itself. Thus, we just have to check that $S(T,f_0,(0,u_\p)) = \eu^{-T\mathcal L}f_0+ S(T,0,(0,u_\p))\in H^{d_1+1}(\T)^d$ because $\partial_{x}^{-d_1}$ sends $H^{d_1+1}(\T)^d$ into $H^{2d_1+1}(\T)^d$.
\begin{itemize}
\item The function $f_0$ belongs to $H^{d_1+1}(\T)$ by hypothesis, so $\eu^{-T\mathcal L}f_0$ also belongs to $H^{d_1+1}(\T)$ (see \Cref{rk:semigroup_Hs}).
\item  The parabolic control $u_{\p}$ belongs to $H^{2d_1+1}_0((T',T)\times\omega)^{d_2}$ by hypothesis, thus for almost every $t \in (0,T)$, $(0,u_{\p})(t,\cdot)$ belongs to $H^{2d_1+1}(\T)$ and thus
\[S(T;0, (0,u_{\p})) = \int_0^T \eu^{-(T-t)\mathcal{L}} (0,u_{\p})(t) \diff t \in H^{2d_1+1}(\T)^d.\]
\end{itemize}

This concludes this first step.

\paragraph{Step 2: Observability inequality associated to the controllability problem of \Cref{LemHyp_BISPar_bis}.}
Let 
\[\Phi_2 \coloneqq \Pi^\h\circ \iota_{2d_1+1} \colon H^{2d_1+1}(\T)^d \rightarrow L^2(\T)^d\] 
be the restriction of $\Pi^\h$ to $H^{2d_1+1}(\T)^d$ and
\[\Phi_3\coloneqq \Pi^\h \circ \mathcal F_{T'}\circ (\iota_{2d_1+1},0)\colon H_0^{2d_1+1}(q_{T'})^{d_1} \to L^2(\T)^d,\]
where $(\iota_{2d_1+1},0)$ stands for the map $u_\h\in H^{2d_1+1}_0(q_{T'})^{d_1}\mapsto (u_\h,0) \in L^2(q_{T'})^d$. Note that $\Phi_2$ and $\Phi_3$ are continuous.

The controllability problem of \Cref{LemHyp_BISPar_bis} is equivalent to the inclusion $\Ima(\Phi_2)\subset \Ima(\Phi_3)$. Therefore, according to the duality \Cref{LemmaDuality}, it is equivalent to the following inequality: there exists $C>0$ such that for every $g_0\in L^2(\T)^d$, $\|\Phi_2^* g_0\|_{H^{2d_1+1}(\T)^d} \leq C\|\Phi_3^* g_0\|_{H^{2d_1+1}_0(q_{T'})^d}$.
Since ${\Pi^\h}^*$ is a projection on $\widetilde{F^\h}$, since $\mathcal F_{T'}^*g_0$ is the restriction of the first $d_1$ components of $\eu^{-(T'-t)\mathcal L^*}g_0$ to $q_{T'}$, and since $\iota_s^*$ is an isometry between $H^s_0$ and $H^{-s}$,\footnote{See \Cref{th:lemma_dual_Hs}, and also recall that because $\T$ has no boundary $H^s_0(\T) = H^s(\T)$.} this inequality reads: there exists $C>0$ such that for every $g_0\in \widetilde F^\h$, the solution $g = \eu^{-t\mathcal L^*}g_0$ of the adjoint system~\eqref{SystAdj} satisfies
\begin{equation}
\label{obs_lem_hyp_par}
\|g_0\|_{H^{-2d_1-1}(\T)^d} \leq C \|g_1\|_{H^{-2d_1-1}(q_{T'})^{d_1}},
\end{equation}
where $g_1$ are the first $d_1$ components of $g$.

Let $g_0 \in \widetilde F^\h$. For the remaining of this proof, we use the notations of \Cref{subsec:Hyp_NC}, and in particular we introduce the decompositions~(\ref{DecompVarphiThyp}) and~(\ref{DecompVarphiTHypJ}). In the following arguments, the constants $C$ do not depend on $g_0$. 

\paragraph{Step 3: We prove the observability inequality (\ref{obs_lem_hyp_par}) assuming that, for every $\mu \in \Sp(A')$, there exists $C>0$ such that the solution $G_\mu^\flat$ of \eqref{eq_gmudiese} satisfies}
\begin{equation} \label{Obs_H(-2d1-1)_transport}
\|G_{\mu}(0,\cdot)\|_{H^{-(2d_1+1)}(\T)} = \|G_{\mu}^\flat(0,\cdot)\|_{H^{-(2d_1+1)}(\T)} \leq C \|G_\mu^\flat \|_{H^{-(2d_1+1)}(q_T)}.
\end{equation}
We will prove \eqref{Obs_H(-2d1-1)_transport} in Step 3.\\
\indent We proceed as in the proof given in \Cref{subsec:Hyp_NC}.
By the explicit expression (\ref{Smu}) of $S_\mu$ and Bessel-Parseval identity, there exists $C=C(T')$ independent of $g_0$ such that
\begin{equation}
\label{defsourcecompBis}
 \norme{S_\mu}_{L^\infty((0,T'),H^{-(2d_1+1)}(\T)^d)} \leq C \norme{ g(0,\cdot)}_{H^{-(2d_1+2)}(\T)^d}.
\end{equation}
Using the Duhamel formula, we obtain that the function $\widetilde{G}_\mu$ defined by (\ref{def:gmutilde}) satisfies
\begin{align} 
& \| \widetilde{G}_\mu - G_\mu^\flat \|_{L^\infty((0,T'),H^{-(2d_1+1)}(\T)^d)}\notag\\
& \leq C \| \eu^{t R_{\mu}^{\h}(0)^*} S_\mu \|_{L^1((0,T'),H^{-(2d_1+1)}(\T)^d)} \leq C
\| g_0 \|_{H^{-(2d_1+2)}(\T)^d}.
\label{dmutilde-gmudiese_bis}
\end{align}
By (\ref{Obs_H(-2d1-1)_transport}), the triangular inequality, (\ref{def:gmutilde}) and (\ref{dmutilde-gmudiese_bis}), we deduce that
\begin{equation} \label{IO_faible_Gmu}
\begin{aligned}
\| G_\mu(0,\cdot) \|_{H^{-(2d_1+1)}(\T)^d} 
& \leq C \left( \| \widetilde{G}_\mu \|_{H^{-(2d_1+1)}(q_{T'})^d} +  \| \widetilde{G}_\mu - G_\mu^\flat \|_{H^{-(2d_1+1)}(q_{T'})^d} \right) \\
& \leq C \left( \|G_\mu\|_{H^{-(2d_1+1)}(q_{T'})^d}+ \| g_0 \|_{H^{-(2d_1+2)}(\T)^d} 
\right).
\end{aligned}
\end{equation}
Using Bessel-Parseval identity and the decomposition (\ref{Gmu_Phmu_decom}), we obtain
\begin{equation} \label{Obs_g1_Gmu}
\|G_{\mu} -P^\h_\mu(0)^* g \|_{L^\infty((0,T'),H^{-(2d_1+1)}(\T)^d)} \leq C \| g_0\|_{H^{-(2d_1+2)}(\T)^d}.
\end{equation}
We deduce from (\ref{IO_faible_Gmu}), the triangular inequality and (\ref{Obs_g1_Gmu}) that
\[ 
\| G_\mu(0,\cdot) \|_{H^{-(2d_1+1)}(\T)^d} \leq C \left(
\| P^\h_\mu(0)^* g \|_{H^{-(2d_1+1)}(q_{T'})^d} + \| g_0\|_{H^{-(2d_1+2)}(\T)^d} \right).
\]
Taking into account that $P^\h_\mu(0)^*=P^\h_\mu(0)^* P^\h(0)^*$, we get\footnote{Remark that if $K$ is a matrix and $f\in (H^{-s})^d$, then $\|Kf\|_{H^{-s}}\leq |K|\|f\|_{H^{-s}}$. Indeed, noting $\langle\cdot,\cdot\rangle$ the duality between $H^s_0$ and $H^{-s}$, we have for every $g\in H^{s}_0$, $\langle Kf,g\rangle = \langle f, K^*g\rangle\leq \|f\|_{H^{-s}} \|K^* g_0\|_{H^s_0}\leq\|f\|_{H^{-s}}|K^*| \| g_0\|_{H^s_0} $, and taking the supremum over $\|g\|_{H^s_0} = 1$, we do have $\|Kf\|_{H^{-s}}\leq |K^*|\|f\|_{H^{-s}}$.} 
\[\| P^\h_\mu(0)^* g \|_{H^{-(2d_1+1)}(q_{T'})^d}
\leq |P^\h_\mu(0)^*| \| P^\h(0)^* g \|_{H^{-(2d_1+1)}(q_{T'})^d} \leq C \|g_1\|_{H^{-(2d_1+1)}(q_{T'})^{d_1}}.\]
Using (\ref{DecompVarphiTHypJ}), the triangular inequality and the previous two estimates, we obtain
\begin{align}
&\|g_0\|_{H^{-(2d_1+1)}(\T)^d}  \notag\\
& \leq \sum\limits_{\mu \in \Sp(A')} \|G_\mu(0,\cdot)\|_{H^{-(2d_1+1)}(\T)^d} 
\leq C \left( \|g_1\|_{H^{-(2d_1+1)}(q_{T'})^{d_1}} + \| g_0\|_{H^{-(2d_1+2)}(\T)^d} \right).
\label{IO_faible_2d1+1}
\end{align}
Proceeding as in the end of the proof given in \Cref{subsec:Hyp_NC}, the inequality (\ref{IO_faible_2d1+1}), together with  a compactness-uniqueness argument, end Step 2.

\paragraph{Step 4: We prove that the solution $G_\mu^\flat$ of \eqref{eq_gmudiese} satisfies (\ref{Obs_H(-2d1-1)_transport})}.
By duality, it is actually enough to prove the following exact-controllability result.
\begin{prop}
\label{PropExactControllabilityResultTransport}
Let $\omega = (a,b)$ and $T' > \frac{2 \pi -(b-a)}{|\mu|}$. For every $(f_0, f_{T'}) \in (H^{2d+1}(\T)^{d})^2$, there exists $u \in H_0^{2d_1+1}(q_T)^{d}$ such that the solution $f$ of 
\begin{equation}
\label{SystAdjSAsympBis}
\left\{
\begin{array}{l l}
 \partial_t f +  \mu \partial_x f = u 1_{\omega} &\text{in}\ Q_{T'},\\
f(0,\cdot)=f_{0}& \text{in}\ \T,
\end{array}
\right.
\end{equation}
satisfies $f(T',\cdot)=f_{T'}$.
\end{prop}

To prove \Cref{PropExactControllabilityResultTransport}, we will use the following lemma, which is an easy adaptation of \cite[Lemma 2.6]{alabau-boussouira_2017}.
\begin{lemma}\label{LemmaCutOffTransport}
Let $\omega = (a,b)$ and $T' > \frac{2 \pi -(b-a)}{|\mu|}$. Then, there exists $\delta >0$ small enough and a cut-off function $\eta \in C^{\infty}([0,T']\times [0, 2\pi])$ with
\begin{equation}
\label{supporteta}
\eta = 0\ \text{in}\ [0,T']\times[0,2\pi] \setminus ((\delta, T'-\delta)\times (a+\delta, b-\delta)),
\end{equation}
such that, for every $x \in [0, 2 \pi]$, 
\begin{equation}
\label{defQx}
Q_x \coloneqq \int_0^{T'} \eta(s, x + \mu s) ds \neq 0.
\end{equation}
\end{lemma}
\begin{remark}
We assumed that the function $\eta$ is extended by $2\pi$-periodicity in the spatial variable.
\end{remark}
Now, we give the proof of \Cref{PropExactControllabilityResultTransport} thanks to \Cref{LemmaCutOffTransport}.
\begin{proof}[Proof of \Cref{PropExactControllabilityResultTransport}]\renewcommand{\qedsymbol}{$\lozenge$}
We take the control
\begin{equation}
\label{Defcontrolu}
u(t,x) = \eta(t,x) Q_{x-\mu t}^{-1}( f_{T'}(x)-f_{0}(x-\mu t)).
\end{equation}
We easily check that the control $u$ belongs to $H_0^{k}(q_T)$ by using the support of $\eta$ \eqref{supporteta}, and the regularity of the three functions $\eta$, $f_{T'}$ and $f_0$. Let $f$ be the solution of \eqref{SystAdjSAsympBis} with initial data $f_0$ and control $u$ defined in \eqref{Defcontrolu}. We just have to check that $f$ satisfies $f(T',\cdot)=f_{T'}$. We write the solution along the characteristic, that is to say
\[ \frac{\diff}{\diff t} f(t, x + \mu t) = u(t,x+t) =  \eta(t,x+\mu t) Q_x^{-1}( f_{T'}(x+\mu t)-f_{0}(x)).\]
By integrating in space between $0$ and $T'$ and by using the defintion of $Q_x$ \eqref{defQx}, we obtain
\[ f(T', \cdot + \mu T') - f(0,\cdot) = f_{T'}(\cdot+ \mu T') - f_0(\cdot),\]
then $f(T',\cdot) = f_{T'}$ which concludes the proof of \Cref{PropExactControllabilityResultTransport}.
\end{proof}
This ends the proof of \Cref{LemHyp_BISPar_bis}.\qed

\begin{appendices}

\section{Pure transport solutions are not enough to disprove the observability inequality}\label{app:transport}

\begin{prop}\label{th:transport?}
Let us assume that the $d\times d^2$ matrix 
\begin{equation*}
\left(\begin{array}{c|c|c|c}
B  & AB & \cdots & A^{d-1} B
\end{array}\right)
\end{equation*}
has rank $=d$, or, equivalently, that there is no eigenvector of $A^*$ in the kernel of $B^*$ (see for instance~\cite[Lemma~1]{beauchard_2011}).
Let  $\mu\in \set R$ and $T>0$. There exists $C=C(\mu,T)>0$\footnote{With the help of \Cref{th:perturb}, we could even prove that $C(\mu,T)$ can be chosen indepentantly of $\mu$.} such that every solution of the adjoint system~\eqref{SystAdj} of the form $g(t,x) = g_0(x-\mu t)$ satisfies $\|g(T,\cdot)\|_{L^2(\set T)^d}\leq C \|g\|_{L^2([0,T]\times\omega)^d}$.
\end{prop}

This statement shows that, for a dense set of matrices $(A,B)$ pure transport solutions of the adjoint system~\eqref{SystAdj} cannot be used to disprove the observability inequality (\ref{ObsHumG}), and thus the null controllability of (\ref{Syst}).

\newcommand{\Sol}{Sol}
\begin{proof}
Let us note $\Sol_\mu$ the set of solutions of the adjoint system~\eqref{SystAdj} of the form $g_0(x-\mu t)$.
Remark that according to the expression~\eqref{eq:sol_adj} of the solutions of the adjoint system, the relation $g_0\in\Sol_\mu$ is equivalent to
\begin{equation}\label{eq:transport?fourier}
\forall n\neq 0,\; nE\left(\frac\iu n\right)^{\!\!*}\hat g_0(n) = \iu \mu \hat g_0(n).
\end{equation}

We claim that $\Sol_\mu$ is finite dimensional. Indeed, if it is infinite dimensional, then, according to the relation~\eqref{eq:transport?fourier}, there is infinitely many $n$ such that $\iu \mu$ is an eigenvalue of $nE(\iu/n)$. Let $(X_{n_k})_{k\geq 0}$ be an associated sequence of eigenvectors, chosen such that $|X_{n_k}| = 1$. Since the unit sphere of $\set C^d$ is compact, we may assume that $(X_{n_k})$ converges to some $X$ with $|X|= 1$. Then we have
\[
n_k B^* X_{n_k} -\iu A^* X_{n_k} + \frac1{n_k} K^*X_{n_k} = n_kE\left(\frac\iu{n_k} \right)^{\!\!*} X_{n_k} = \iu \mu X_{n_k} \xrightarrow[k\to +\infty]{} \iu \mu X.
\]
And since $-\iu A^* X_{n_k} +(n_k)^{-1} K^*X_{n_k} \xrightarrow[k\to+\infty]{} -\iu A^* X$, we must have $B^*X = 0$ and $A^*X = -\mu X$. But this is in contradiction with the hypothesis of the Proposition. Therefore $\Sol_\mu$ is finite dimensional.

So, according to the description~\eqref{eq:transport?fourier} of $\Sol_\mu$, there exists $N>0$ such that every solution of the adjoint system~\eqref{SystAdj} of the form $g_0(x-\mu t)$ has no frequencies higher than $N$: $\Sol_\mu \subset \Span\{e_n,\:|n|<N\}$. But finite linear combination of exponentials have the unique continuation property.\puncfootnote{For instance because they are entire functions, and entire functions have the unique continuation property.} So the expressions $\|g_0(\cdot - \mu T)\|_{L^2(\T)^d}$ and $\|g_0(x-\mu t)\|_{L^2([0,T]\times \omega)^d}$ both define a norm on $\Sol_\mu$. Since $\Sol_\mu$ is finite dimensional, these two norms are equivalent. This proves the claimed inequality.
\end{proof}

\subparagraph{Acknowledgments:} The second author would like to thank his Ph.\ D.\  advisor, Gilles Lebeau, for numerous discussions and in particular for pointing us~\cite{lebeau_1998}. He also thanks Émmanuel Trélat for helping us with some duality arguments.

 The second author was partially supported by the ERC advanced grant SCAPDE, seventh framework program, agreement no. 320845.

\bibliographystyle{plain}
\bibliography{bib}
\end{appendices}

\end{document}